\documentclass[reqno,12pt]{amsart}
\usepackage[margin=2cm]{geometry}
\usepackage{amsmath}
\usepackage{amssymb}
\usepackage{amsbsy}
\usepackage[utf8]{inputenc}
\usepackage{dsfont}
\usepackage{tcolorbox}
\usepackage{setspace}  
\usepackage{subcaption}
\usepackage{arydshln}
\usepackage{mathtools}
\usepackage[colorlinks=true,breaklinks=true,linkcolor=blue]{hyperref}
\usepackage[normalem]{ulem}
\usepackage{cancel}

\usepackage{enumerate}

\def\a{{\boldsymbol a}}
\def\n{{\boldsymbol n}}
\def\x{{\boldsymbol x}}

\def\u{{\boldsymbol u}}
\def\v{{\boldsymbol v}}
\def\L{{\boldsymbol L}}
\def\R {{\mathds R}}

\def\dx     {{\rm d}{\boldsymbol x}}

\definecolor{mygreen}{RGB}{36, 158, 81}
\definecolor{myyellow}{RGB}{222, 203, 0}

\newtheorem{theorem}{Theorem}[section]
\newtheorem{corollary}{Corollary}
\newtheorem{lemma}[theorem]{Lemma}
\newtheorem{proposition}[theorem]{Proposition}
\newtheorem{remark}[theorem]{Remark}

\begin{document}

\title[Numerical blow-up for the Keller--Segel--Navier--Stokes equations]{Exploring numerical blow-up phenomena for the  Keller--Segel--Navier--Stokes equations}

\author[J. Bonilla]{Jesús Bonilla$^\ddag$}
\address{$\ddag$ Los Alamos National Laboratory, Los Alamos, NM 87545, USA. E-mail: {\tt \href{mailto:jbonilla@lanl.gov}{jbonilla@lanl.gov}}}
\thanks{$\ddag$ Los Alamos National Laboratory, an affirmative action/equal opportunity employer, is operated by Triad National Security, LLC for the National Nuclear Security
Administration of U.S. Department of Energy under contract 89233218CNA000001. Los Alamos National Laboratory strongly supports academic freedom and a researcher's right to publish; as an institution, however, the Laboratory does not endorse the viewpoint of a publication or guarantee its technical correctness. LA-UR-23-20373}

\author[J. V. Gutiérrez-Santacreu]{Juan Vicente Gutiérrez-Santacreu$^\S$}
\address{$\S$Dpto. de Matemática Aplicada I\\
         E. T. S. I. Informática\\
         Universidad de Sevilla\\
         Avda. Reina Mercedes, s/n.\\
         E-41012 Sevilla\\
         Spain\\
         E-mail: {\tt \href{mailto:juanvi@us.es}{juanvi@us.es}}}

\thanks{$^\S$ JVGS was partially supported by the Spanish Grant No. PGC2018-098308-B-I00 from Ministerio de Ciencias e Innovación - Agencia Estatal de Investigación with the participation of FEDER and by the Andalusian Grant No. P20\_01120 from Junta de Andalucía (Consejería de Economía, Conocimiento, Empresas y Universidad)}

\date{\today}

\begin{abstract} The Keller-Segel-Navier-Stokes system governs chemotaxis in liquid environments. This system is to be solved for the organism and chemoattractant densities and for the fluid velocity and pressure. It is known that if the total initial cell density mass is below $2\pi$ there exist globally defined generalised solutions, but what is less understood is whether there are blow-up solutions beyond such a threshold and its optimality. 

Motivated by this issue, a numerical blow-up scenario is investigated. Approximate solutions computed via a stabilised finite element method founded on a shock capturing technique are such that they satisfy \emph{a priori} bounds as well as lower and $L^1(\Omega)$ bounds for the cell and chemoattractant densities. In particular, this latter properties are essential in detecting numerical blow-up configurations, since the non-satisfaction of these two requirements might trigger numerical oscillations leading to non-realistic finite-time collapses into persistent Dirac-type measures. 

Our findings show that the existence threshold value $2\pi$ encountered for the cell density mass may not be optimal and hence it is conjectured that the critical threshold value $4\pi$ may be inherited from the fluid-free Keller-Segel equations. Additionally it is observed that the formation of singular points can be neglected if the fluid flow is intensified.          
\end{abstract}
\maketitle


\tableofcontents
\section{Introduction}
\subsection{The model} 
Keller and  Segel \cite{Keller_Segel_1970, Keller_Segel_1971} brought in the early 70's the first model governing chemotaxis growth at macroscopic level. It consists of a system of partial differential equations of parabolic type, which is to be solved for the organism and chemoattractant densities as follows.  Let $\Omega\subset \mathds{R}^2$ be an open, bounded domain, with $\boldsymbol{n}$ being its outward-directed unit normal vector to $\Omega$, and let $[0,T]$ be a time interval. Take $Q=(0,T]\times \Omega$ and $\Sigma=(0,T]\times\partial\Omega$. Then find  $n: \bar Q\to (0,\infty)$, the organism density, and $c:\bar Q \to [0,\infty)$, the chemoattractant density,  satisfying 
\begin{equation}\label{KS}
\left\{
\begin{array}{rcll}
\partial_t n-\Delta n&=&-\nabla\cdot(n\nabla c)&\mbox{ in } Q,
\\
\partial_t c -\Delta c&=&n-c&\mbox{ in }Q,
\end{array}
\right.
\end{equation}
subject to the no-flux boundary conditions
\begin{equation}\label{BC_KS}
\nabla n\cdot \boldsymbol{n}=0\quad\mbox{ and }\quad \nabla c\cdot\boldsymbol{n}=0\quad\mbox{ on }\quad \Sigma,
\end{equation} 
and the initial conditions 
\begin{equation}\label{IC_KS}
n(0)=n_0\quad\mbox{ and }\quad c(0)=c_0\quad\mbox{ in }\quad \Omega.
\end{equation}

This model exhibits many interesting properties. For instance, solutions to this model can be found, whenever $\int_\Omega n_0(\x)\,\dx\in(0,4\pi)$, that remain uniformly bounded for all time \cite{Nagai_Senba_Yoshida_1997}. On the contrary, if $\int_\Omega n_0(\x)$ is larger than such a threshold value, the situation changes drastically; there exists solutions blowing up either in finite or infinite time \cite{Horstmann_Wang_2001, Senba_Suzuki_2001}\footnote{Here one needs $\Omega$ to be simply connected.}.  Furthermore, this threshold turns out to be critical in the sense that, for each $\varepsilon>0$, there exist $n_0$ with $\int_\Omega n_0(\x)\, \dx>4\pi +\varepsilon$ that develop a finite-time collapse into persistent  Dirac-type measures.

When chemotaxis occurs in a fluid, the original Keller--Segel equations need to be coupled with a Navier--Stokes type of equations. In this context, the fluid flow will be influenced by the self-enhanced chemotactic motion inducing a velocity profile; and more importantly, the converse is further true, a change in the velocity profile will accordingly alter the chemotactic growth. This implies that, if the initial fluid velocity is zero, the evolution of the chemotactic growth will induce a velocity, and this velocity will in turn affect the behaviour of the 
chemotactic growth. 

The  Keller--Segel--Navier-Stokes equations for governing the chemotaxis of unicellular organisms under the influence of a fluid flow are written as 

\begin{equation}\label{KSNS}
\left\{
\begin{array}{rclcc}
\displaystyle
\partial_t n+\u\cdot\nabla n-\Delta n+ \nabla\cdot(n\nabla c)&=&0&\mbox{ in }& Q ,
\\
\displaystyle
\partial_t c+\u\cdot\nabla c-\Delta c+ c&=&n&\mbox{ in }& Q,
\\
\partial_t\u+(\u\cdot\nabla)\u-\Delta\u+\nabla p-n\nabla \Phi&=&\boldsymbol{0}&\mbox{ in }& Q,
\\
\nabla\cdot\u&=&0&\mbox{ in }& Q.
\end{array}
\right.
\end{equation}
Here $n:\overline{\Omega\times (0,T]}\to (0,\infty)$ represents the organism density, $c: \overline{\Omega\times(0,T)}\to [0,\infty) $ represents the chemoattractant density, $\u:\overline{\Omega\times(0,T]}\times\to\mathds{R}^2$ is the fluid velocity, and $p:\overline{\Omega\times(0,T]}\times\to\mathds{R}$ is the fluid pressure; furthermore, $\Phi$ is a gravitational force.

These equations are supplemented with homogeneous Neumann boundary conditions for the Keller-Segel subsystem and homogeneous Dirichlet boundary conditions for the Navier-Stokes equations, i.e.,
\begin{equation}\label{BC_KSNS}
\partial_\n n=0, \quad\partial_\n c =0 \mbox{ and } \quad \u=0 \mbox { on } \quad  \partial \Omega\times (0,T),
\end{equation}
and with initial conditions
\begin{equation}\label{IC_KSNS}
n(0)=n_0,\quad
c(0)=c_0\mbox{ and } \u(0)=\u_0\quad\mbox{ in }\quad \Omega.
\end{equation}

As far as we are concerned, the work of Winkler \cite{Winkler_2020} is the only mathematical analysis available in the literature for system \eqref{KSNS}, where a similar mass threshold phenomenon is put forward possibly deciding between the boundedness and unboundedness of solutions. More precisely,  he found that, if $\int_\Omega n_0(\x)\,\x\in(0,2\pi)$, problem \eqref{KSNS}-\eqref{IC_KSNS} possesses generalised solutions globally in time. Therefore, the theory in studying mathematical properties of solutions to problem \eqref{KSNS}-\eqref{IC_KSNS} seems to be in a very primitive stage in comparison with those for problem \eqref{KS}-\eqref{IC_KS}. Thus, at this point, two accordingly feasible scenarios can be conjectured. On the one hand, it might occur that the chemotaxis-fluid system inherits the mathematical properties of solutions in a consistent fashion from the fluid-free system thereof, since the influence of the fluid through the transport and gravitational effect is essentially dominated by the cross-diffusion one. On the other hand, the fluid mechanism might interact with the cross-diffusion principle causing that the mass critical takes a certain value between $2\pi$ and $4\pi$. This might be due to the creation of extremely large values of density gradients by means of amplifying the nonlinear cross-diffusion process.

Detecting blow-up configurations via the numerical solution to problem \eqref{KSNS}-\eqref{IC_KSNS} is extremely difficult because the growth of densities in time is an ubiquitous phenomenom in smooth solutions to problem \eqref{KSNS}-\eqref{IC_KSNS} (and \eqref{KS}-\eqref{IC_KS}), which must be treated with caution. Therefore, in an evolving smooth solution with a point of actively growing gradients, we must decide whether there is merely a huge growth of the density or actually a finite-time singularity development. Furthermore, it is pointed out in \cite{GS_RG_2021} that discretising directly without enforcing lower bounds for the chemoattractant and organism densities will lead to unstable numerical solutions.

Inspired on \cite{Badia_Bonilla_GS_2022}, we develop a numerical scheme founded on a finite element method stabilised via adding nonlinear diffusion combined with an Euler time-stepping integrator. The nonlinear diffusion draw on a graph-Laplacian operator together with a shock detector in order to minimise the amount of numerical diffusion introduced in the system. This approach has been turned out efficient in the fluid-free chemotaxis \cite{Badia_Bonilla_GS_2022}. 

The determination of potential candidate singular solutions from numerical simulation presents a variety of challenging issues inherited from the mathematical analysis that we need to cope with thoroughly. These numerical issues are: 
\begin{itemize}
\item lower bounds: positivity for the chemoattractant density and nonnegativity for the organism density;
\item time-independent integrability bounds: particularly mass conservation for the chemoattractant density; and
\item time-independent square integrability bounds and time-dependent square integrability for the gradient of the organism density and the fluid velocity.  
\end{itemize} 

Problem \eqref{KS}--\eqref{IC_KS} has already studied in the numerical literature with different techniques \cite{Saito_2012, GS_RG_2021, Strehl_Sokolov_Kuzmin_Turek_2010, Strehl_Sokolov_Kuzmin_Horstmann_Turek_2013, Li_Shu_Yang_2017, Chertock_Epshteyn_Hu_Kurganov_2018, Chertock_Kurganov_2008}, but these have not been applied to contexts of self-enhanced chemotactic motion in the presence of a fluid flow.

\subsection{Notation} Throughout, we adopt the standard notation for Sobolev spaces. Let $\mathcal{O}\subset \mathds{R}^2$ be an open, bounded domain.  For $p\in[1,\infty]$, we denote by $L^p(\Omega)$ the usual Lebesgue space, i.e.,
$$
L^p(\mathcal{O}) = \{v : \Omega \to \R\, :\, v \mbox{ Lebesgue-measurable}, \int_\Omega |v(\x)|^p {\rm d}\x<\infty \}.
$$
or
$$
L^\infty(\mathcal{O}) = \{v : \Omega \to \R\, :\, v \mbox{ Lebesgue-measurable}, {\rm ess}\sup_{\x\in \Omega} |v(\x)|<\infty \}.
$$
This space is a Banach space endowed with the norm
$\|v\|_{L^p(\Omega)}=(\int_{\Omega}|v(\x)|^p\,{\rm d}\x)^{1/p}$ if $p\in[1, \infty)$ or $\|v\|_{L^\infty(\Omega)}={\rm ess}\sup_{\x\in \Omega}|v(\x)|$ if $p=\infty$. In particular, when $p=2$, $L^2(\Omega)$ is a Hilbert space.  We shall use $(u,v)=\int_{\Omega}u(\x)v(\x){\rm d}\x$ for its inner product

Let $\alpha = (\alpha_1, \alpha_2)\in \mathds{N}^2$ be a multi-index with $|\alpha|=\alpha_1+\alpha_2$, and let
$\partial^\alpha$ be the differential operator such that
$$\partial^\alpha=
\Big(\frac{\partial}{\partial{x_1}}\Big)^{\alpha_1}\Big(\frac{\partial}{\partial{x_2}}\Big)^{\alpha_2}.$$
For $m\ge 0$ and $p\in[1, \infty)$, we shall consider $W^{m,p}(\Omega)$ to be the Sobolev space of all functions whose derivatives up to order $m$ are in $L^p(\Omega)$, i.e.,
$$
W^{m,p}(\Omega) = \{v \in L^p(\Omega)\,:\, \partial^k v \in L^2(\Omega)\ \forall ~ |k|\le m\},
$$ associated to the norm
$$
\|f\|_{W^{m,p}(\Omega)}=\left(\sum_{|\alpha|\le m} \|\partial^\alpha f\|^p_{L^p(\Omega)}\right)^{1/p} \quad \hbox{for} \ 1 \leq p < \infty,
$$
and
$$
\|f\|_{W^{m,p}(\Omega)}=\max_{|\alpha|\le m} \|\partial^\alpha f\|_{L^\infty(\Omega)} \quad  \hbox{for} \  p = \infty.
$$
For $p=2$, we denote $W^{m,2}(\Omega)=H^m(\Omega)$.

Let $\mathcal{D}(\Omega)$ be the space of infinitely times
differentiable functions with compact support on $\Omega$. 
The closure of $\mathcal {D}(\Omega)$ in
$H^{m}(\Omega)$ is denoted by $H^{m}_0(\Omega)$. 
We will also make use of the following space of vector fields:
$$
\boldsymbol{\mathcal{V}}=\{\v\in \boldsymbol{\mathcal{D}}(\Omega): 
\nabla\cdot\v=0 \mbox{ in } \Omega \}. 
$$
The closure of $\boldsymbol{\mathcal{V}}$ in the $\L^2(\Omega)$ norm is denoted by  $\boldsymbol{H}$ and  is characterised \cite{Temam2001} (for $\Omega$ being Lipschitz-continuous) by 
$$
\boldsymbol{H}= \{ \u \in \L^2(\Omega) : \nabla\cdot\u =0 \mbox{ in } 
\Omega, \u\cdot\boldsymbol{n} = 0 \hbox{ on }
\partial\Omega \},
$$
where $\n$ is the outward normal to $\Omega$ on $\partial \Omega$.    
Finally, we consider 
$$
L^2_0(\Omega)= \{ p \in L^2(\Omega) : \ \int_\Omega
p(\x)\, d\x =0 \}.
$$

Distinction is made between scalar- or vector-valued functions, so spaces of vector-valued functions and their elements are identified with bold font. 

For any sequence $\{\eta_h^m \}_{m=0}^M$, we use the notation $\delta_t \eta^{m+1}_h=\frac{\eta_h^{m+1}-\eta^{m}_h}{k}$. We set  $B(\x_0; r )=\{ \x\in\mathds{R}^2:  \|\x-\x_0\|_{E}<r \}$ and $B_+(\x_0; r)=\{ \x\in\mathds{R}^2:  \|\x-\x_0\|_{E}<r \mbox{ and } x_2\ge 0\}$ and furthermore  $\bar\eta =\frac{1}{|\Omega|}\int_\Omega \eta(\x)\, \dx$.

\subsection{Outline} After this introductory section, Section 2 introduces the numerical scheme developed as well as all required notation. Section 3 contains technical preliminaries needed for the subsequent Section 4, where the main theoretical properties of the scheme are discussed. Section~5 is devoted to numerical experiments. Finally, we draw some conclusions in Section 6.
\section{Finite element approximation}
\subsection{Finite element spaces}Let $\Omega\subset\mathds{R}^2$ be an open, bounded domain whose boundary $\partial\Omega$ is polygonal whose vertices and edges are denoted by $\mathcal{V}_h$ and $\mathcal{E}_h$, respectively. Then it is considered a quasi-uniform family $\{\mathcal{T}_h\}_h$ of conforming triangulation of $\bar\Omega$, i.e., the intersection of any two triangles $T, T'\in \mathcal{T}_h$ is empty, an edge or a vertex, such that  $\bar\Omega=\cup_{T\in\mathcal{T}_h } T$, where $h_T={\rm diam }\, T$ and $h=\max_{T\in\mathcal{T}_h} h_T$. Moreover, let  $\mathcal{N}_h=\{\boldsymbol{a}_i\}_{i=1}^{I}$ be the coordinates of all the vertices of $\mathcal{T}_h$. 

Associated to each $\mathcal{T}_h$, it is constructed $X_h$ consisting of continuous piecewise polynomials, i. e.,
$$
X_h^k=\{x\in C^0(\bar\Omega): x_h|_T\in\mathds{P}_k\, \forall T\in\mathcal{T}_h\}, 
$$
where $\mathds{P}_k$ is the set of all polynomials on $T$ of degree less than or equal to $k$. When $k=1$, we simply write $X_h$ and denote by $\{\varphi_{\a_i}\}_{i=1}^{I}$  its corresponding nodal basis, with $\Omega_{\a_i}={\rm supp }\,\varphi_{\a_i}$ being the macro-element associated to each $\varphi_{\a_i}$.  For each $\a_i\in\mathcal{N}_h$, we choose $\a_{ij}^{\rm sym}$ to be the point at the intersection between the line that passes through $\a_i$ and $\a_j$ and $\partial\Omega_{\a_i}$ not being $\a_j$. Thus the set of all the symmetric nodes for $\a_i$ is denoted by $\mathcal{N}_h^{\rm sym}(\Omega_{\a_i})$. If $x_h\in X_h$, we write $x(\a_i)=x_i$ for all $\a_i\in\mathcal{N}_h$.

\begin{figure}
    \centering
    \includegraphics[width=0.3\textwidth]{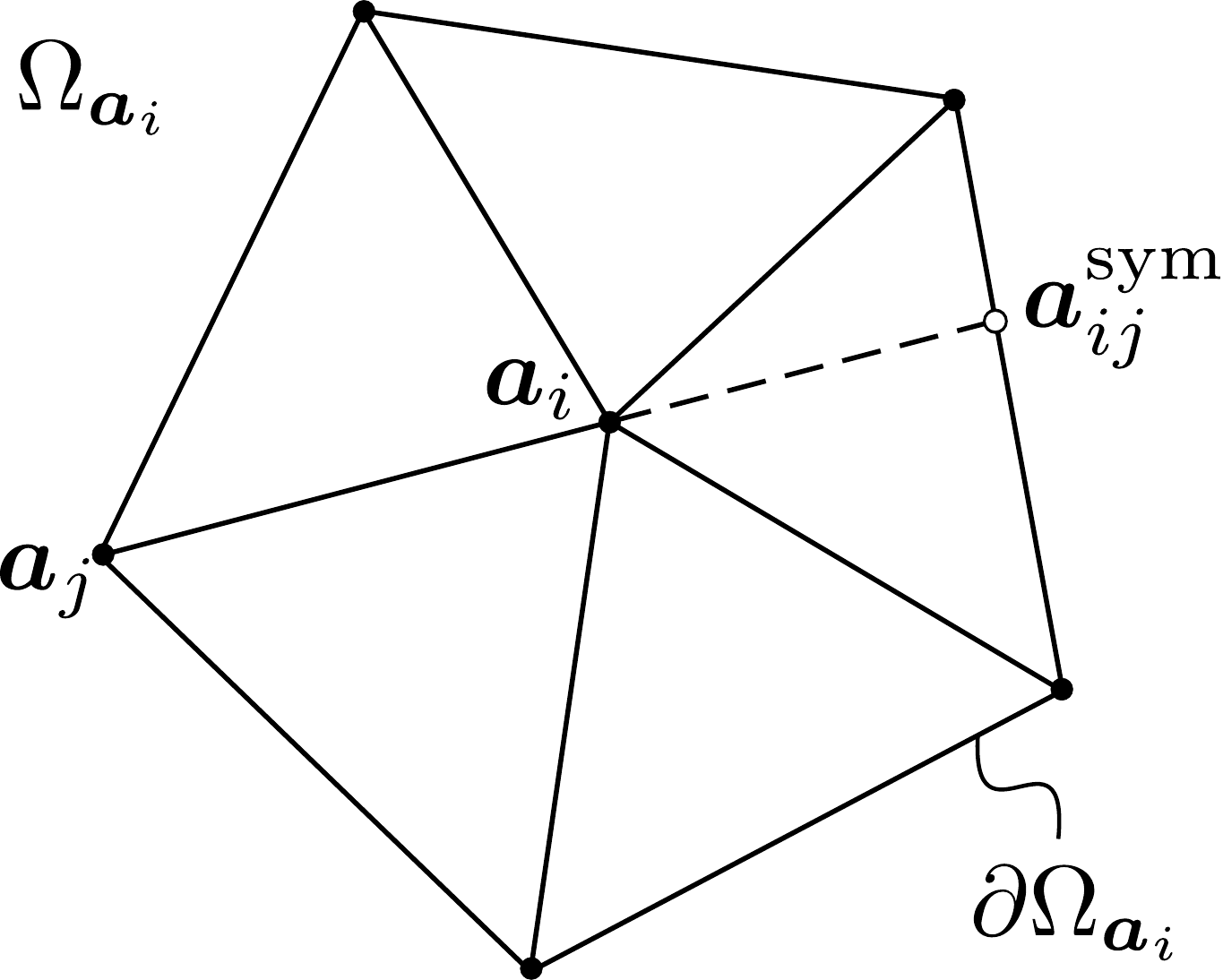}
    \caption{Construction of the symmetric node $\a_{ij}^{\rm sym}$}
    \label{fig:Sym_node}
\end{figure}

Our choice for the cell and chemoattractant spaces is  $N_h=X_h$ and $C_h=X_h$, respectively. Instead, for the velocity and pressure, we take the Taylor-Hood space pair, i.e.,  $\boldsymbol{U}_h= \boldsymbol{X}^2_h\cap \boldsymbol{H}^1_0(\Omega)$ and $Q_h=X_h\cap L^2_0(\Omega)$. Observe that we have selected velocity/pressure finite element spaces that do satisfy the LBB inf-sup condition.

We introduce $i_h: C(\bar\Omega)\to X_h$, the nodal interpolation operator, such that $i_h x(\a_i)=x(\a_i)$ for $i\in I$. A discrete inner product on $C(\bar\Omega)$ is then defined
$$
(x_h, \bar x_h )_h=\int_\Omega i_h(x_h(\x)\bar x_h(\x))\,\dx.
$$ 
We also introduce the following averaged interpolation operator $\mathcal{I}_h: L^p(\Omega)\to X_h$ defined as follows.  Take, for each node $\a\in\mathcal{N}_h$, a triangle $T_\a$ such that $\a\in T$. Thus one regards:  
$$
\mathcal{I}_h \eta =\sum_{i\in I} \left(\frac{1}{|T_{\a_i}|}\int_{T_{\a_i}} \eta(\x)\,\dx\right)\varphi_{\a_i}.
$$ 
We know \cite{Girault_Lions_2001, Scott_Zhang_1990} that there exists $C_{\rm sta}>0$, independent of $h$, such that, for all $\eta\in L^p(\Omega)$,   
\begin{equation}
\|\mathcal{I}_h \eta\|_{L^p(\Omega)}\le C_{\rm sta} \| \eta \|_{L^p(\Omega)}\quad \mbox{for }  1\le p\le\infty,
\end{equation}
and 
\begin{equation}
\mathcal{I}_h \psi \ge \mbox{or} > 0\quad\mbox{ if }\quad \psi\ge \mbox{or} >0. 
\end{equation}
In addition to the above interpolation operators, we consider the Ritz-Darcy projection operator $\mathcal{RD}_h : \boldsymbol{H}\to \boldsymbol{U}_h $ defined as: given $\u\in \boldsymbol{H}$, find $\mathcal{RD}_h \u$ such that, for all $(\bar\u_h, \bar p_h)\in \boldsymbol{U}_h\times P_h$, 
\begin{equation}
\left\{
\begin{array}{rcl}
(\mathcal{RD}_h \u, \bar \u_h)+(\nabla p_h, \bar\u_h )&=&(\u, \bar\u_h ),
\\
(\nabla\cdot\mathcal{RD}_h \u, \bar p_h)&=&0
\end{array}
\right.
\end{equation}
It is readily to prove that 
$\|\mathcal{RD}_h\u\|_{\L^2(\Omega)}\le \|\u\|_{\L^2(\Omega)}$ holds.
\subsection{Heuristics} It will be next proceeded without regards to rigour to develop the finite element formulation for approximating system \eqref{KSNS}--\eqref{IC_KSNS}. We take as our starting point the standard finite element formulation of system \eqref{KSNS}--\eqref{IC_KSNS} combined with a time-stepping integration, which is implicit with respect to the linear terms and semi-implicit with respect to the nonlinear terms, except for the chemotaxis term being implicit. This method it is then read as follows:

Let $\Phi\in W^{1,\infty}(\Omega)$ and assume that $(n_0, c_0,\u_0)\in L^1(\Omega)\times L^2(\Omega)\times \boldsymbol{H}$ with $n_0>0$ and $c_0\ge0$. Then consider  $(n_{0h}, c_{0h}, \u_{0h})\in N_h\times C_h\times \boldsymbol{H}^1_0(\Omega)$, where $n_{0h}=\mathcal{I}_h n_0$, $c_{0h}=\mathcal{I}_h c_0$, and $\u_{0h}=\mathcal{RD}_h\u_0$. 

Let $\{t_n \}_{n=0}^M$ be a sequence of points  partitioning $[0,T]$ into subintervals of the same length $k=\frac{T}{M}$ with $M\in\mathds{N}$ and select  $n_h^0=n_{0h}$, $c^0_h=c_{0h}$, and $\u^0_h=\u_{0h}$. Given $(n^m_h, c^m_h, \u^m_h)\in N_h\times C_h\times \boldsymbol{U}_h$, find  \linebreak  $(n^{m+1}_h, c^{m+1}_h, \u^{m+1}_h, p^{m+1}_h)\in N_h\times C_h\times \boldsymbol{U}_h\times Q_h$ such that, for all $(\bar n_h, \bar c_h,\bar\u_h, \bar p_h)\in N_h\times C_h\times U_h\times Q_h$, 
\begin{equation}\label{eq_aux:n_h}
(\delta_t n^{m+1}_h, \bar n_h)-(\u^m_h n^{m+1}_h, \nabla\bar n_h)+(\nabla n^{m+1}_h,\nabla\bar n_h)-(n^{m+1}_h\nabla c^{m+1}_h,\nabla\bar n_h)=0,
\end{equation}
\begin{equation}\label{eq_aux:c_h}
\begin{array}{rcl}
\displaystyle
(\delta_t c^{m+1}_h, \bar c_h)+(\u^m_h\cdot\nabla c^{m+1}_h, \bar c_h)+\frac{1}{2} (\nabla\cdot\u^{m+1}_h c^{m+1}_h, \bar c_h)&&
\\[5pt]
+(\nabla c^{m+1}_h,\nabla\bar c_h)+ (c^{m+1}_h,\bar c_h)
-(n^{m+1}_h, \bar c_h)&=&0,
\end{array}
\end{equation}
\begin{equation}\label{eq_aux:u_h}
\begin{array}{rcl}
\displaystyle
(\delta_t \u^{m+1}_h, \bar\u_h)+
(\u^m_h\cdot\nabla\u^{m+1}_h,\bar\u_h)+\frac{1}{2}(\nabla\cdot\u^m_h\, \u^{m+1}_h, \bar \u_h)&&
\\
+(\nabla \u^{m+1}_h,\nabla\bar \u_h)+(\nabla p^{m+1}_h, \bar\u_h)-(n^{m+1}_h\nabla\Phi, \bar\u_h)&=&0,
\end{array}
\end{equation}
and
\begin{equation}\label{eq_aux:p_h}
(\nabla\cdot\u^{m+1}_h,\bar p_h)=0. 
\end{equation}

The obtainment of the underlying properties that lead to deriving a priori estimates for the discrete solutions to  \eqref{eq_aux:n_h}-\eqref{eq_aux:p_h} is by no means a direct computation.  For equation  \eqref{eq_aux:n_h}, we need to control the quantity $\sum_{m=0}^{M-1} k \|\nabla\log(n^{m+1}_h+1)\|_{\L^2(\Omega)}$,  which is deduced by testing it against the nonlinear test function $\frac{1}{n^{m+1}_h}$. Nevertheless, it does not seem evident how to do it directly from  \eqref{eq_aux:n_h}. Positivity and mass conservation are as well required from \eqref{eq_aux:n_h}. For equation  \eqref{eq_aux:c_h},  a uniform-in-time $L^2(\Omega)$-bound is sought, but it is apparently connected to  $\sum_{m=0}^{M-1} k \|\nabla\log(n^{m+1}_h+1)\|_{L^2(\Omega)}$ for \eqref{eq_aux:n_h}. Furthermore, the stabilising term $\frac{1}{2} (\nabla\cdot\u^{m+1}_h c^{m+1}_h, \bar c_h)$ has added to rule out the convective term from \eqref{eq_aux:c_h}  when tested with $c^{m+1}_h$. Nonnegativy and an $L^1(\Omega)$-bound  are additionally needed from  \eqref{eq_aux:c_h}.
For equation \eqref{eq_aux:u_h}, we wish a uniform-in-time $\L^2(\Omega)$-bound. In doing so,  the stabilising term  $\frac{1}{2}(\nabla\cdot\u^{n}_h\, \u^{m+1}_h, \bar \u_h)$ has been incorporated to deal with the convective term. 

It is interesting to note that $L^1(\Omega)$ bounds -- in particular, mass conservation -- are straightforwardly derived from \eqref{eq_aux:n_h} and \eqref{eq_aux:c_h}; therefore it is desirable to  keep them. 


In what follows we set forth some modifications for scheme \eqref{eq_aux:n_h}--\eqref{eq_aux:p_h} in the forthcoming subsections in order for the above-mentioned properties to hold.    

\subsubsection{Chemotaxis term} The first modification is with regard to the chemotaxis term. It is to be recalled that we seek  an estimate for $\sum_{m=0}^{M-1} k \|\nabla\log(n^{m+1}_h+1)\|_{\L^2(\Omega)}$ from \eqref{eq_aux:n_h}, which stems from testing by $\frac{1}{n^{m+1}_h}$. We proceed in the spirit of \cite{Badia_Bonilla_GS_2022} by writing 
$$
\begin{array}{rcl}
( n_h\nabla  c_h,\nabla \bar n_h)&=&\displaystyle
\sum_{k,j,i\in I} n_k c_j  \bar n_i (\varphi_{\a_k} \nabla \varphi_{\a_j}, \nabla\varphi_{\a_i}) 
\\
&=&\displaystyle
\sum_{\tiny\begin{array}{c}k\in I \\i<j\in I\end{array}} n_k (c_j- c_i)  (\bar n_i-\bar n_j) (\varphi_{\a_k} \nabla \varphi_{\a_j}, \nabla\varphi_{\a_i}).
\end{array}
$$
Observe that $\boldsymbol{a}_k, \boldsymbol{a}_j,\boldsymbol{a}_i$ must belong to ${\rm supp }\, \varphi_{\a_k}\cap{\rm supp }\, \varphi_{\a_j}\cap{\rm supp }\, \varphi_{\a_i}$ so that $(\varphi_{\a_k} \nabla\varphi_{\a_j}, \nabla\varphi_{\a_i})\not=0$. That is,  $\boldsymbol{a}_k, \boldsymbol{a}_j,\boldsymbol{a}_i$ belong to the same triangle $T$. Thus we change the value $n_k$ to that of the geometric mean $\sqrt{n_i}\sqrt{n_j}$ corresponding to its neighbouring  nodes. Accordingly we have  
\begin{equation}\label{New_KS_term: not-truncated}
(n_h\nabla c_h,\nabla \bar n_h)\approx\displaystyle
\sum_{\tiny\begin{array}{c}i<j\in I\end{array}} \sqrt{n_i}\sqrt{n_j}  (c_j - c_i)  (\bar n_i-\bar n_j) (\nabla \varphi_{\a_j}, \nabla\varphi_{\a_i}).
\end{equation}
Here it should be pointed out that  the well-posedness of \eqref{New_KS_term: not-truncated} holds under the formal condition $n_h\ge0$, which is not known yet. So it is indispensable to consider 
\begin{equation}\label{def: gamma_ij_chem}
\gamma_{ij}^{ch}=\sqrt{[n_i]_+}\sqrt{[n_j]_+}
\end{equation} 
so that we are able to define the approximation of $(n_h\nabla c_h,\nabla \bar n_h)$ as 
\begin{equation}\label{New_KS_term}
(n_h\nabla c_h,\nabla \bar n_h)_*=\sum_{i<j\in I}\gamma_{ji}^{\rm ch} (c_j- c_i)  (\bar n_i-\bar n_j) (\nabla \varphi_{\a_j}, \nabla\varphi_{\a_i}),
\end{equation}
where  $[x]_+=\max\{0,x\}$ stands for the positivity part. Fortunately, the use of the truncating operator  will be superfluous once positivity for $n_h$ is proved.  
\subsubsection{Convective term}

For the convective term in \eqref{eq_aux:n_h}, we proceed as follows. Let $\u^m_h\in \boldsymbol{V}_h$ be given. Then, for $n_h, \bar n_h\in N_h$,  write 
$$
\begin{array}{rcl}
(n_h\u^m_h,\nabla \bar n_h)&=&((n_h+1)\u^m_h ,\nabla \bar n_h)
\\
&=&\displaystyle
\sum_{i,j\in I} (n_j+1)  (\bar n_i-\bar n_j) (\varphi_{\a_j} \u^m_h, \nabla\varphi_{\a_i}).
\end{array}
$$
To obtain the first line, the incompressibility condition \eqref{eq_aux:p_h} was used. We now replace $n_j+1$ by $\gamma^{\rm c}_{ji}$, where
\begin{equation}\label{def: gamma_ij_conv not extended}
\gamma^{\rm c}_{ji}=\left\{
\begin{array}{ccl}
\displaystyle
\frac{\log (n_j+1)-\log (n_i+1)}{\frac{1}{n_i+1}-\frac{1}{n_j+1}}&\mbox{ if }& n_j\not=n_i,
\\
n_j+1 &\mbox{ if }&n_j=n_i.
\end{array}
\right.
\end{equation}
Thus we arrive at  
\begin{equation}\label{New_Discrete_Conv: not-extended}
(n_h\u^m_h ,\nabla \bar n_h)\approx
\sum_{i,j\in I} \gamma^c_{ji}  (\bar n_i-\bar n_j) (\varphi_{\a_j}\u^m_h, \nabla\varphi_{\a_i}).
\end{equation}
As before the well-posedness of \eqref{New_Discrete_Conv: not-extended} through \eqref{def: gamma_ij_conv not extended}  is only entailed for $n_h+1>0$; nevertheless, we cannot ensure such a  condition. Consequently, we must allow for negative values of $n_h$. In doing so, the following extension of the logarithmic function to negative values is used. For $0<\varepsilon<1$, define
\begin{equation}\label{def:g_eps}
g_\varepsilon(s)=\left\{
\begin{array}{ccl}
\log s&\mbox{ if }& s>\varepsilon,
\\
\frac{s}{\varepsilon}+\log \varepsilon-1&\mbox{ if }& s\le \varepsilon. 
\end{array}
\right.
\end{equation}
Therefore it remains  
\begin{equation}\label{New_conv_term: extended}
(n_h \u^m_h,\nabla \bar n_h)_*=\sum_{i<j\in I}\gamma_{ji}^{\rm c}(\bar n_i-\bar n_j)(\varphi_{\a_j}\u^m_h, \nabla\varphi_{\a_i}),
\end{equation}
with
\begin{equation}\label{def: gamma_ij_conv extended}
\gamma^c_{ji}=\left\{
\begin{array}{ccl}
\displaystyle
\frac{g_\varepsilon(n_j+1)-g_\varepsilon(n_i+1)}{\frac{1}{n_i+1}-\frac{1}{n_j+1}}&\mbox{ if }& n_j\not=n_i,
\\
\left[n_i\right]_++1 &\mbox{ if }&n_j=n_i.
\end{array}
\right.
\end{equation}

\subsubsection{Stabilising terms}  As lower bounds are most likely to fail for \eqref{eq_aux:n_h} and \eqref{eq_aux:c_h} (see \cite{GS_RG_2021}), one needs to use some additional technique so as to derive them. We are particularly interested in developing two stabilising operators $B_n$ and $B_c$, for \eqref{eq_aux:n_h} and \eqref{eq_aux:c_h}, respectively, based on artificial diffusion, which depends on a shock detector and is defined through the graph-Laplacian operator ensuring that the minimal amount of numerical diffusion is introduced \cite{Badia_Bonilla_2017,Badia_Bonilla_GS_2022}. 

For $B_n$, we define
\begin{equation}\label{Bn}
(B_n(n_h, \u_h) \tilde n_h,  \bar n_h)=\sum_{i<j\in I}\nu_{ji}^n(n_h,\u_h) (\tilde n_j- \tilde n_i) (\bar n _j-\bar n_i),
\end{equation}
where
\begin{equation}\label{def:nu_n}
\nu_{ji}^n(n_h, \u_h)=\max\{\alpha_{\a_i}(n_h) f^n_{ij}, \alpha_{\a_j}(n_h)f^n_{ji}, 0\}\quad\mbox{ for }\quad i\not= j,
\end{equation}
with
$$
f^n_{ij}=\left\{
\begin{array}{rcl}
\displaystyle
-(n_j+1) (n_i+1)\frac{g_\varepsilon(n_j+1)-g_\varepsilon(n_i+1)}{(n_j-n_i)^2} (\varphi_{\a_j}\u^m_h, \nabla\varphi_{\a_i})+(\nabla \varphi_{\a_j}, \nabla \varphi_{\a_i})&\mbox{ if }& n_j\not=n_i,
\\
0&\mbox{ if }&n_j=n_i,
\end{array}
\right.
$$
and
\begin{equation}\label{def:nu^n_ii}
\nu_{ii}^n(w_h, v_h)=\sum_{j\in I(\Omega_{\a_i})\backslash\{i\}}\nu_{ij}^n(w_h, v_h).
\end{equation}
For $B_c$, we consider   
\begin{equation}\label{Bc}
(B_c(c_h,\u_h)\tilde c_h, \bar c_h)=\sum_{i\in I}\sum_{j\in I(\Omega_{\a_i})}\nu^c_{ij}(c_h, \u_h) \tilde c_j \bar c_i \ell(i,j),
\end{equation}
with 
\begin{equation}\label{def:nu^c_ij}
\nu_{ij}^c(c_h, \u_h)=\max\{\alpha_i(c_h) f^c_{ij}, \alpha_j(c_h)f^c_{ji}, 0\}\quad\mbox{ for }\quad i\not= j,
\end{equation}
and
\begin{equation}\label{def:nu^c_ii}
\nu_{ii}^c(c_h, \u_h)=\sum_{j\in I(\Omega_{\a_i})\backslash\{i\}}\nu_{ij}^c(c_h,\u_h),
\end{equation}
where $f^c_{ij}$ is given by 
$$
f^c_{ij}=k^{-1}(\varphi_{\a_j}, \varphi_{\a_i})+(\u^m_h\cdot\nabla\varphi_{\a_j}, \varphi_{\a_i})+\frac{1}{2}(\nabla\cdot \u^m_h \varphi_{\a_j}, \varphi_{\a_i})+(\nabla\varphi_{\a_j},\nabla\varphi_{\a_i}) + (\varphi_{\a_j}, \varphi_{\a_i}).
$$
Additionally, $\ell(i,j)= 2\delta_{ij}-1$ is the graph-Laplacian operator, where $\delta_{ij}$ is the Kronecker delta.

Let  $q \in \mathds{R}^+$ and $\eta_h\in X_h$. For each $\boldsymbol{a}_i \in \mathcal{N}_h$, the shock detector $\alpha_{\a_i}(\eta_h)$  reads:
\begin{equation}\label{def:alpha_min_max}
\alpha_{\a_i}(\eta_h) = \left\{
\begin{array}{cc}  
\left[
\frac{\left[{\sum_{j\in I(\Omega_{\a_i})} [\![\nabla \eta_h]\!]_{ij}}\right]_+}{\sum_{j\in I(\Omega_{\a_i})} 2\{\!\!\{|\nabla \eta_h \cdot \hat{\boldsymbol{r}}_{ij}\}\!\!\}|_{ij}}
\right]^q & \text{if } 
\sum_{j\in I(\Omega_{\a_i})} {{|\nabla \eta_h\cdot \hat{\boldsymbol{r}}_{ij}|}}_{ij} \neq 0, 
\\
0 & \text{otherwise},
\end{array}
\right. 
\end{equation}
where 
$$
[\![\nabla \eta_h]\!]_{ij} =\frac{\eta_j - \eta_i}{|\boldsymbol{r}_{ij}|} + \frac{\eta_j^{\rm sym} - \eta_i}{|\boldsymbol{r}_{ij}^{\rm sym}|}, 
$$
and
$$
\{\!\!\{|\nabla \eta_h\cdot \hat{\boldsymbol{r}}|_{ij}\}\!\!\}_{ij} =\frac{1}{2}
\left(\frac{|\eta_j-\eta_i|}{|\boldsymbol{r}_{ij}|}+\frac{|\eta_j^{\rm sym}-\eta_i|}{|\boldsymbol{r}_{ij}^{\rm sym}|}\right),
$$
with $\boldsymbol{r}_{ij} = \a_j - \a_i$, $\hat{\boldsymbol{r}}_{ij} = \frac{\boldsymbol{r}_{ij}}{|\boldsymbol{r}_{ij}|}$, and $\boldsymbol{r}_{ij}^{\rm sym} = \a_{ij}^{\rm sym} - \a_i$ for $\a^{\rm sym}_{ij}\in\mathcal{N}^{\rm sym}_h(\Omega_{\a_i})$.


A significant implication of the definition of $\alpha_{\a_i}$ is stated in the following lemma. See \cite[Lemma 3.1]{Badia_Bonilla_2017} for a proof.
\begin{lemma}\label{lm: alpha_i} Let $\eta_h\in X_h$. Then, for each $\a_i\in\mathcal{N}_h$, it follows that $0\le \alpha_{\a_i}(\eta_h)\leq 1$ and that $\alpha_{\a_i}(\eta_h)=1$ for any minimum value at $\a_i$.
\end{lemma}

\subsection{Finite element scheme} Here we announce our numerical algorithm that turns out from including \eqref{New_KS_term} and \eqref{New_conv_term: extended} instead of their original discretisation and from incorporating the stabilising terms \eqref{Bn} and \eqref{Bc} into equations \eqref{eq_aux:n_h} and \eqref{eq_aux:c_h}. We thus arrive at the following algorithm: 

Let $n_h^0=n_{0h}$, $c^0_h=c_{0h}$, and $\u^0_h=\u_{0h}$. Known $(n^m_h,c^m_h, \u^m_h)\in N_h\times C_h\times\boldsymbol{U}_h$, find \linebreak  $(n^{m+1}_h,  c^{m+1}_h, \u^{m+1}_h, p^{m+1}_h)\in N_h\times C_h\times \boldsymbol{U}_h\times Q_h$ such that, for all $(\bar n_h,\bar c_h, \bar{\boldsymbol{u}}_h, \bar p_h)\in N_h\times C_h\times \boldsymbol{U}_h\times Q_h$,  
\begin{equation}\label{eq:n_h}
\begin{array}{rcl}
(\delta_t n^{m+1}_h, \bar n_h)_h&-&(n^{m+1}_h\u^m_h, \nabla \bar n_h)_*+(\nabla n^{m+1}_h, \nabla \bar n_h)
\\
&-&(n^{m+1}_h\nabla c^{m+1}_h, \nabla \bar n_h)_*+(B_n(n^{m+1}_h, \u^m_h) n^{m+1}_h, \bar n_h)=0,
\end{array}
\end{equation}
\begin{equation}\label{eq:c_h}
\begin{array}{rcl}
(\delta_t c^{m+1}_h, \bar c_h)&+&\displaystyle
(\u^m_h\cdot\nabla c^{m+1}_h, \bar c_h)+\frac{1}{2}(\nabla\cdot \u^m_h c^{m+1}_h, \bar c_h)
\\
&+&(\nabla c^{m+1}_h,\nabla \bar c_h)+(c^{m+1}_h, \bar c_h)
\\
&+&(B_c(c^{m+1}_h, \u^m_h) c^{m+1}_h, \bar c_h)=(n^{m+1}_h, \bar c_h)_h,
\end{array}
\end{equation}
\begin{equation}\label{eq:u_h}
\begin{array}{rcl}
\displaystyle
(\delta_t \u^{m+1}_h, \bar\u_h)+
(\u^m_h\cdot\nabla\u^{m+1}_h,\bar\u_h)+\frac{1}{2}(\nabla\cdot\u^{m}_h\, \u^{m+1}_h, \bar \u_h)&&
\\
+(\nabla \u^{m+1}_h,\nabla\bar \u_h)+(\nabla p^{m+1}_h, \u_h)-(i_h(n^{m+1}_h\nabla\Phi), \bar\u_h)_h&=&0,
\end{array}
\end{equation}
and
\begin{equation}\label{eq:p_h}
(\nabla\cdot\u^{m+1}_h,\bar p_h)=0. 
\end{equation}

For the sake of numerical analysis,  we have applied the lumping procedure to some terms in \eqref{eq:n_h}-\eqref{eq:u_h}, which we recall that is denoted as $(\cdot,\cdot)_h$.

\section{Technical preliminaries}
Following is a collection of fundamental functional inequalities that will be employed throughout this work.

First of all, let us recall Moser--Trudinger's inequality \cite{Moser_1971, Trudinger_1967}.
\begin{theorem} Let $\Omega$ be bounded domain in $\mathds{R}^2$. Further assume $\eta\in H^1_0(\Omega)$. Then there exists $C>0$ such that 
\begin{equation}\label{ineq:Moser-Trudinger}
\int_{\Omega} e^{|\eta(\x)|}\dx\le C|\Omega| e^{\frac{1}{16\pi} \|\nabla \eta\|^2_{\L^2(\Omega)}}.
\end{equation}
\end{theorem}

We next set forth the analogue of a variant of \eqref{ineq:Moser-Trudinger} given in \cite[Th. 2.2]{Nagai_Senba_Yoshida_1997} but for polygonal domains. 
\begin{theorem}\label{th:Moser-Trundinger}Assume that $\Omega$ is a bounded polygonal domain in $\R^2$ with $\alpha_\Omega$ being  the minimum interior angle at the vertices of $\mathcal{V}_h$.   Further suppose that $\eta\in H^1(\Omega)$ with $\eta>0$. Then there exists $\beta_\Omega\in [1, 2)$, depending upon $\Omega$, so that, for each $\lambda>0$, one can find $C_\lambda>0$ such that  
\begin{equation}\label{ineq:Trudinger-Moser_New}
\int_\Omega  e^{\eta(\x)}\dx\le C e^{\displaystyle\beta^2_\Omega\frac{1+\lambda}{8\alpha_\Omega} \|\nabla \eta\|^2_{\L^2(\Omega)}+C_\lambda\|\eta\|^2_{L^1(\Omega)}}.
\end{equation}
\end{theorem}
\begin{proof} Let $\{U_j\}_{j=1}^J$ be a covering of $\partial\Omega$, i. e., $\partial\Omega\subset\bigcup_{j=1}^J U_j$, such that $U_j=B(\boldsymbol{p}_j; \rho_j)$, where $\{\boldsymbol{p}_j\}_{j=1}^J\subset\partial\Omega$ satisfying $\mathcal{V}_h\subset\{\boldsymbol{p}_j\}_{j=1}^J$ and $\rho_j>0$. In addition, if $\boldsymbol{p}_j\in\mathcal{V}_h$, then there exist $E, \tilde E\in\mathcal{E}_h $ such that $\boldsymbol{p}_j=E\cap\tilde E$ with $\rho_j=\min \{{\rm meas}(E), {\rm meas}(\tilde E)\}$; otherwise, if $\boldsymbol{p}_j\in {\rm int }\, E$ for some $E\in\mathcal{E}_h$, then $U_j\cap \partial\Omega\subset {\rm int}\, E$. 

It is well-known \cite{Adams_Fournier_2003} that there exist a family of functions $\{\theta_i\}_{j=0}^J\subset C^{\infty}(\R^2)$ such that $0\le\theta_j\le 1$ for all $j=0,\cdots, J$, with ${\rm supp}\, \theta_j\subset U_j$ if $j\not=0$ and ${\rm supp }\, \theta_0\subset \Omega$, and $\sum_{j=0}^J \theta_j=1$. As a result, one may write
\begin{equation}\label{partition}
\begin{array}{rcl}
\displaystyle
\int_\Omega  e^{\eta(\x)}\dx&=&\displaystyle
\int_\Omega  e^{\sum_{j=0}^J (\theta_j \eta)(\x)}\dx\le\sum_{j=0}^J \int_{U_j\cap\Omega}  e^{(\theta_j \eta)(\x)}\dx,
\end{array}
\end{equation}
where we denoted $U_0={\rm supp }\,\theta_0$. From inequality \eqref{ineq:Moser-Trudinger}, we bound
$$
\int_{U_0\cap \Omega} e^{(\theta_0 \eta)(\x)}\dx\le C |U_0\cap\Omega| e^{\frac{1}{16\pi}\|\nabla(\theta_0 \eta)\|^2_{\L^2(U_0\cap\Omega)}}.
$$
Now  Gagliardo--Nirenberg's interpolation and Young's inequality yield
$$
\begin{array}{rcl}
\|\nabla(\theta_0 \eta)\|_{\L^2(U_0\cap\Omega)}&\le&\|\nabla \eta\|_{\L^2(\Omega)}+ \|\nabla\theta_0\|_{\L^\infty(\R^2)}  \|\eta\|_{L^2(\Omega)}
\\
&\le&\|\nabla \eta\|_{\L^2(\Omega)}+ C_1 \|\nabla\theta_0\|_{\L^\infty(\R^2)} \|\nabla\eta\|^{\frac{1}{2}}_{\L^2(\Omega)}  \|\eta\|_{L^1(\Omega)}^\frac{1}{2}+C_2 \|\nabla\theta_0\|_{\L^\infty(\R^2)} \|\eta\|_{L^1(\Omega)}
\\
&\le&(1+\tilde\lambda)\|\nabla\eta\|_{\L^2(\Omega)}+ C_{\tilde\lambda} \|\eta\|_{L^1(\Omega)}
\end{array}
$$ 
and therefore
$$
\begin{array}{rcl}
\displaystyle
\int_{U_0\cap \Omega} e^{(\theta_0\eta)(\x)}\dx&\le& C |U_0\cap\Omega| e^{\frac{1}{16\pi}((1+  \tilde\lambda) \|\nabla \eta\|_{\L^2(\Omega)}+ C_{\tilde\lambda} \|\eta\|_{L^1(\Omega)})^2}
\\
&\le&\displaystyle
C |U_0\cap\Omega|e^{\frac{1}{16\pi}((1+3\tilde\lambda+\tilde\lambda^2)\|\nabla \eta\|^2_{\L^2(\Omega)}+ C_{\tilde\lambda} \|\eta\|^2_{L^1(\Omega)})}
\\
&\le&\displaystyle
C|U_0\cap\Omega|e^{\frac{1+\lambda}{16\pi}\|\nabla\eta\|^2_{\L^2(\Omega)}+ C_{\lambda} \|\eta\|^2_{L^1(\Omega)}},
\end{array}
$$
where we chose $\lambda=3\tilde\lambda+\tilde\lambda^2$. Let $\boldsymbol{p}_j\in\mathcal{V}_h$ and denote $\alpha_j$ to  be the interior angle at $\boldsymbol{p}_j$. Without loss of generality, one can assume that $\boldsymbol{p}_j=\boldsymbol{0}$ and $U_j\cap\Omega=S(\rho_j, \alpha_j)$, where
$$
S(\rho_j, \alpha_j)=\{\x=r e^{i \beta}: 0\le r\le \rho_j \mbox{ and } \frac{\pi}{2}-\frac{\alpha_j}{2}\le \beta\le \frac{\pi}{2}+\frac{\alpha_j}{2}\}.
$$
Consider $T_P:\R^2\to\R^2$ to be the polar coordinate mapping, i.e., $T_P( r,\beta)=(r\cos\beta, r\sin\beta)$ and  define the invertible mapping $H_j: [0,\tilde\rho_j]\times[0,\pi]\to[0,\rho_j]\times[ \frac{\pi}{2}-\frac{\alpha_j}{2}, \frac{\pi}{2}+\frac{\alpha_j}{2}] $, with $\tilde\rho_j>0$, as 
$$H_j(\tilde r, \tilde \beta)= (\tilde r\frac{\rho_j}{\tilde\rho_j}, \frac{\pi}{2}+\frac{\alpha_j}{\pi} (\tilde\beta-\frac{\pi}{2})),
$$
whose Jacobian determinant is $ {\rm det} J_H=\frac{\rho_j}{\tilde\rho_j}\frac{\alpha_j}{\pi}$.  Thus we write 
$$
\begin{array}{rcl}
\displaystyle
\int_{U_j\cap\Omega} e^{(\theta_j\eta)(\x)}\,\dx&=&\displaystyle
\int_{[0,\rho_j]\times[ \frac{\pi}{2}-\frac{\alpha_j}{2}, \frac{\pi}{2}+\frac{\alpha_j}{2}] }  r e^{(\theta_j \eta\circ T_P)(r,\beta)}\,dr d\beta
\\
&=&\displaystyle
\frac{\rho_j^2}{\tilde\rho_j^2}\frac{\alpha_j}{\pi}\int_{[0,\tilde\rho_j]\times[0,\pi]} \tilde re^{(\theta_j \eta\circ T_P\circ H_j)(\tilde r,\tilde \beta)}\,d\tilde rd\tilde\beta
\\
&=&\displaystyle
\frac{\rho_j^2}{\tilde\rho_j^2}\frac{\alpha_j}{\pi}\int_{B_+(\boldsymbol{0},\tilde\rho_j)} e^{(\theta_j \eta\circ T_P\circ H_j\circ T_P^{-1})(\x)}\,\dx.
\end{array}
$$
We know that $\theta_j \eta\circ T_P\circ H_j\circ T_P^{-1}\in H^{1}(B_+(\boldsymbol{0};\rho_j))$, since $T_P\circ H_j\circ T_P^{-1}\in C_{\rm Lip}(B_+(\boldsymbol{0}; \tilde\rho_j); S(\rho_j,\alpha_j) )$. Indeed, observe that 
\begin{align*}
T_P\circ &H_j\circ T_P^{-1}(x_1, x_2)
\\
&=(\frac{\rho_j}{\tilde\rho_j}\sqrt{x_1^2+x_2^2} \cos(\frac{\pi}{2}+\frac{\alpha_j}{\pi}(\arg(x_1,x_2)-\frac{\pi}{2})),\frac{\rho_j}{\tilde\rho_j}\sqrt{x_1^2+x_2^2} \sin(\frac{\pi}{2}+\frac{\alpha_j}{\pi}(\arg(x_1,x_2)-\frac{\pi}{2}))),
\end{align*}
where ${\rm arg}(\cdot,\cdot)\in [-\frac{\pi}{2}, \frac{3\pi}{2})$. It is clear that $T_P\circ H_j\circ T_P^{-1}\in C(B_+(\boldsymbol{0};\tilde\rho_j);S(\rho_j, \alpha_j))$ and  $\frac{\partial(T_P\circ H_j\circ T_P^{-1})_i}{\partial x_k}\in C(B_+(\boldsymbol{0};\tilde\rho_j); S(\rho_j, \alpha_j)\backslash\{\boldsymbol{0}\})$ for $i,k=1,2$. Moreover, there exists $\beta_j>0$ such that  $\|J_{T_P\circ H_j\circ T_P^{-1}}\|_{L^\infty(\Omega)}\linebreak\le\frac{\rho_j}{\tilde\rho_j} \beta_j$, where $J_{T_P\circ H_j\circ T_P^{-1}}$ is the Jacobian matrix. For instance, we have that
$$
\begin{array}{rcl}
\displaystyle
\frac{\partial (T_P\circ H_j\circ T_P^{-1})_1}{\partial x_1}(x_1,x_2)&=&\displaystyle
\frac{\rho_j}{\tilde\rho_j \pi} \frac{1}{\sqrt{x_1^2+x_2^2}} \Big( \pi x_1 \sin (\frac{\alpha_j}{\pi} \arg(x_1,x_2)-\frac{1}{2})
\\
&&\displaystyle
-\alpha_j x_2\cos (\frac{\alpha_j}{\pi} \arg(x_1,x_2)-\frac{1}{2}) \Big);
\end{array}
$$ 
thereby using polar coordinates $\lim_{r\to 0}\Big(\frac{\partial (T_P\circ H_j\circ T_P^{-1})_1}{\partial x_1}\circ T^{-1}_P\Big)(r,\beta)$ does not exist, but 
\begin{align*}
|\lim_{r\to 0}\Big(\frac{\partial (T_P\circ H_j\circ T_P^{-1})_1}{\partial x_1}&\circ T^{-1}_P\Big)(r,\beta)|
\\
&=|\frac{\rho_i}{\tilde\rho_j \pi} \Big( \pi \cos(\beta) \sin (\frac{\alpha_j}{\pi} (\beta-\frac{\pi}{2}))
-\alpha_j \sin(\beta)\cos (\frac{\alpha_j}{\pi} (\beta-\frac{\pi}{2})) \Big)|
\\
&\le \frac{\rho_j}{\tilde\rho_j}\beta^{11}_j,
\end{align*}
where $\beta^{11}_j=\max_{\beta\in[0,\pi]} \frac{1}{\pi}| \pi \cos(\beta) \sin (\frac{\alpha_j}{\pi}(\beta-\frac{\pi}{2}))
-\alpha_j \sin(\beta)\cos (\frac{\alpha_j}{\pi} (\beta-\frac{\pi}{2})|$ and hence take $\beta_j=max_{i,k=1,2} \beta^{ik}_j$.  

Next define the extension on $B(\boldsymbol{0};\tilde\rho_j)$ by reflection of $\theta_j n\circ T_P\circ H_j\circ T_P^{-1}$, which is denoted by $(\theta_j \eta\circ T_P\circ H_j\circ T_P^{-1})^*$. We thus have that $(\theta_j \eta\circ T_P\circ H_j\circ T_P^{-1})^*\in H^1_0(B(\boldsymbol{0};\tilde\rho_j))$. In view of \eqref{ineq:Moser-Trudinger}, we find that 
$$
\begin{array}{rcl}
\displaystyle
\int_{U_j\cap\Omega} e^{(\theta_j \eta)(\x)}\dx&=&\displaystyle
\frac{\rho_j^2}{2\tilde\rho_j^2}\frac{\alpha_j}{\pi}\int_{B(\boldsymbol{0},\tilde\rho_j)} e^{(\theta_j\eta\circ T_P\circ H_j\circ T_P^{-1})^*(\x)}\dx
\\
&\le&\displaystyle C \frac{\rho_j^2}{2\tilde\rho_j^2}\frac{\alpha_j}{\pi} |B(\boldsymbol{0},\tilde\rho_j)|
e^{\frac{1}{16\pi}\|\nabla(\theta_j \eta\circ T_P\circ H_j\circ T_P^{-1})^*\|^2_{\L^2(B(\boldsymbol{0},\tilde\rho_j))}}
\\
&\le&\displaystyle C  |U_j\cap \Omega|
e^{\frac{1}{8\pi}\|\nabla(\theta_j \eta\circ T_P\circ H_j\circ T_P^{-1})\|^2_{\L^2(B_+(\boldsymbol{0},\tilde\rho_j))}}
\\
&\le&\displaystyle C |U_j\cap \Omega|
e^{\beta_j^2\frac{1}{8\alpha_j}\|\nabla(\theta_j \eta)\|^2_{\L^2(U_j\cap\Omega)}}
\\
&\le&\displaystyle C  |U_j\cap \Omega|
e^{\beta_j^2\frac{1+\lambda}{8\alpha_\Omega}\|\nabla \eta\|^2_{L^2(\Omega)}+  C_{\lambda} \|\eta\|^2_{L^1(\Omega)}},
\end{array}
$$
where we used the fact that 
$$
\begin{array}{rcl}
\|\nabla(\theta_j \eta\circ T_P\circ H_j\circ T_P^{-1})\|_{\L^2(B_+(\boldsymbol{0},\tilde\rho_j))}&\le&\displaystyle
 \|J_{T_P\circ H_j\circ T_P^{-1}}\|_{\L^\infty(\Omega)}\| \|\nabla(\theta_j \eta)(T_P\circ H_j\circ T_P^{-1})\|_{\L^2(B_+(\boldsymbol{0};\tilde\rho_j))}
\\
&\le&\displaystyle
\beta_j \frac{\rho_j}{\tilde\rho_j} \|\nabla(\theta_j \eta)(T_P\circ H_j\circ T_P^{-1})\|_{\L^2(B_+(\boldsymbol{0};\tilde\rho_j))}
\\
&\le &\beta_j \sqrt{\frac{\pi}{\alpha_j}} \|\nabla(\theta_j \eta)\|_{\L^2(U_j\cap\Omega)}.
\end{array}
$$
If $0<\alpha_j< \pi $, one can check that $0<\beta_j<1$; otherwise, if $\pi\le\alpha_j < 2 \pi$, one has $1\le\beta_j<2$. 

When $\boldsymbol{p}_j\not\in\mathcal{V}_h$, one does select $\alpha_j=\pi$ to define $H_j\equiv Id$.  Thus 
$$
\begin{array}{rcl}
\displaystyle
\int_{U_i\cap\Omega} e^{(\theta_i\eta)(\x)}\dx&=&\displaystyle
\frac{1}{2}\int_{B(\boldsymbol{p}_j,\tilde\rho_j)} e^{(\theta_j \eta)^*(\x)}\,\dx
\\
&\le&\displaystyle C  |U_j\cap \Omega|
e^{\frac{1}{16\pi}\|\nabla(\theta_j \eta)^*\|^2_{\L^2(B(\boldsymbol{p}_j,\tilde\rho_j))}}
\\
&\le&\displaystyle C  |U_j\cap \Omega|
e^{\frac{1}{8\pi}\|\nabla(\theta_i \eta)\|^2_{\L^2(U_j\cap\Omega)}}
\\
&=&\displaystyle C |U_j\cap \Omega|
e^{\frac{1+\lambda}{8\alpha_j}\|\nabla\eta\|^2_{\L^2(\Omega)}+  C_{\lambda} \|\eta\|^2_{L^1(\Omega)}},
\end{array}
$$
where we took $\beta_j=1$.

Defining $\beta_\Omega=\max_{j=1,\cdots, J} \beta_j$ and returning to \eqref{partition}, we obtain, on using the above estimates, that 
$$
\begin{array}{rcl}
\displaystyle
\int_\Omega  e^{\eta(\x)}\dx&\le&\displaystyle  C 
\sum_{j=0}^J  |U_j\cap \Omega|
e^{\beta^2_\Omega\frac{1+\lambda}{8\alpha_\Omega}\|\nabla \eta\|^2_{\L^2(\Omega)}+ C_{\lambda} \|\eta\|^2_{L^1(\Omega)}}
\\
&=&\displaystyle C  |\Omega| e^ {\beta^2_\Omega\frac{1+\lambda}{8\alpha_\Omega} \|\nabla \eta\|^2_{\L^2(\Omega)}+ C_{\lambda} \|\eta\|^2_{L^1(\Omega)}}.
\end{array}
$$
\end{proof}

\begin{remark} Inequality \eqref{ineq:Trudinger-Moser_New} may be improved if some geometrical properties of $\partial\Omega$ are  taken into account. For instance, if $\Omega=[0,\ell]\times[0,\ell]$ with $\ell>0$,  it takes the form
\begin{equation}\label{ineq:Trudinger-Moser_New_square}
\int_\Omega  e^{\eta(\x)}\dx\le C e^{\displaystyle\frac{1+\lambda}{8\pi} \|\nabla \eta\|^2_{L^2(\Omega)}+C_\lambda\|\eta\|^2_{L^1(\Omega)}},
\end{equation}
since $\alpha_j=\frac{\pi}{2}$ and $\beta_j=\frac{\sqrt{2}}{2}$, where $\partial \Omega$ may be recovered by $B(\boldsymbol{p}_j, \sqrt{2}\ell)$, with $\boldsymbol{p}_j\in\mathcal{V}_h$.
Furthermore, if $\Omega$ is convex, one has $\beta_\Omega=1$. 
\end{remark}

A further generalisation will play a crucial role in connection with the finite element framework.  

\begin{theorem}  Let $\eta_h\in X_h$ with $\eta_h>0$. Then, for $\lambda>0$, there exists a constant $C_\lambda>0$, independent of $h$, such that
\begin{equation}\label{functional_ineq_I}
\int_\Omega i_h (e^{\eta_h(\x)})\,\dx\le C (1+ \|\nabla \eta_h\|^2_{\L^2(\Omega)}) e^{\displaystyle\beta_\Omega^2\frac{1+\lambda}{8\alpha_\Omega} \|\nabla \eta_h\|^2_{\L^2(\Omega)}+C_\lambda\|\eta_h\|^2_{L^1(\Omega)}}.
\end{equation}
\end{theorem}
\begin{proof} By invoking the inequality \cite[Co. 2.6]{GS_RG_2021}, for any $m\in\mathds{N}$, 
$$
\|\eta_h^m-i_h(\eta_h^m)\|_{L^1(\Omega)}\le
C m(m-1)h^2 \int_\Omega |\eta_h(\x)|^{m-2} |\nabla \eta_h(\x)|^2\, \dx,
$$
we have
\begin{equation}\label{co2.8-lab1}
\begin{array}{rcl}
\displaystyle
\int_\Omega i_h (e^{\eta_h(\x)})\,\dx&=&\displaystyle
\int_\Omega (1+\eta_h(\x))\,\dx+\sum_{m=2}^\infty\frac{1}{m!} \int_\Omega i_h(\eta_h^m(\x))\dx
\\
&\le&\displaystyle
\sum_{n=0}^\infty\frac{1}{m!} \int_\Omega \eta^m_h(\x)\,\dx
 \\
&&\displaystyle +\sum_{m=2}^\infty\frac{C m(m-1) h^2}{m!}  \int_\Omega |\nabla \eta_h(\x)|^2 \eta_h^{m-2}(\x)\,\dx
\\
&=&\displaystyle
\int_\Omega (1+ C h^2 |\nabla \eta_h(\x)|^2) e^{\eta_h(\x)}\,\dx
\\
&\le&\displaystyle
(1+ C\|\nabla \eta_h(\x)\|^2_{\L^2(\Omega)}) \int_\Omega e^{\eta_h(\x)}\,\dx.
\end{array}
\end{equation}
In the last line the inverse inequality $ \|\nabla \eta_h\|_{\L^\infty(\Omega)}\le C h^{-1} \|\nabla \eta_h\|_{\L^2(\Omega)}$ was utilised. Inequality \eqref{functional_ineq_I} follows on applying \eqref{ineq:Trudinger-Moser_New}. 
\end{proof}

For convenience, we use $\chi_\Omega$ for a shorthand of $\frac{\alpha_\Omega}{\beta^2_\Omega}$, where $\alpha_\Omega$ and $\beta_\Omega$ were defined  in Theorem~\ref{th:Moser-Trundinger}.
\begin{corollary} Let $\phi_h, \psi_h\in X_h$ with $\phi_h, \psi_h>0$. Then, for any $\lambda,\delta, \mu>0$, there exists $M=M(\lambda,\delta, \Omega)>0$ such that  
\begin{equation}\label{functional_ineq_II}
\begin{array}{rcl}
(\phi_h, \psi_h)_h&\le&\displaystyle\frac{1}{\mu} (\phi_h, \log\frac{\phi_h}{\bar \phi_h} )_h
\\
&&\displaystyle
+ \mu (\delta + \frac{1+\lambda}{8\chi_\Omega})\|\phi_h\|_{L^1(\Omega)} \|\nabla \psi_h\|^2_{\L^2(\Omega)}
\\
&&\displaystyle
 +M \mu \|\phi_h\|_{L^1(\Omega)} \|\psi_h\|^2_{L^1(\Omega)}
\\
&&\displaystyle
+\frac{M}{\mu} \|\phi_h\|_{L^1(\Omega)}.
\end{array} 
\end{equation}
\end{corollary}
\begin{proof} On the one hand, from Jensen's inequality, we obtain
$$
\begin{array}{rcl}
\displaystyle
\log \int_\Omega i_h (e^{\mu \psi_h(\x)})\,\dx&=&\displaystyle
\log \sum_{i\in I} \frac{e^{\mu\psi_i}  \|\phi_h\|_{L^1(\Omega)}}{\phi_i} \frac{\phi_i}{\|\phi_h\|_{L^1(\Omega)}}\int_\Omega \varphi_{\a_i}(\x)\,\dx
\\
&\ge&\displaystyle
\sum_{i\in I} \log \left(\frac{e^{\mu\psi_i} \|\phi_h\|_{L^1(\Omega)}}{\phi_i}\right) \frac{\phi_i}{ \|\phi_h\|_{L^1(\Omega)}}\int_\Omega \varphi_{\a_i}(\x)\,\dx
\\
&=&\displaystyle
\sum_{i\in I} \Big(\mu\psi_i+\log \|\phi_h\|_{L^1(\Omega)}- \log \phi_i\Big)\frac{\phi_i}{\|\phi_h\|_{L^1(\Omega)}}\int_\Omega \varphi_{\a_i}(\x)\,\dx
\\
&=&\displaystyle
\frac{\mu}{ \|\phi_h\|_{L^1(\Omega)}} (\phi_h, \psi_h)_h-\frac{1}{ \|\phi_h\|_{L^1(\Omega)}} (\phi_h, \log\frac{\phi_h}{\bar\phi_h} )_h+\log |\Omega|.
\end{array}
$$
On the other hand, from \eqref{functional_ineq_I}, we get 
$$
\log \int_\Omega i_h (e^{\mu \psi_h(\x)})\,\dx\le \log \frac{C}{\delta}+ \log(\delta(1+ \mu^2\|\nabla \psi_h\|^2_{\L^2(\Omega)}))+ \frac{(1+\lambda)\mu^2}{8\chi_\Omega} \|\nabla \psi_h\|^2_{\L^2(\Omega)}+C_\lambda \mu^2\|\psi_h\|^2_{L^1(\Omega)},
$$
whence 
$$
\log \int_\Omega i_h (e^{\mu \psi_h(\x)})\,\dx\le \log \frac{C}{\delta}+\delta+(\delta +\frac{1+\lambda}{8\chi_\Omega})\mu^2 \|\nabla \psi_h\|^2_{\L^2(\Omega)}+C_\lambda \mu^2\|\psi_h\|^2_{L^1(\Omega)}.
$$
Combining the above inequalities yields 
$$
\begin{array}{rcl}
(\phi_h, \psi_h)_h&\le&\displaystyle
\frac{1}{\mu} (\phi_h, \log\frac{\phi_h}{\bar \phi_h} )_h
\\
&&\displaystyle
+\frac{\|\phi_h\|_{L^1(\Omega)}}{\mu}(\log\frac{C }{\delta |\Omega|}+\delta)
\\
&&\displaystyle
+ (\delta+\frac{1+\lambda}{8\chi_\Omega}) \mu \|\phi\|_{L^1(\Omega)} \|\nabla \psi_h\|^2_{\L^2(\Omega)}
\\
&&\displaystyle
+C_\lambda \mu \|\phi\|_{L^1(\Omega)} \|\psi_h\|^2_{L^1(\Omega)}.
\end{array}
$$
Selecting $M=\max\{\log\frac{C_\lambda }{\delta |\Omega|}+\delta, C_\lambda\}$ completes the proof. 
\end{proof}
\begin{corollary} Let $\phi_h\in X_h$ with $\phi_h>0$.  Then, for any $\lambda>0$, there exists $M=M(\lambda, \Omega)>0$ such that
\begin{equation}\label{functional_ineq_III}
\begin{array}{rcl}
(\phi_h, \log(\phi_h+1))_h&\le&\displaystyle
4(\delta + \frac{1+\lambda}{8\chi_\Omega}) \|\phi_h\|_{L^1(\Omega)} \|\nabla i_h(\log(\phi_h+1))\|^2_{\L^2(\Omega)}+4 M \|\phi_h\|_{L^1(\Omega)}^3
\\
&&\displaystyle
+(M-\log \bar\phi_h) \|\phi_h\|_{L^1(\Omega)}.
\end{array}
\end{equation}
\end{corollary}
\begin{proof} Select $\psi_h=i_h(\log(\phi_h+1))$ and  $\mu=2$ in \eqref{functional_ineq_II} to get
$$
\begin{array}{rcl}
(\phi_h, \log(\phi_h+1))_h&\le&\displaystyle
\frac{1}{2} (\phi_h, \log\frac{\phi_h}{\bar \phi_h} )_h+\frac{M}{2} \|\phi_h\|_{L^1(\Omega)}
\\
&&\displaystyle
 + 2 (\delta + \frac{1+\lambda}{8\chi_\Omega}) \|\phi_h\|_{L^1(\Omega)} \|\nabla i_h(\log(\phi_h+1))\|^2_{\L^2(\Omega)}
 \\
 &&
 \displaystyle
 +2 M\|\phi_h\|_{L^1(\Omega)} \|i_h(\log(\phi_h+1))\|^2_{L^1(\Omega)}.
\end{array} 
$$
On noting that 
$$
\begin{array}{rcl}
\displaystyle
\frac{1}{2}(\phi_h, \log\frac{\phi_h}{\bar \phi_h})_h&=&\displaystyle
\frac{1}{2}(\phi_h, \log\phi_h)_h-\frac{1}{2} \log \bar\phi_h\|\phi_h\|_{L^1(\Omega)} 
\\
&\le&\displaystyle
\frac{1}{2}(\phi_h, \log(\phi_h+1))_h-\frac{1}{2}\log\bar\phi_h \|\phi_h\|_{L^1(\Omega)}
\end{array}
$$
and that $\| i_h(\log(\phi_h+1))\|_{L^1(\Omega)}\le \|\phi_h\|_{L^1(\Omega)}$, we complete the proof.
\end{proof}

We end up this section with a pointwise inequality.
\begin{proposition} It follows that, for all $x, y>0$, 
\begin{equation}\label{ineq:log}
(\log(x+1)-\log(y+1))^2\le \frac{(x-y)^2}{(x+1)(y+1)}
\end{equation}
holds.
\end{proposition}
\begin{proof} Let $g(z)= (z-1)^2-z \log^2 z$.  We know that $g(z)\ge 0$ for $z>0$ since $g$ attains its unique global minimum $0$ at $z=1$. Therefore, we deduce that  $\log^2 z\le \frac{(z-1)^2}{z}$ for all $z>0$. Taking $z=\frac{x+1}{y+1}$ completes the proof.
\end{proof}

\section{Main results}
In this section the main theoretical results of this work are stated and proved: lower bounds, time-independent and -dependent integrability bounds.  These will be successfully attained as an upshot of  the new discretisation for the convective and chemotaxis terms in \eqref{eq:n_h} and of the design of  the stabilising terms $B_n$ and $B_c$ in \eqref{Bn} and $\eqref{Bc}$, which have been neatly devised. Nor should the role of the functional inequalities \eqref{functional_ineq_II}  and  \eqref{functional_ineq_III} be forgotten, which will be of great importance; its last consequence is  deriving a priori bounds under a smallness condition for $\|n_0\|_{L^1(\Omega)}$.

\subsection{Lower bounds} We open our discussion with the proof of the lower bounds for $(n^{m+1}_h, c^{m+1}_h)$, because they are closely connected with the $L^1(\Omega)$ bounds and later on with the a priori energy bounds.  
\begin{lemma} It follows that the discrete solution pair  $(n^{m+1}_h, c^{m+1}_h)$ of \eqref{eq:n_h}-\eqref{eq:c_h} satisfies 
\begin{equation}\label{lower_bounds}
n^{m+1}_h>0\quad\mbox{ and }\quad c^{m+1}_h\ge 0.
\end{equation} 
\end{lemma}
\begin{proof} We establish \eqref{lower_bounds} by induction on $m$; the case $m=0$ being entirely analogue to the general one is omitted. It will be firstly shown that $n^{m+1}_h$ cannot take non-positive values. Suppose the contrary at a certain node $\boldsymbol{a}_i\in\mathcal{N}_h$, i.e., $n^{m+1}_i:=n^{m+1}_h(\boldsymbol{a}_i)\le 0$, which is a local minimum. For the sake of simplicity and with no loss in generality, one may order the indexes such $i<j$ for all $ j\in I(\Omega_{\a_i})$. Let $I^*(\Omega_{\a_i})=\{j\in I(\Omega_{\a_i}): n^{m+1}_j=n^{m+1}_i\}$ and let  $I^*_c(\Omega_{\a_i})$ be  its complementary.  Then choose $\bar n_h=\varphi_{\a_i}$ in \eqref{eq:n_h} to get 
$$
(\delta_t n^{m+1}_h, \varphi_{\a_i})_h-(n^{m+1}_h\u^m_h, \nabla\varphi_{\a_i})_*+(\nabla n^{m+1}_h, \nabla \varphi_{\a_i})+(B_n(n^{m+1}_h, \u^m_h) n^{m+1}_h, \varphi_{\a_i})=0,
$$
where use is made of the fact that  $(n^{m+1}_h\nabla  c^{m+1}_h, \nabla \varphi_{\a_i})_*=0$ from \eqref{def: gamma_ij_chem} and \eqref{New_KS_term} due to $[n^{m+1}_i]_+=0$. By virtue of \eqref{New_conv_term: extended} and \eqref{Bn}, one writes
$$
\begin{array}{c}
\displaystyle
k^{-1}(1,\varphi_{\a_i}) n^{m+1}_i+\sum_{j\in I^*_c(\Omega_{\a_i})} (n_j^{m+1}-n_i^{m+1})  (\nabla\varphi_{\a_j}, \nabla\varphi_{\a_i})
\\
\displaystyle
+\sum_{j\in I^*_c(\Omega_{\a_i})} (n^{m+1}_j+1) (n^{m+1}_i+1)\frac{g_\varepsilon(n^{m+1}_j+1)-g_\varepsilon(n^{m+1}_i+1)}{(n^{m+1}_j-n^{m+1}_i)^2}(n^{m+1}_j-n^{m+1}_i)(\varphi_{\a_j}\u^m_h, \nabla\varphi_{\a_i})
\\
\displaystyle
-\sum_{j\in I^*_c(\Omega_{\a_i})}
\nu_{ji}^n(n^{m+1}_h,\u^m_h) (n^{m+1}_j- n^{m+1}_i)
\\
= k^{-1}  (1,\varphi_{\a_i})n^m_i.
\end{array}
$$
From the simple observation that 
$$
\begin{array}{rcl}
\displaystyle
(n^{m+1}_j+1)(n^{m+1}_i+1)\frac{g_\varepsilon(n^{m+1}_j+1)-
g_\varepsilon(n^{m+1}_i+1)}{(n^{m+1}_i-n^{m+1}_j)^2}(\varphi_{\a_j}\u^m_h, \nabla\varphi_{\a_i})
\\
+(\nabla\varphi_{\a_j}, \nabla\varphi_{\a_i})-\nu_{ji}^n(n^{m+1}_h,\u^m_h)&\le& 0 
\end{array}
$$
holds for all $j\in I^*_c(\Omega_i)$ on noting \eqref{def:nu_n} and $\alpha_i(n^{m+1}_h)=1$ from Lemma \ref{lm: alpha_i}, it follows that 
$$
0>n^{m+1}_i\ge n^m_i >0. 
$$ 
This gives a contradiction since $n^m_i>0$ by the induction hypothesis. 

With little change in argument, one can prove  $c^{m+1}_h>0$. Just as before, let $\a_i\in\mathcal{N}_h$ be a local minimum of $c^{m+1}_h$ such that $c^{m+1}_i<0$ and set $\bar c_h=\varphi_{\a_i}$ in \eqref{eq:c_h} to find 
$$
\begin{array}{c}
\displaystyle
 \sum_{j\in I(\Omega_{\a_i})} c^{n+1}_j\Big[(1+k^{-1})(\varphi_{\a_j},\varphi_{\a_i})+(\nabla\varphi_{\a_j},\nabla\varphi_{\a_i})+(\u^m_h\cdot\nabla \varphi_{\a_j},\varphi_{\a_i})  
 \\
 \displaystyle
+\frac{1}{2} (\nabla\cdot\u^m_h\,\varphi_{\a_j}, \varphi_{\a_i})+ (B_c(c^{m+1}_h, \u^m_h)\varphi_{\a_j},\varphi_{\a_i}) \Big]= n^{n+1}_i(1,\varphi_{\a_i}) + k^{-1}\sum_{j\in I(\Omega_{\a_i})} c^{m}_j (\varphi_{\a_j},\varphi_{\a_i}).
\end{array}
$$
On account of \eqref{Bc} and $\alpha_i(c^{m+1}_h)=1$ from Lemma \ref{lm: alpha_i}, one has 
\begin{align*}
k^{-1}(\varphi_{\a_j}, \varphi_{\a_i})+(\u^m_h\cdot\nabla\varphi_{\a_j}, \varphi_{\a_i})+\frac{1}{2}(\nabla\cdot \u^m_h \varphi_{\a_j}, \varphi_{\a_i}) &
\\
+(\nabla\varphi_{\a_j},\nabla\varphi_{\a_i})+ (\varphi_{\a_j}, \varphi_{\a_i}) + (B_c(c^{m+1}_h, \u^m_h)\varphi_{\a_j},\varphi_{\a_i}) &\le 0\quad \mbox{ for }\quad i\not=j.
\end{align*}
Thus, 
$$
\begin{array}{rcl}
\displaystyle
\sum_{j\in I(\Omega_{\a_i})} c^{n+1}_i\Big[(1+k^{-1})(\varphi_{\a_j},\varphi_{\a_i}) +(\nabla\varphi_{\a_j},\nabla\varphi_{\a_i}) 
+(\u^m_h\cdot\nabla\varphi_{\a_j}, \varphi_{\a_i})
\\
\displaystyle
+\frac{1}{2}(\nabla\cdot \u^m_h \varphi_{\a_j}, \varphi_{\a_i})
+ (B_c(c^{m+1}_h)\varphi_{\a_j},\varphi_{\a_i}) \Big] & \ge &
\\
\displaystyle
\sum_{j\in I(\Omega_{\a_i})} (\varphi_{\a_j},\varphi_{\a_i}) (k^{-1}c_j^m + n^{m+1}_j) &\ge& 0.
\end{array}
$$
Observe that  $(\nabla1,\nabla \varphi_{\a_j})=0$, $(\u^m_h\cdot\nabla 1, \varphi_{\a_i})+\frac{1}{2}(\nabla\cdot \u^m_h, \varphi_{\a_i})=0$  and $(B_c(c^{m+1}_h, \u^m_h) 1,\varphi_{\a_i})=0$ from \eqref{eq:p_h} and \eqref{Bc}-\eqref{def:nu^c_ii}, respectively, which in turn imply that  
$$
0>c^{m+1}_i (1+k^{-1}) \ge (n_j^{m+1}+k^{-1}c_j^m ) \ge 0,
$$
which is a contradiction; thereby closing the proof.
\end{proof}

\begin{remark} Now that it is known that  $n^{m+1}_h$ is positive, we are allowed to rule out the truncating operator $[\cdot]_+$ in \eqref{def: gamma_ij_chem} defining \eqref{New_KS_term}. Furthermore, we can consider $\gamma^{\rm c}_{ji}$ defined as in \eqref{def: gamma_ij_conv not extended} rather than in \eqref{def: gamma_ij_conv extended}, since  $g_\varepsilon(n^{m+1}_h+1)$ is nothing but $\log(n^{m+1}_h+1)$ from our choice of $\varepsilon$.
\end{remark}

\subsection{\texorpdfstring{$L^1(\Omega)$}{Lg} bounds}
The $L^1(\Omega)$ bounds for $(n^{m+1}_h, c^{m+1}_h)$ are the second step to regarding from \eqref{eq:n_h} and \eqref{eq:c_h}. They are somehow naturally inherited from our starting algorithm; namely, from equations \eqref{eq_aux:n_h} and \eqref{eq_aux:c_h} and from the conservation structures \eqref{def:nu^n_ii} and \eqref{def:nu^c_ii} for the stabilising operators $B_n$ and $B_c$ in \eqref{Bn} and \eqref{Bc}, respectively. 
\begin{lemma}[$L^1(\Omega)$-bounds]  There holds that the discrete solution pair $(n^{m+1}_h, c_h^{n+1})\in N_h\times C_h$ computed via \eqref{eq:n_h} and \eqref{eq:c_h} fulfils
\begin{equation}\label{L1-Bound-nh}
\|n^{m+1}_h\|_{L^1(\Omega)}=\|n_h^0\|_{L^1(\Omega)}
\end{equation}
and
\begin{equation}\label{L1-Bound-ch}
\|c^{m+1}_h\|_{L^1(\Omega)}\le \frac{1}{(1+k)^{m+1}}  \|c^0_h\|_{L^1(\Omega)}+(1-\frac{1}{(1+k)^m})\|n^0_h\|_{L^1(\Omega)}.
\end{equation}
\end{lemma}
\begin{proof} On choosing $\bar n_h=1$ in \eqref{eq:n_h} and on noting that $(B_n(n^{m+1}_h,\u^m_n) n^{m+1}_h, 1)=0$ holds, it follows immediately after a telescoping cancellation that
\begin{equation}\label{lm5.2-lab1}
\int_\Omega n^{m+1}_h(\x)\,\dx=\int_\Omega n^0_h(\x)\,\dx.
\end{equation}
Consequently we get that \eqref{L1-Bound-nh} holds by the positivity of $n^{n+1}_h$ and $n^0_h$. Likewise let $\bar c_h=1$ in \eqref{eq:c_h} to get, on noting $(B_c(n^{m+1}_h,\u^m_n) c^{m+1}_h, 1)=0$ as well, that 
$$
\int_\Omega c^{m+1}_h(\x)\,\dx+k \int_\Omega c^{m+1}_h(\x)\,\dx=\int_\Omega c^m_h(\x)\,\dx+k\int_\Omega n^{m+1}_h(\x)\,\dx.
$$
A simple calculation shows that
$$
\int_\Omega c^{m+1}_h(\x)\,\dx=\frac{1}{(1+k)^{m+1}} \int_\Omega c^0_h(\x)\,\dx+\left(\int_\Omega n^0_h(\x)\,\dx\right) \sum_{j=1}^{m+1}\frac{k}{(1+k)^j} ,
$$
where use was made of \eqref{lm5.2-lab1}. Thus we infer that 
$$
\int_\Omega c^{m+1}_h(\x)\,\dx=\frac{1}{(1+k)^{m+1}} \int_\Omega c^0_h(\x)\,\dx+ (1-\frac{1}{(1+k)^m})\int_\Omega n^0_h(\x)\,\dx,
$$
thereby concluding the proof of \eqref{L1-Bound-ch} by the non-negativity of $c^{m+1}_h$ and $c^0_h$ and again the positivity of $n^0_h$. 
\end{proof}

\subsection{A priori energy bounds}
We turn finally to a discussion of our third objective, which is showing a priori bounds for $(n^{m+1}_h, c^{m+1}_h, \u^{m+1}_h)$ as an application of \eqref{functional_ineq_II} and \eqref{functional_ineq_III}. We will need on the way to assume that  $\mathcal{T}_h$ is weakly acute, i.e. if every sum of two angles opposite to an interior edge does not exceed $\pi$. As a result, there holds
\begin{equation}\label{cond:acute}
(\nabla\varphi_{\a_i}, \nabla\varphi_{\a_j})\le 0 \quad \mbox{ for all } \quad i\not= j\in I.
\end{equation}

Let us define the  following energy-like functional: 
$$
\mathcal{F}_h(n_h, c_h) =-(\log(n_h+1), 1)_h+\|c_h\|^2_{L^2(\Omega)}.
$$
Our first step is determining the evolution of this functional, which is independent of the fluid part. 
\begin{lemma}\label{lm:a_priori_KS} Let $n_h^0\in N_h$ be such that 
\begin{equation}\label{mass_condition: n_0}
\int_\Omega n_h^0(\x)\,\dx<2\chi_\Omega.
\end{equation} 
Then there exist two constants  $F(n^0_h,c^0_h)>0$ and $ D(n^0_h)>0$ such that  the discrete solution pair $(n^{m+1}_h, c^{m+1}_h)\in N_h\times C_h$ computed via \eqref{eq:n_h} and \eqref{eq:c_h} satisfies 
\begin{equation}\label{ineq:energy_nh_ch}
\begin{array}{rcl}
\displaystyle
\mathcal{F}(n^m_h, c^m_h)-\mathcal{F}(n^{m+1}_h, c^{m+1}_h)&&
\\
\displaystyle
+ \frac{k}{2} \Big(\frac{1}{(n^{m+\theta}_h)^2},(\delta_t u^{m+1}_h)^2\Big)_h+ k D(n^0_h) \|\nabla i_h(\log (n^{m+1}_h+1))\|^2_{\L^2(\Omega)}&&
\\
\displaystyle
+k D(n^0_h)\|\nabla c^{m+1}_h\|^2_{\L^2(\Omega)}+\lambda (n^{m+1}_h, \log\frac{n^{m+1}_h}{\bar n^0_h} )_h&\le& F(n^0_h, c^0_h).
\end{array}
\end{equation} 
\end{lemma}
\begin{proof}  Select $\bar n_h=-i_h \frac{1}{n^{m+1}_h+1}$ in \eqref{eq:n_h} and $\bar c_h= c^{m+1}_h$ in \eqref{eq:c_h} to get
\begin{equation}\label{lm4.2-lab1}
\begin{array}{rcl}
\displaystyle
-(\delta_t n^{m+1}_h, i_h \frac{1}{n^{m+1}_h+1})_h&+&\displaystyle
(n^{m+1}_h\u^m_h, \nabla i_h \frac{1}{n^{m+1}_h+1})_*
\\
&-&\displaystyle
(\nabla n^{m+1}_h, \nabla i_h \frac{1}{n^{m+1}_h+1})
\\
&+&\displaystyle
(n^{m+1}_h\nabla c^{m+1}_h, \nabla i_h \frac{1}{n^{m+1}_h+1})_*
\\
&-&\displaystyle
(B_n(n^{m+1}_h, \u^m_h) n^{m+1}_h,i_h \frac{1}{n^{m+1}_h+1})=0
\end{array}
\end{equation}
and
\begin{equation}\label{lm4.2-lab2}
\frac{1}{2 k}\|c^{m+1}_h\|^2_{L^2(\Omega)}-\frac{1}{2 k}\|c^m_h\|^2_{L^2(\Omega)}+\frac{1}{2k}\|c^{m+1}_h-c^m_h\|^2_{L^2(\Omega)}
+\|c^{m+1}_h\|^2_{H^1(\Omega)}=(n_h^{m+1}, c^{m+1}_h)_h.
\end{equation}

Taylor's theorem applied to $\log(s+1)$ for $n^{m+1}_h$ and evaluated at $n^m_h$ gives
$$
\log(n^m_h+1)=\log(n^{m+1}_h+1)+\frac{1}{n^{m+1}_h+1}(n^m_h-n^{m+1}_h)-\frac{1}{2(n^{m+\theta}_h)^2}(n^{m}_h-n^{m+1}_h)^2,
$$
where $\theta\in(0,1)$ is such that $n^{n+\theta}_h=\theta n^{m+1}_h+(1-\theta) n^n_h$. Hence, 
$$
-\Big(\delta_t n^{m+1}_h, \frac{1}{n^m_h+1}\Big)_h=\frac{1}{k}\Big(\log(n^m_h+1),1\Big)_h-\frac{1}{k}\Big(\log(n^{m+1}_h+1),1\Big)_h+ \frac{k}{2} \Big(\frac{1}{(n^{m+\theta}_h)^2},(\delta_t n^{m+1}_h)^2\Big)_h.
$$

The convective term is handled, on noting \eqref{New_conv_term: extended} together with   \eqref{def: gamma_ij_conv not extended}, as: 
$$
\begin{array}{rcl}
\displaystyle
-( n^{m+1}_h\u^m_h, \nabla i_h \frac{1}{n^{m+1}_h+1})_*&=&\displaystyle
\sum_{i,j\in I}  (\log(n^{m+1}_j+1)-\log(n^{m+1}_i+1))
(\varphi_{\a_j} \u^m_h, \nabla\varphi_{\a_i})
\\
&=&
-(\u^m_h,\nabla i_h(\log(n^{m+1}_h+1))=0,
\end{array}
$$
where we used twice that $(\u^m_h, \nabla \bar n_h)=0$, for all $\bar n_h\in N_h$, from the incompressibility condition \eqref{eq:p_h}.

The diffusion term is treated as:
$$
-(\nabla n^{m+1}_h, \nabla i_h \frac{1}{n^{m+1}_h+1})=-\sum_{i<j\in I}\frac{(n^{m+1}_i-n^{m+1}_j)^2}{(n^{m+1}_j+1)(n^{m+1}_i+1)} (\nabla\varphi_{\a_j}, \nabla\varphi_{\a_i})\ge 0
$$  
owing to \eqref{cond:acute}.

The chemotaxis term is estimated, on using \eqref{New_KS_term} together with \eqref{def: gamma_ij_chem}, as follows: 
\begin{align*}
\displaystyle
(n^{m+1}_h\nabla c^{m+1}_h&,\nabla i_h \frac{1}{n_h^{m+1}+1})_*
\\
=&\displaystyle
\sum_{i<j\in I}\sqrt{n^{m+1}_i} \sqrt{n^{m+1}_j} (c^{n+1}_j-c^{m+1}_i) \frac{n^{m+1}_j-n^{m+1}_i}{(n^{m+1}_i+1)(n^{m+1}_j+1)}(\nabla\varphi_{\a_j}, \nabla\varphi_{\a_i})
\\
\le&\displaystyle
-\left(\sum_{i<j\in I} \frac{n^{m+1}_i n^{m+1}_j}{(n^{m+1}_h+1)(n^{m+1}_j+1)} (c^{m+1}_j-c^{m+1}_i)^2  \right)^{\frac{1}{2}}
\\
&\displaystyle
\left(\sum_{i<j\in I} \frac{1}{(n^{m+1}_i+1)(n^{m+1}_j+1)} (n^{m+1}_j-n^{m+1}_i)^2\right)^{\frac{1}{2}}(\nabla\varphi_{\a_j}, \nabla\varphi_{\a_i})
\\
\le&\displaystyle
-\frac{1}{2}\sum_{i<j\in I} (c^{m+1}_j-c^{m+1}_i)^2 (\nabla\varphi_{\a_j}, \nabla\varphi_{\a_i})
\\
&-\displaystyle
\frac{1}{2}\sum_{i<j\in I} \frac{(n^{m+1}_j-n^{m+1}_i)^2}{(n^{m+1}_i+1)(n^{m+1}_j+1)} (\nabla\varphi_{\a_j}, \nabla\varphi_{\a_i}).
\end{align*}

The stabilising terns is rewritten as:
$$
\begin{array}{rcl}
\displaystyle
(B_n(n^{m+1}_h, \u^m_h) n^{m+1}_h, i_h \frac{1}{n^{m+1}_h+1})&=&\displaystyle\sum_{i<j\in I} \nu_{ji}^n (n^{m+1}_j-n^{m+1}_i)\Big(\frac{1}{n^{m+1}_j+1}-\frac{1}{n^{m+1}_i+1}\Big)
\\
&=&\displaystyle
-\sum_{i<j\in I} \nu_{ji}^n \frac{(n^{m+1}_j-n^{m+1}_i)^2}{(n^{m+1}_i+1)(n^{m+1}_j+1)} \le0.
\end{array}
$$ 
This follows since $\nu_{ji}^n>0$, for $ i\not=j$, from  \eqref{def:nu_n}.

Plugging all the above computations yields 
\begin{equation}\label{ineq:energy_I}
\begin{array}{rcl}
\displaystyle
\mathcal{F}_h(n^{m+1}_h, c^{m+1}_h)-\mathcal{F}_h(n^m_h, c^m+_h)&&
\\
\displaystyle
+ \frac{k^2}{2} \Big(\frac{1}{(n^{n+\theta}_h)^2},(\delta_t u^{n+1}_h)^2\Big)_h-k (B_n(n^{m+1}_h, \u^m_h) n^{m+1}_h, \frac{1}{n^{m+1}_h+1})&&
\\
\displaystyle
-\frac{k}{2}\sum_{i<j\in I}\frac{(n^{n+1}_i-n^{n+1}_j)^2}{(n^{n+1}_j+1)(n^{n+1}_i+1)} (\nabla\varphi_{\a_j}, \nabla\varphi_{\a_i})+\frac{k}{2}\|\nabla c^{m+1}_h\|^2_{\L^2(\Omega)}&\le& k (n^{m+1}_h, c^{m+1}_h)_h.
\end{array}
\end{equation}
Our goal now is to treat the term $(n^{m+1}_h, c^{m+1}_h)_h.$ Observe from $\|n_0\|_{L^1(\Omega)}\in (0,2\chi_\Omega)$ that there exists $\mu>0$ such that 
\begin{equation}\label{ineq:mu}
\frac{\|n^0_h\|_{L^1(\Omega)}}{\chi_\Omega}<\mu<\frac{4\chi_\Omega}{\|n^0_h\|_{L^1(\Omega)}}.
\end{equation}
Thus we can select $\delta, \lambda >0$ such that 
\begin{equation}\label{ineq:mu_eps_I}
4(\lambda+\frac{1}{\mu}) \left(\delta+\frac{1+\lambda}{8\chi_\Omega}\right)\|n^0_h\|_{L^1(\Omega)}<\frac{1}{2}
\end{equation}
and
\begin{equation}\label{ineq:mu_eps_II}
\mu\left(\delta + \frac{1+\lambda}{8\chi_\Omega}\right)\|n^0_h\|_{L^1(\Omega)}<\frac{1}{2}.
\end{equation}
In view of \eqref{functional_ineq_II} and \eqref{L1-Bound-nh}  for the above election of $\lambda$ and $\mu$, we have
$$
\begin{array}{rcl}
(n^{m+1}_h, c^{m+1}_h)_h&\le&\displaystyle
\frac{1}{\mu} (n^{m+1}_h, \log\frac{n^{m+1}_h}{\|n^0_h\|_{L^1(\Omega)}} )_h
\\
&&\displaystyle
 + \mu (\delta + \frac{(1+\lambda)}{8\chi_\Omega})  \|n^0_h\|_{L^1(\Omega)} \|\nabla c^{m+1}_h\|^2_{\L^2(\Omega)}
\\  
&&\displaystyle
+M \mu\|n^0_h\|_{L^1(\Omega)} \|c^{m+1}_h\|^2_{L^1(\Omega)}
\\
&&\displaystyle
+\frac{M}{\mu} \|n^0_h\|_{L^1(\Omega)}.
\end{array} 
$$
We continue by applying \eqref{functional_ineq_III} to get 
\begin{equation}\label{ineq:rhs}
\begin{array}{rcl}
(n^{m+1}_h, c^{m+1}_h)_h&+&\displaystyle\lambda (n^{m+1}_h, \log\frac{n^{m+1}_h}{\bar n_{h}^0})_h
\\
&\le&\displaystyle
(\lambda+\frac{1}{\mu}) (n^{m+1}_h, \log\frac{n^{m+1}_h}{\|n^0_h\|_{L^1(\Omega)}} )_h
\\
&&\displaystyle
 + \mu (\delta + \frac{1+\lambda}{8\chi_\Omega})  \|n^0_h\|_{L^1(\Omega)} \|\nabla c^{m+1}_h\|^2_{\L^2(\Omega)}
\\  
&&\displaystyle
+M \mu\|n^0_h\|_{L^1(\Omega)} \|c^{m+1}_h\|^2_{L^1(\Omega)}
\\
&&\displaystyle
+\frac{M}{\mu} \|n^0_h\|_{L^1(\Omega)}
\\
&\le&\displaystyle
(\lambda+\frac{1}{\mu})\Big\{4(\delta + \frac{1+\lambda}{8\chi_\Omega}) \|n^0_h\|_{L^1(\Omega)} \|\nabla i_h(\log(\phi_h+1))\|^2_{\L^2(\Omega)}
\\
&&
+4 M \|n^0_h\|_{L^1(\Omega)}^3
\\
&&
+(M-\log \bar n^0_h) \|n^0_h\|_{L^1(\Omega)}\Big\}
\\
&&\displaystyle
-(\lambda+\frac{1}{\mu})\|n^0_h\|_{L^1(\Omega)}\log \bar n^0_h
\\
&&\displaystyle
+ \mu(\delta + \frac{1+\lambda}{8\chi_\Omega}) \|n^0_h\|_{L^1(\Omega)} \|\nabla c^{m+1}_h\|^2_{\L^2(\Omega)}
\\  
&&\displaystyle
+ M \mu\|n^0_h\|_{L^1(\Omega)} \|c^{m+1}_h\|^2_{L^1(\Omega)}
\\
&&\displaystyle
+\frac{M}{\mu} \|n^0_h\|_{L^1(\Omega)}
\\
&\le&\displaystyle
4(\lambda+\frac{1}{\mu})(\delta + \frac{1+\lambda}{8\chi_\Omega}) \|n^0_h\|_{L^1(\Omega)} \|\nabla i_h(\log(n^{m+1}_h+1))\|^2_{\L^2(\Omega)}
\\
&&\displaystyle
+4 (\lambda+\frac{1}{\mu}) M \|n^0_h\|_{L^1(\Omega)}^3
\\
&&\displaystyle
+ M (\lambda+\frac{1}{\mu})  \|n^0_h\|_{L^1(\Omega)}
\\
&&\displaystyle
 + \mu (\delta + \frac{1+\lambda}{8\chi_\Omega})  \|n^0_h\|_{L^1(\Omega)} \|\nabla c^{m+1}_h\|^2_{\L^2(\Omega)}
\\  
&&\displaystyle
+M \mu\|n^0_h\|_{L^1(\Omega)} \|c^{m+1}_h\|^2_{L^1(\Omega)}
\\
&&\displaystyle
+\frac{M}{\mu} \|n^0_h\|_{L^1(\Omega)}
\\
&&\displaystyle
-2 (\lambda+\frac{1}{\mu}) \|n^0_h\|_{L^1(\Omega)} \log_- \bar n^0_h.
\end{array} 
\end{equation}
On account of \eqref{L1-Bound-ch} and \eqref{ineq:mu}, one has
$$
\begin{array}{rcl}
M \mu\|n^0_h\|_{L^1(\Omega)} \|c^{m+1}_h\|^2_{L^1(\Omega)}&\le& \displaystyle
M \mu \|n^0_h\|_{L^1(\Omega)} \Big(\frac{1}{(1+k)^{m+1}}  \|c^0_h\|_{L^1(\Omega)}+\|n^0_h\|_{L^1(\Omega)}\Big)^2
\\
&\le& \displaystyle
M \mu \|n^0_h\|_{L^1(\Omega)} \Big(2 \|n^0_h\|^2_{L^1(\Omega)}+  \frac{2}{(1+k)^{2(n+1)}}  \|c^0_h\|^2_{L^1(\Omega)}\Big).
\end{array}
$$
Finally,  we deduce, from \eqref{ineq:energy_I}, \eqref{ineq:rhs} and \eqref{ineq:log}, that 
\begin{align*}
\displaystyle
\mathcal{F}(n^m_h, c^{m+1}_h)-\mathcal{F}(n^{m+1}_h, c^m_h)+&D_n  k\|\nabla i_h(\log (n^{m+1}_h+1))\|^2_{\L^2(\Omega)}
\\
\displaystyle
+k D_c\|\nabla c^{m+1}_h\|^2_{\L^2(\Omega)}
&+\lambda k(n^{m+1}_h, \log\frac{n^{m+1}_h}{\bar n^0_h} )_h
\\
\le&\displaystyle
4 (\lambda+\frac{1}{\mu}+\frac{\mu}{2}) M \|n^0_h\|_{L^1(\Omega)}^3
+ M (\lambda+\frac{2}{\mu})  \|n^0_h\|_{L^1(\Omega)}
\\
&+2 M \mu \|n^0_h\|_{L^1(\Omega)} \|c^0_h\|^2_{L^1(\Omega)}
-2 (\lambda+\frac{1}{\mu})  \|n^0_h\|_{L^1(\Omega)}\log_- \bar n^0_h
\\
:=&F(n^0_h, c^0_h),
\end{align*}
where 
$$
D_n=\frac{1}{2}-4(\frac{1}{\mu}+\lambda)(\delta + \frac{1+\lambda}{8\chi_\Omega})\|n^0_h\|_{L^1(\Omega)}>0
$$
and
$$
D_c=\frac{1}{2}-\mu(\delta + \frac{1+\lambda}{8\chi_\Omega})\|n^0_h\|_{L^1(\Omega)}>0.
$$
Choosing $D=\min\{D_n, D_c, \lambda\}$ concludes the proof.  
\end{proof}
In the context of the Navier--Stokes equations, a control of the $L^2(\Omega)$-norm for the velocity components is a basic estimate to be dealt with.  For the Keller--Segel--Navier--Stokes equations, it will be so as well. As the convective term vanishes together with the pressure term when tested against  the solution itself owing to the incompressibility condition,  the fundamental term to be treated is the potential one for which some functional inequalities need to be invoked.   

\begin{lemma}\label{lm:a_priori_NS} It follows that there exists $K>0$ such that the solution $\u^{m+1}_h$ to \eqref{eq:u_h} fulfils 
\begin{equation}\label{ineq:energy_uh}
\frac{1}{k}\|\u^{m+1}_h\|^2_{\L^2(\Omega)}-\frac{1}{k}\|\u^m_h\|^2_{\L^2(\Omega)}+\|\nabla\u^{m+1}_h\|^2_{\L^2(\Omega)}\le K \|n^0_h\|_{L^1(\Omega)} (n^{m+1}_h, \log\frac{n^{m+1}_h}{\bar n^0_h})_h+K \|n^0_h\|^2_{L^1(\Omega)}.
\end{equation}
\end{lemma}
\begin{proof} Substituting $\bar\u_h=\u^{m+1}_h$ into \eqref{eq:u_h} and $\bar p_h=p^{m+1}_h$ into \eqref{eq:p_h} and combining both resulting equations yields 
\begin{equation}\label{lm4.5-lab1}
\frac{1}{2k}\|\u^{m+1}_h\|^2_{\L^2(\Omega)}-\frac{1}{2k}\|\u^m_h\|^2_{\L^2(\Omega)}+\|\nabla\u^{m+1}_h\|^2_{\L^2(\Omega)}=(n^{m+1}_h\nabla\Phi, \u^{m+1}_h)_h.
\end{equation}
Now the right-hand side is estimated as follows. On noting the inequality that results from applying  \eqref{functional_ineq_II} for $\lambda=1$ and that
$$
\|\u^{m+1}_h\|^2_{\L^1(\Omega)}\le |\Omega| \|\u^{m+1}_h\|^2_{\L^2(\Omega)}\le C_{P} |\Omega| \|\nabla \u^{m+1}_h\|^2_{\L^2(\Omega)},
$$
one obtains 
$$
\begin{array}{rcl}(n^{m+1}_h\nabla\Phi, \u^{m+1}_h)_h&\le&\displaystyle
\|\nabla\Phi\|_{\L^\infty(\Omega)} \sum_{\ell=1,2} (n^{m+1}_h,|u_{h\ell}|)
\\
&\le&\displaystyle\frac{2}{\mu} \|\nabla\Phi\|_{\L^\infty(\Omega)}(n^{m+1}_h, \log\frac{n^{m+1}_h}{\bar n^0_h})_h
\\
&&\displaystyle
+ 2\mu (\delta + \frac{1}{4\chi_\Omega})\|\nabla\Phi\|_{\L^\infty(\Omega)}\|n^0_h\|_{L^1(\Omega)} \|\nabla \u^{m+1}_h\|^2_{\L^2(\Omega)}
\\
&&\displaystyle
 +2M \mu \|\nabla\Phi\|_{\L^\infty(\Omega)} \|n^0_h\|_{L^1(\Omega)} \|\u^{m+1}_h\|^2_{L^1(\Omega)}
\\
&&\displaystyle
+\frac{2M}{\mu} \|\nabla\Phi\|_{\L^\infty(\Omega)} \|n^0_h\|_{L^1(\Omega)}
\\
&\le&\displaystyle\frac{2}{\mu} \|\nabla\Phi\|_{\L^\infty(\Omega)}(n^{m+1}_h, \log\frac{n^{m+1}_h}{\bar n^0_h} )_h
\\
&&\displaystyle
+ 2\mu \Big(\delta + \frac{1}{4\chi_\Omega}+M C_P |\Omega|\Big)\|\nabla\Phi\|_{\L^\infty(\Omega)} \|n^0_h\|_{L^1(\Omega)} \|\nabla \u^{m+1}_h\|^2_{\L^2(\Omega)}
\\
&&\displaystyle
+\frac{2M}{\mu} \|\nabla\Phi\|_{\L^\infty(\Omega)} \|n^0_h\|_{L^1(\Omega)}.
\end{array} 
$$
Rearranging the above inequality, on denoting $\tilde K=2\Big(\delta + \frac{1}{4\chi_\Omega}+M C_P |\Omega|\Big)\|\nabla\Phi\|_{\L^\infty(\Omega)} $ and taking $\mu=\frac{1}{2 \|n^0_h\|_{L^1(\Omega)} K(n^0_h)}$,  there results finally 
\begin{equation}\label{lm4.5-lab2}
(n^{m+1}_h\nabla\Phi, \u^{m+1}_h)_h\le 4 \tilde K\|n^0_h\|_{L^1(\Omega)} (n^{m+1}_h, \log\frac{n^{m+1}_h}{\bar n^0_h})+\frac{1}{2}\|\nabla \u^{m+1}_h\|^2_{\L^2(\Omega)}+4 \tilde K M \|n^0_h\|^2_{L^1(\Omega)}.
\end{equation}
The proof is completed by selecting $K=4 \tilde K \max\{1, M\}$ and combining \eqref{lm4.5-lab1} and \eqref{lm4.5-lab2}.
\end{proof}

Finally we end up with a concluding theorem compiling the a priori estimates resulting from Lemmas \ref{lm:a_priori_KS} and \ref{lm:a_priori_NS}. 
\begin{theorem} Let $n_h^0\in N_h$ be such that 
$$
\int_\Omega n^0_h(\x)\,\dx<2\chi_\Omega.
$$
Then there exist three constants $D=D(n^0_h)>0$, $F=F(n^0_h, c^0_h)>0$ and $K>0$ such that the sequence of discrete solutions $\{(n^m_h, c^m_h, u^m_h)\}_{m=0}^M$ satisfies 
\begin{equation}\label{ineq:energy_nh_ch_II}
\begin{array}{rcl}
k D(n^0_h) &\displaystyle\sum_{m=0}^{M-1} & \Big(\|\nabla i_h(\log (n^{m+1}_h+1))\|^2_{\L^2(\Omega)} +\|\nabla c^{m+1}_h\|_{\L^2(\Omega)}^2
\\
&&\displaystyle+(n^{m+1}_h, \log\frac{n^{m+1}_h}{\bar n^0_h} )_h \Big)\le T F(n^0_h, c^0_h) + \|n^0_h\|_{L^1(\Omega)}
\end{array}
\end{equation}
and 
\begin{equation}\label{ineq:energy_uh_II}
\begin{array}{rcl}
\displaystyle
\|\u^{\ell}_h\|^2_{\L^2(\Omega)}+ k \sum_{m=0}^{\ell-1}\|\nabla\u^{m+1}_h\|^2_{\L^2(\Omega)}&\le&\displaystyle \|\u^0_h\|^2_{\L^2(\Omega)} + K \frac{F(n^0_h, c^0_h)}{D(n^0_h)} T \|n^0_h\|_{L^1(\Omega)} 
\\
&&\displaystyle+  K \Big(T+\frac{F(n^0_h, c^0_h)}{D(n^0_h)}\Big) \|n^0_h\|^2_{L^1(\Omega)},
\end{array}
\end{equation}
for all $\ell=1, \cdots, M$.
\end{theorem}
\begin{proof} First of all, we note that  $(\phi_h, \log\frac{\phi_h}{\bar\phi_h})_h\ge 0$ holds on using that $f(x)=x \log (x)$ is a convex function and Jensen's inequality. Then summing up \eqref{ineq:energy_nh_ch} over $m$ follows \eqref{ineq:energy_nh_ch_II}  on noting $0\le(\log(n^{M}_h+1)_h, 1)_h\le \|n^M_h\|_{L^1(\Omega)}=\|n^0_h\|_{L^1(\Omega)}$. Secondly, summing up \eqref{ineq:energy_uh} over $m$ leads to \eqref{ineq:energy_uh_II} on using \eqref{ineq:energy_nh_ch_II}.     
\end{proof}

\section{Numerical experiments}
Following the existence theory  \cite{Nagai_Senba_Yoshida_1997} of the Keller--Segel equations \eqref{KS}-\eqref{IC_KS},  Morse--Trudinger's inequality \eqref{ineq:Trudinger-Moser_New} -- which is adapted to two-dimensional domains with polygonal boundary -- is the key tool in establishing the critical value $4\chi_\Omega$  for the integrability of $n_0$, i.e. $\|n_0\|_{L^1(\Omega)}< 4\chi_\Omega$, as an existence threshold, where $ \chi_\Omega =\frac{4}{\beta_\Omega^2}\alpha_\Omega$ with $\beta_\Omega$ depending upon the geometrical shape of domains. This critical value is apparently reduced to $2\chi_\Omega $ with use of \eqref{functional_ineq_II} and \eqref{functional_ineq_III}, i.e. $\|n_0\|_{L^1(\Omega)}<  2\chi_\Omega$, for the Navier--Stokes--Keller--Segel equations \eqref{KSNS}-\eqref{IC_KSNS}.  Then in this section we attempt to shed light on the outstanding problem of diagnosing whether, on the one hand, the bound $2\chi_\Omega$ for $\int_\Omega n_0(\x)\, \dx$ found in \cite{Winkler_2020} is critical for the global-in-time existence of generalised solutions to problem \eqref{KSNS}-\eqref{IC_KSNS} or, on the other hand, one can expect that problem \eqref{KSNS}-\eqref{IC_KSNS} inherits the same mass critical value $4\chi_\Omega$ from problem \eqref{KS}-\eqref{IC_KS}.  Clearly we shall need a good deal more than  mere numerical evidence if we are to make rigorous claims for chemotaxis-fluid interaction, but it might pave the way for analysts to focus their effort on proving the optimal threshold for distinguishing between  bounded and unbounded solutions.

With an eye to demonstrating whether or not  $2\chi_\Omega$ is optimal, we consider an example quite intensely studied in the numerical context of the Keller--Segel equations and hence it can be served as a benchmark \cite{Strehl_Sokolov_Kuzmin_Turek_2010, GS_RG_2021,  Badia_Bonilla_GS_2022} for comparing  chemotaxis with and without fluid interaction. As the domain we take a unit ball, namely $\Omega = B((0,0.1); 1)$, and as the initial conditions we use    
$$
n_0(x,y)= \eta_0 e^{-100(x^2+y^2)}\quad\mbox{ and }\quad c_0(x,y)=0
$$ 
and
$$
\u_0(x,y)=\boldsymbol{0}.
$$
With regard to the potential $\Phi$, it is selected as being  a gravitational-like one, i.e., 
$$
\Phi(x,y)= - \Phi_0 y.
$$

The numerical setting for Algorithm \eqref{eq:n_h}-\eqref{eq:p_h} is a triangulation $\mathcal{T}_h$ of $B((0,0.1); 1)$ as despited in Figure \ref{meshes} (right) on which we take two first-order finite element spaces $N_h$ and $C_h$ for approximating the chemoattractant and organism densities, respectively,  and a Taylor--Hood finite element space pair $(\boldsymbol{U}_h, P_h)$ for approximating velocity and pressure. When needed, we will make use of a finer triangulation $\mathcal{T}_h^*$ as shown in Figure \ref{meshes} (left) to emphasise the growth of possible singularities, since it is known \cite{Badia_Bonilla_GS_2022} that such a growth is controlled by the $L^1(\Omega)$-norm and the measure of the macroelement on which the singularity evolves, i.e., the maximum reached by the singularity is proportional to the $L^1(\Omega)$-norm and inversely proportional to the macroelement measure.  Hence the macroelement where the singularity is supported is required to be as small as possible if the size of the $L^1(\Omega)$-norm is not large to boost its enlargement. It should be further noted that when $\Omega$ is a ball whose boundary is approximated by a polygonal, we have that $\alpha_\Omega= \pi - \epsilon$ for certain $\epsilon >0$ being very small depending on $\mathcal{T}_h$, since $\beta_\Omega=1$ due to the convexity of $\Omega$.  The parameter $\varepsilon$ is taken to be $10^{-6}$. Different time steps are considered ranging from $10^{-2}$ through $10^{-5}$. The latter is meant to be used for large values of $\eta_0$ and $\Phi_0$. As a reference to see how evolves the organism density over time, Figure \ref{Snapshots_n0h} illustrates the initial distribution of $n_{0h}$. 
\begin{figure}[b]
    \centering
    \includegraphics[width=0.3\textwidth]{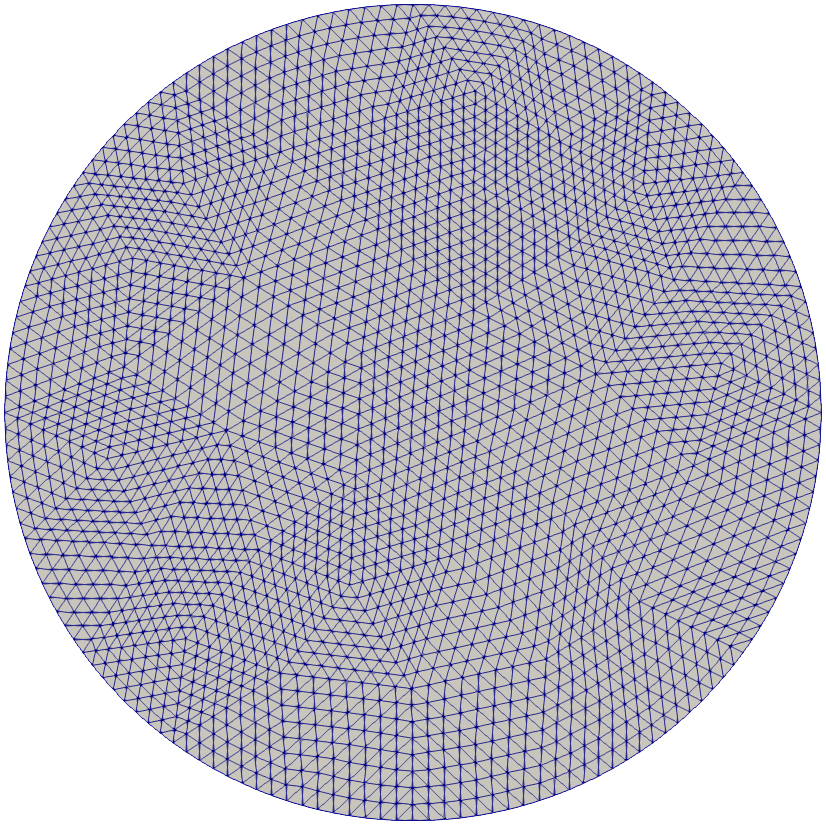}
    \includegraphics[width=0.3\textwidth]{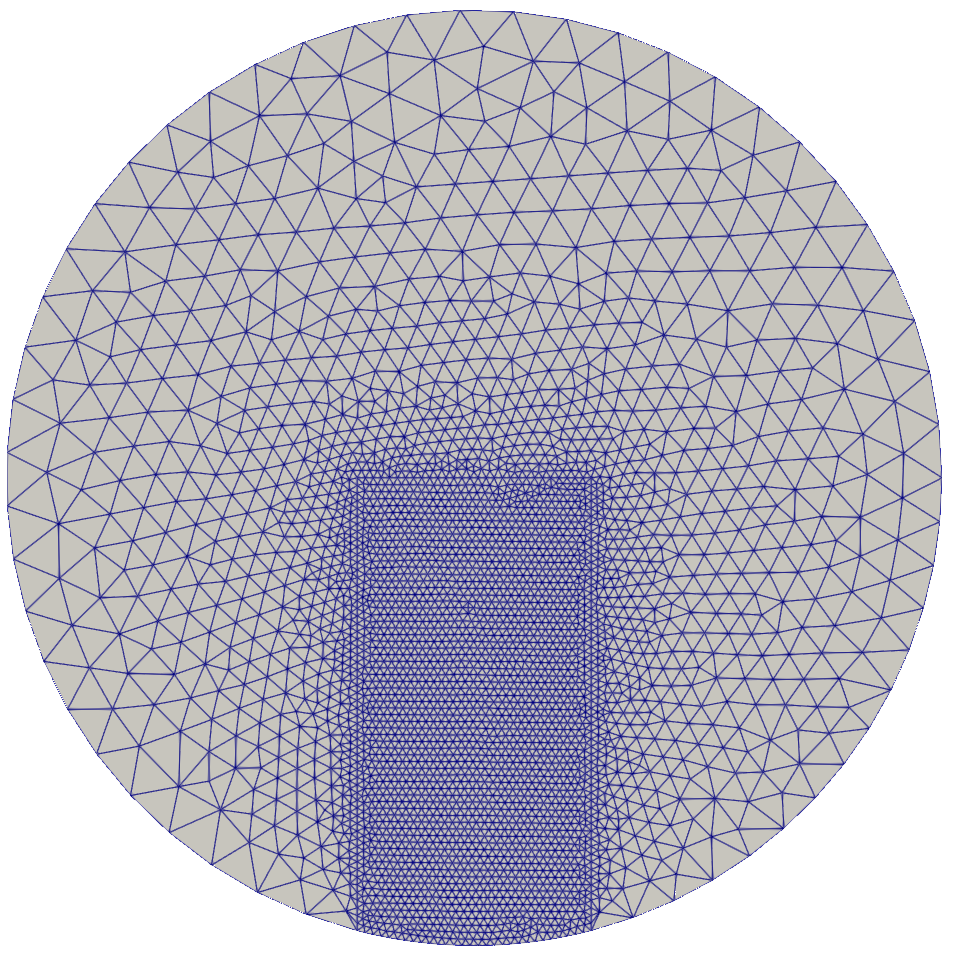}
    \caption{Triangulation $\mathcal{T}_h$ of $\Omega$ (right) and triangulation $\mathcal{T}_h^*$ refined on the singularity path.}
    \label{meshes}
\end{figure} 
\begin{figure}[b]
        \centering
        \includegraphics[width=0.4\textwidth]{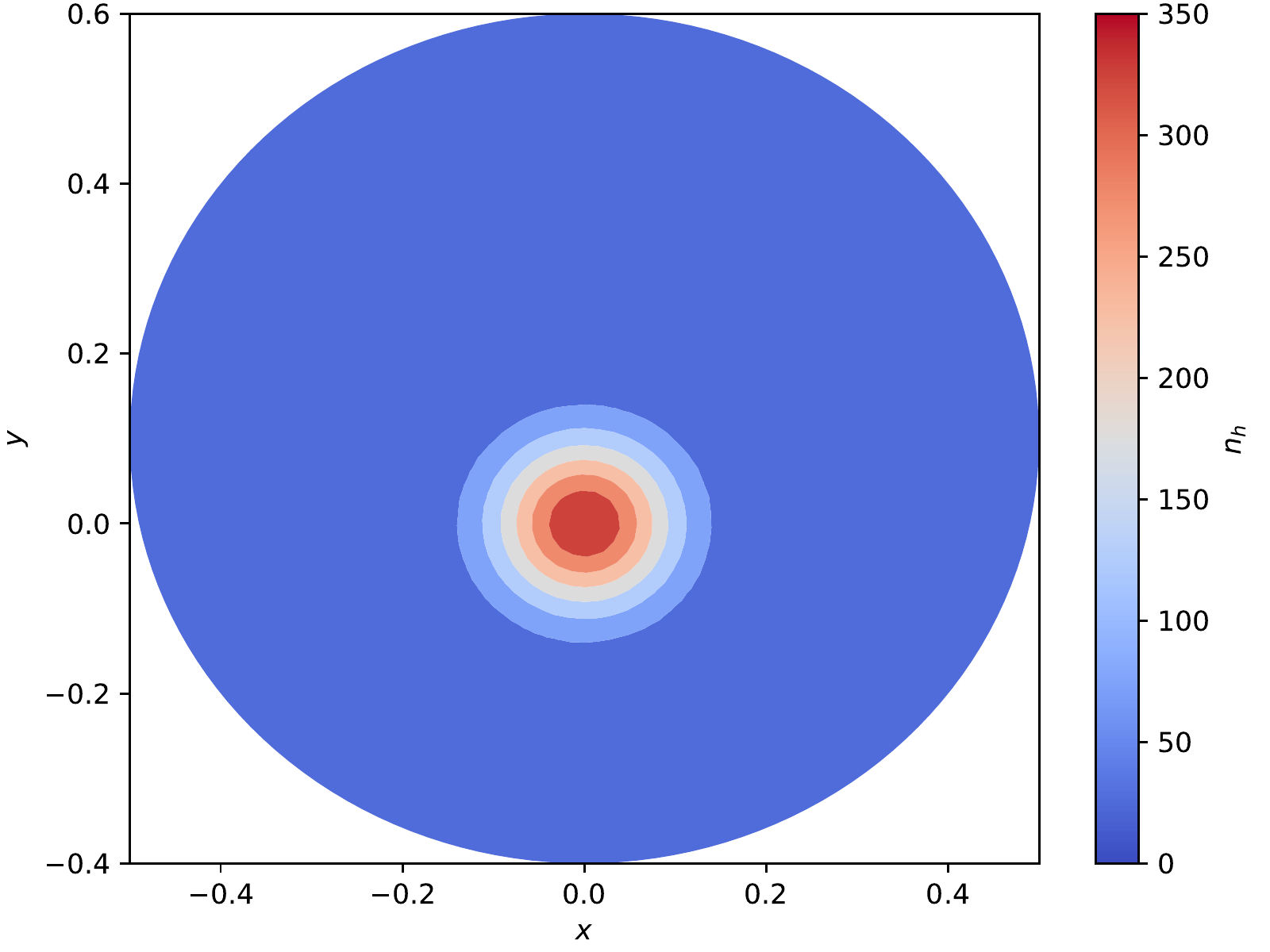}
        \caption{Initial distribution of the organism density $n_{0h}$.}\label{Snapshots_n0h}
\end{figure}

In the above framework, the numerical solution (if $\eta_0=1000$) provided by an algorithm \cite{Badia_Bonilla_GS_2022} similar to that developed in the paper but for approximating the Keller--Segel equations \eqref{KS}-\eqref{IC_KS} is that of the highest density concentration of both the chemoattractant and organisms moving from the origin to the point $(0, -0.9)$, where a Dirac-type singularity takes place. Accordingly, we may expect that as the fluid transports the singularity it may enhance its growth; thereby resulting in a reduction of the existence threshold below $4\chi_\Omega$.      
 
In order to see if a fluid can essentially modify the existence threshold for the fluid-free chemotactic behaviour, we will use the parameters $\eta_0$ and $\Phi_0$ to control the size of the  $L^1(\Omega)$-norm of $n_0$ and the size of the $W^{1,\infty}(\Omega)$-norm of  $\Phi$, respectively, in a set of numerical experiments. 

Finally, as a linearisation of Algorithm \eqref{eq:n_h}-\eqref{eq:p_h}, a Picard technique is implemented with a stopping criterium being the relative error between two different iterations for a tolerance of  $10^{-3}$.    

\subsection{Case: \texorpdfstring{$\eta_0=350$}{Lg} and \texorpdfstring{$\Phi_0=10$}{Lg}} To begin with, we take $\Phi_0=100$ and $\eta_0=350$, which gives the initial chemoattractant mass $\|n_0\|_{L^1(\Omega)}\approx11.0027<4\pi$. This amount is preserved for the $L^1(\Omega)$-norm of  $\{n_h^m\}_m$, whereas that of  $\{c_h^m\}_m$ increases; see Figure \ref{Graphs_350} (left). Observe in Figure \ref{Graphs_350} (middle) that maxima for $\{n_h^m\}_m$ drop rapidly in the very beginning and then start rising until becoming nearly constantly $100$ and for $\{c_h^m\}_m$ grow gradually up to over $14$. In contract to maxima for $\{n_h^m\}_m$, minima in Figure \ref{Graphs_350} (right) move inversely towards approximately $3$ and for $\{c_h^m\}_m$ increase reaching slightly over $10$. 

Snapshots at times $t=0.02$, $0.04$,  $1.0$ and $10.0$ of $\{n_h^m\}_m$, $\{c_h^m\}_m$ and $\{\u_h^m\}_m$ are depicted in Figures \ref{Snapshots_350_nh}, \ref{Snapshots_350_ch} and \ref{Snapshots_350_uh}. It is noticeable that the highest density area of both the chemoattractant and organisms moves from $(0,0.1)$ to $(0, -0.9)$;  in particular,  that of organisms occupies several macroelements around $(0, -0.9)$. The fluid flow in the beginning generates two vortices as a result of the structure of $n_{0h}$, which vanish as the high density area touches the boundary. 
\begin{figure}
    \begin{subfigure}[b]{0.23\textwidth}
        \centering
        \includegraphics[width=1.0\textwidth]{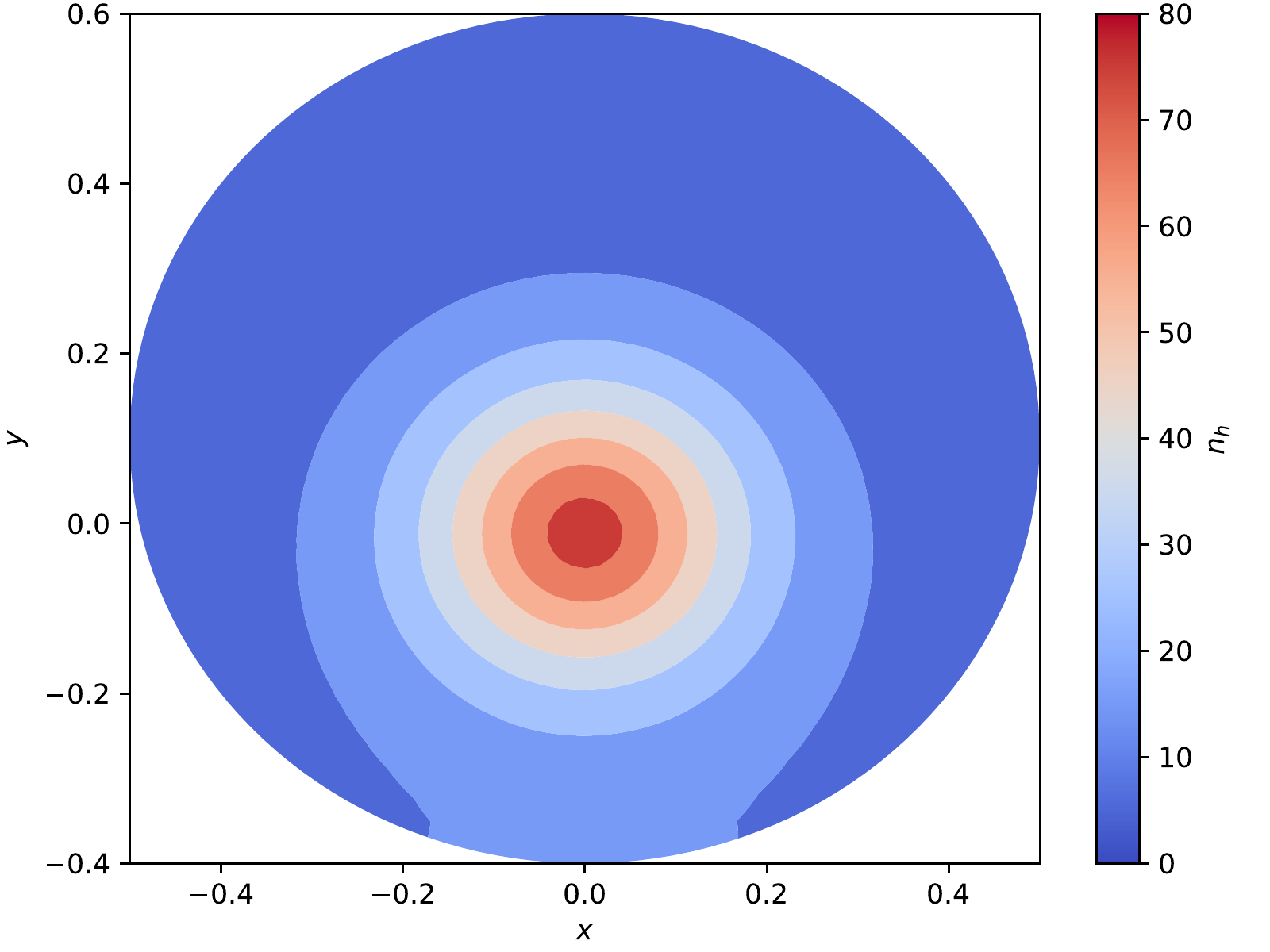}
    \end{subfigure}
    \begin{subfigure}[b]{0.23\textwidth}
        \centering
        \includegraphics[width=1.0\textwidth]{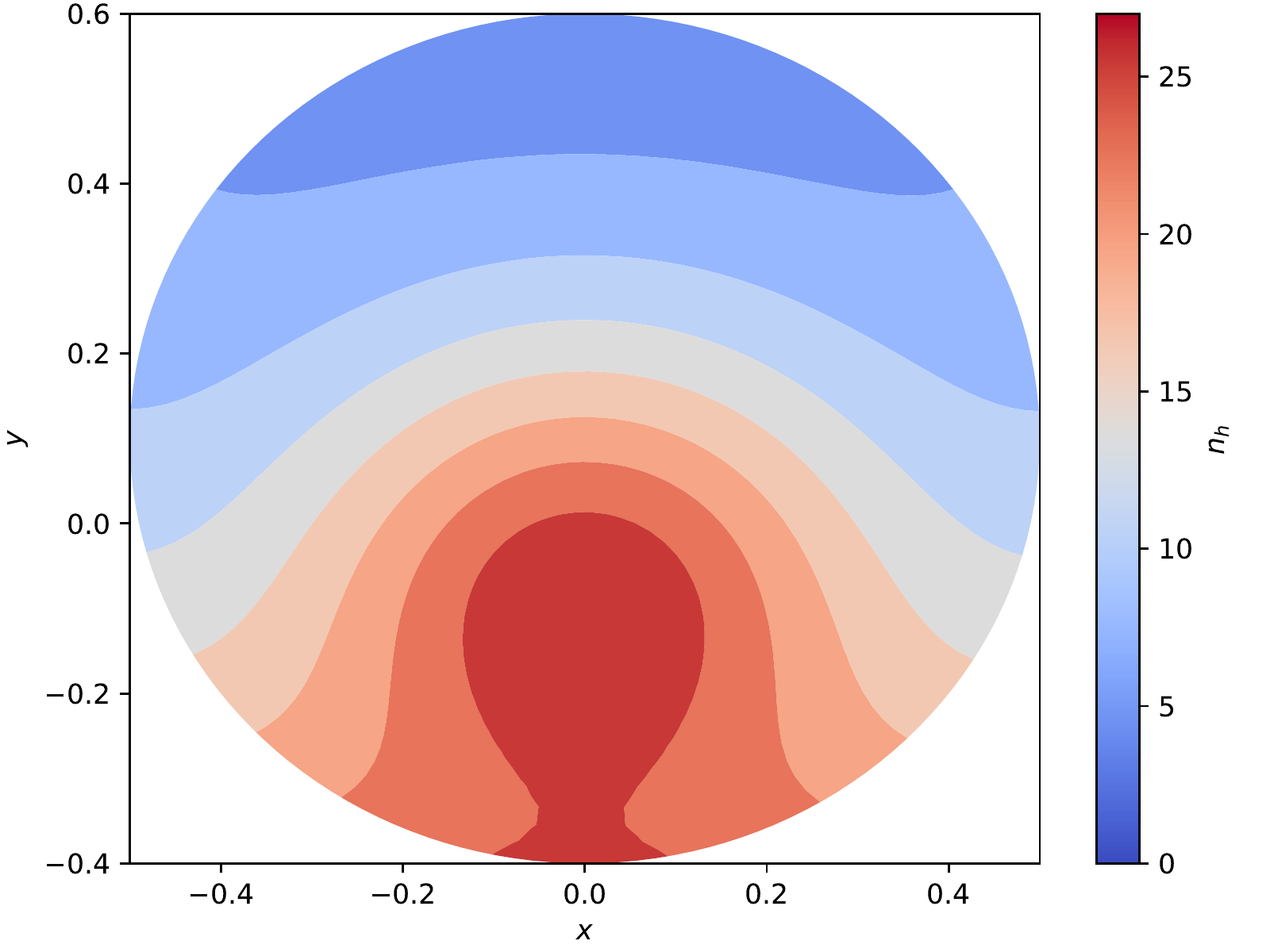}
    \end{subfigure}
    \begin{subfigure}[b]{0.23\textwidth}
        \centering
        \includegraphics[width=1.0\textwidth]{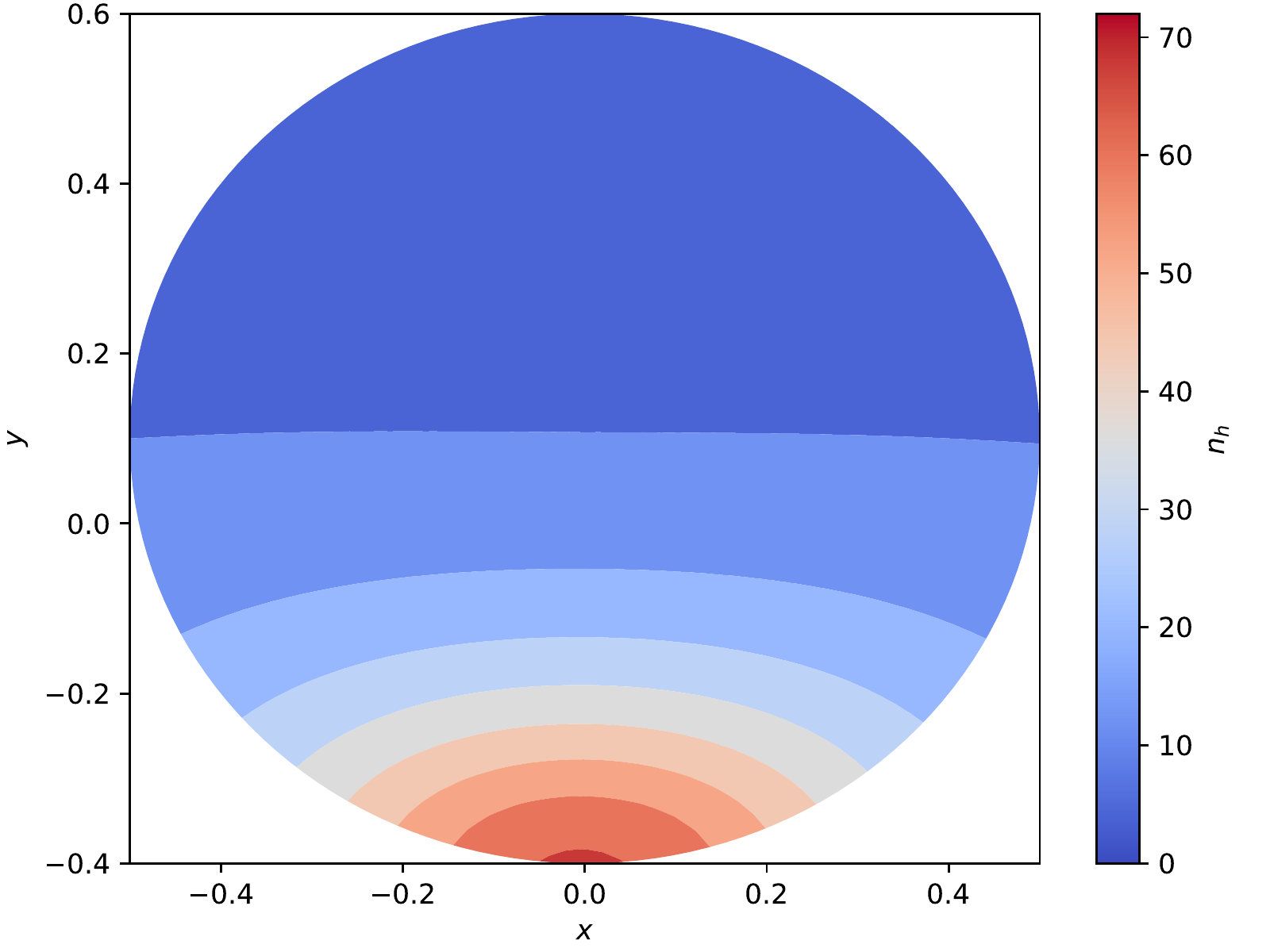}
    \end{subfigure}
    \begin{subfigure}[b]{0.23\textwidth}
        \centering
        \includegraphics[width=1.0\textwidth]{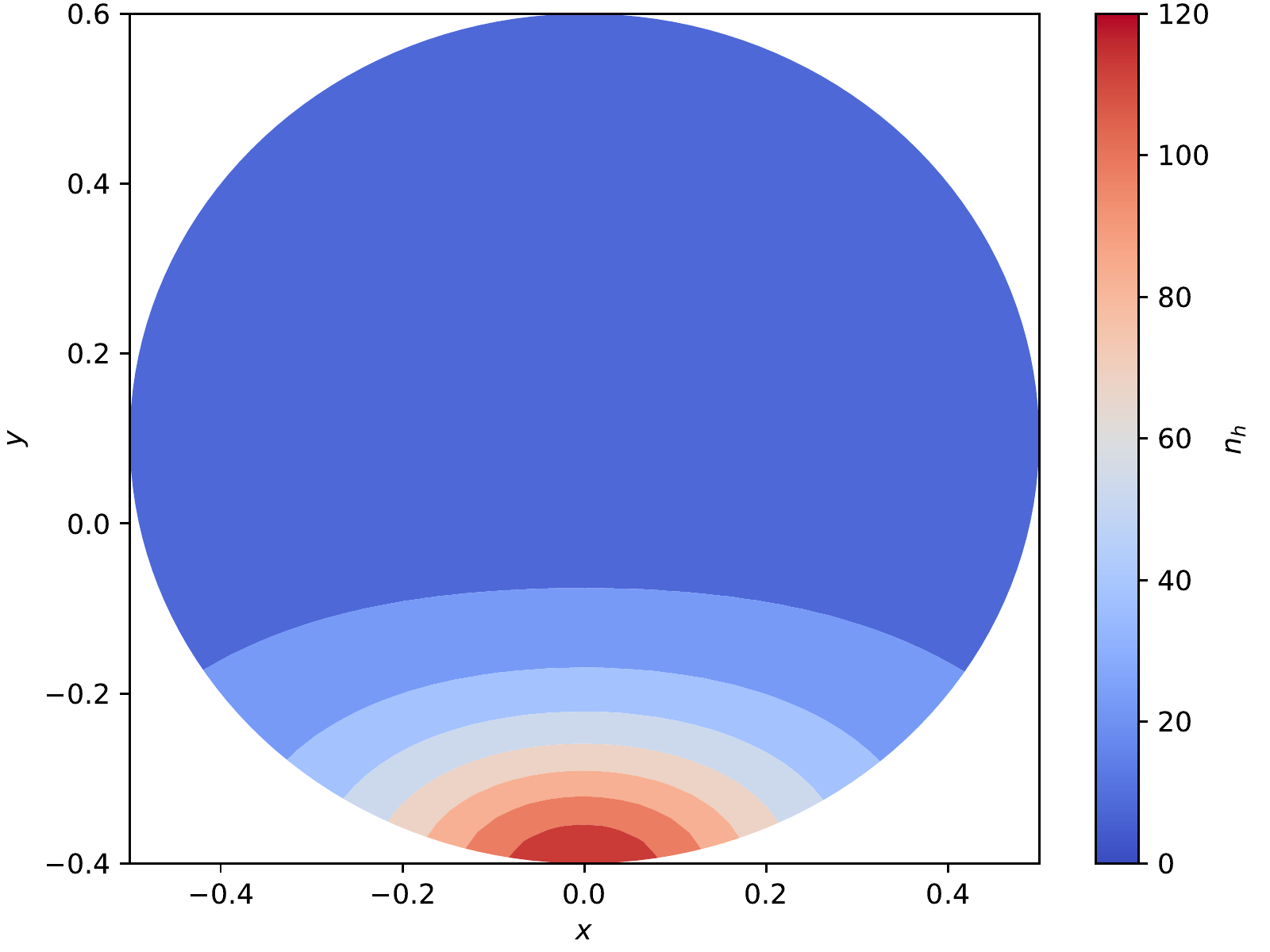}
    \end{subfigure}
            \caption{Snapshots at $t=0.02$, $t=0.04$, $t=1$,  and $10$ of $n_h$ for $\eta_0=350$}\label{Snapshots_350_nh}
            \end{figure}
    \begin{figure}
    \begin{subfigure}[b]{0.23\textwidth}
        \centering
        \includegraphics[width=1.0\textwidth]{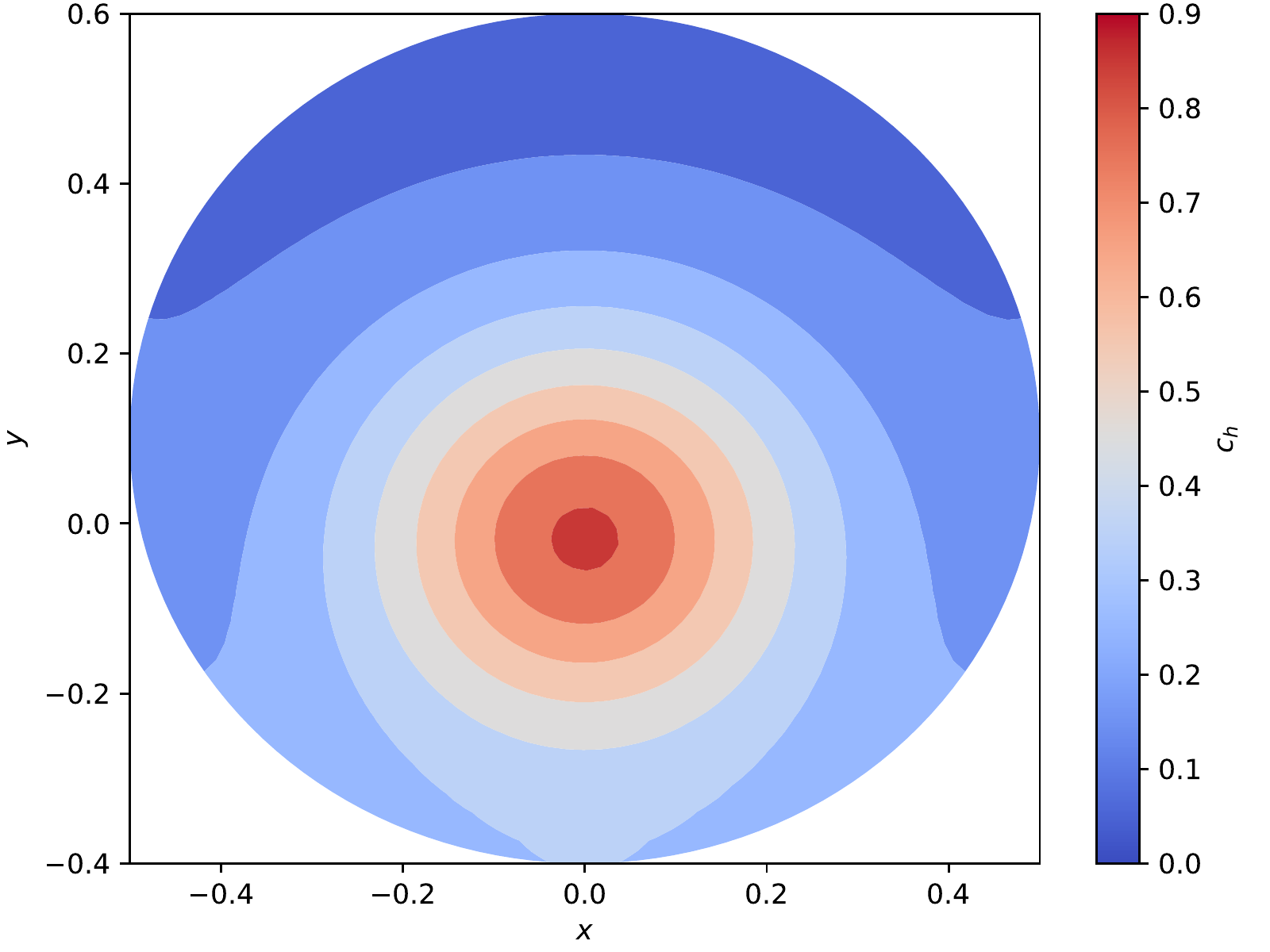}
    \end{subfigure}
    \begin{subfigure}[b]{0.23\textwidth}
        \centering
        \includegraphics[width=1.0\textwidth]{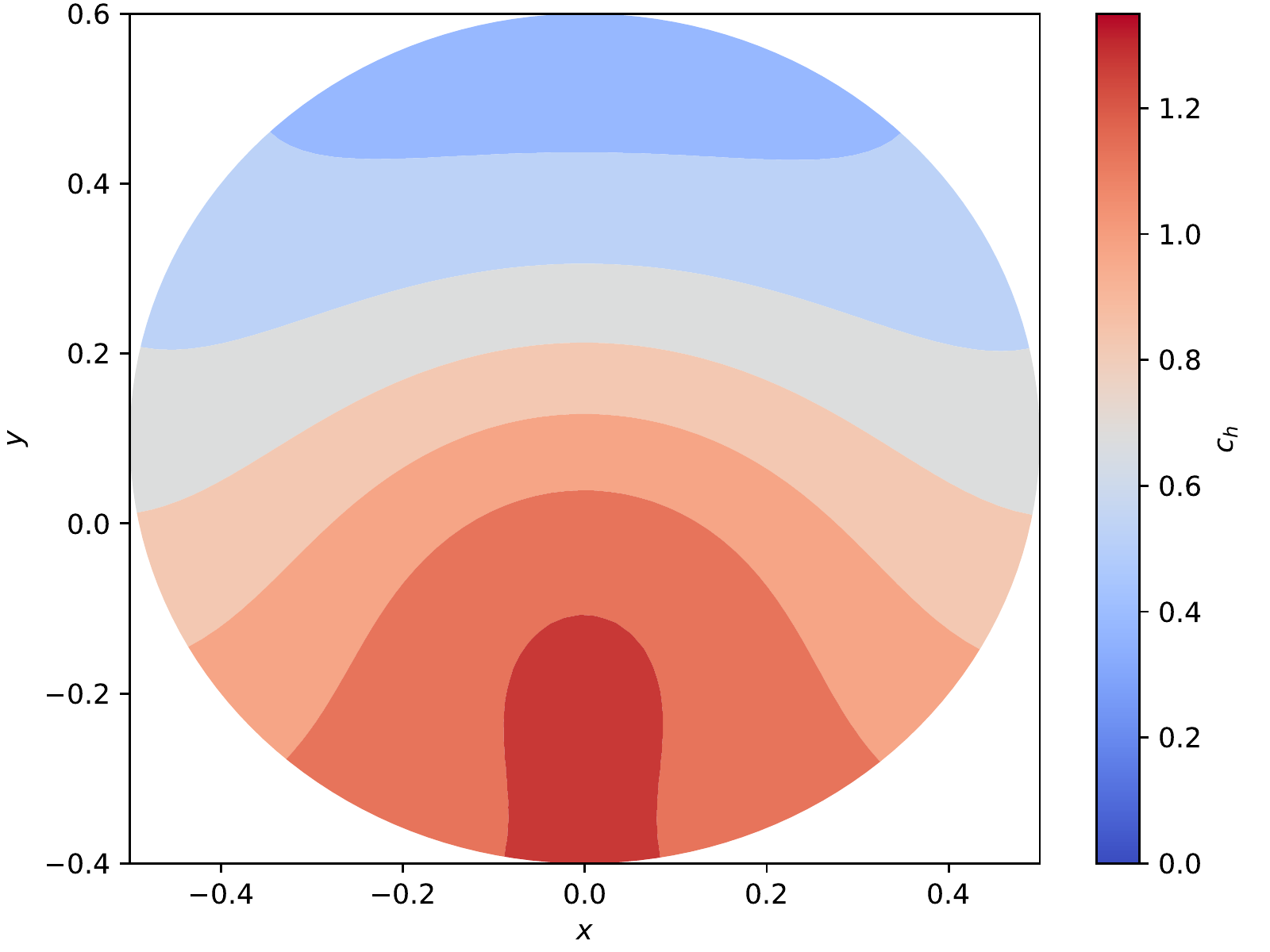}
    \end{subfigure}
    \begin{subfigure}[b]{0.23\textwidth}
        \centering
        \includegraphics[width=1.0\textwidth]{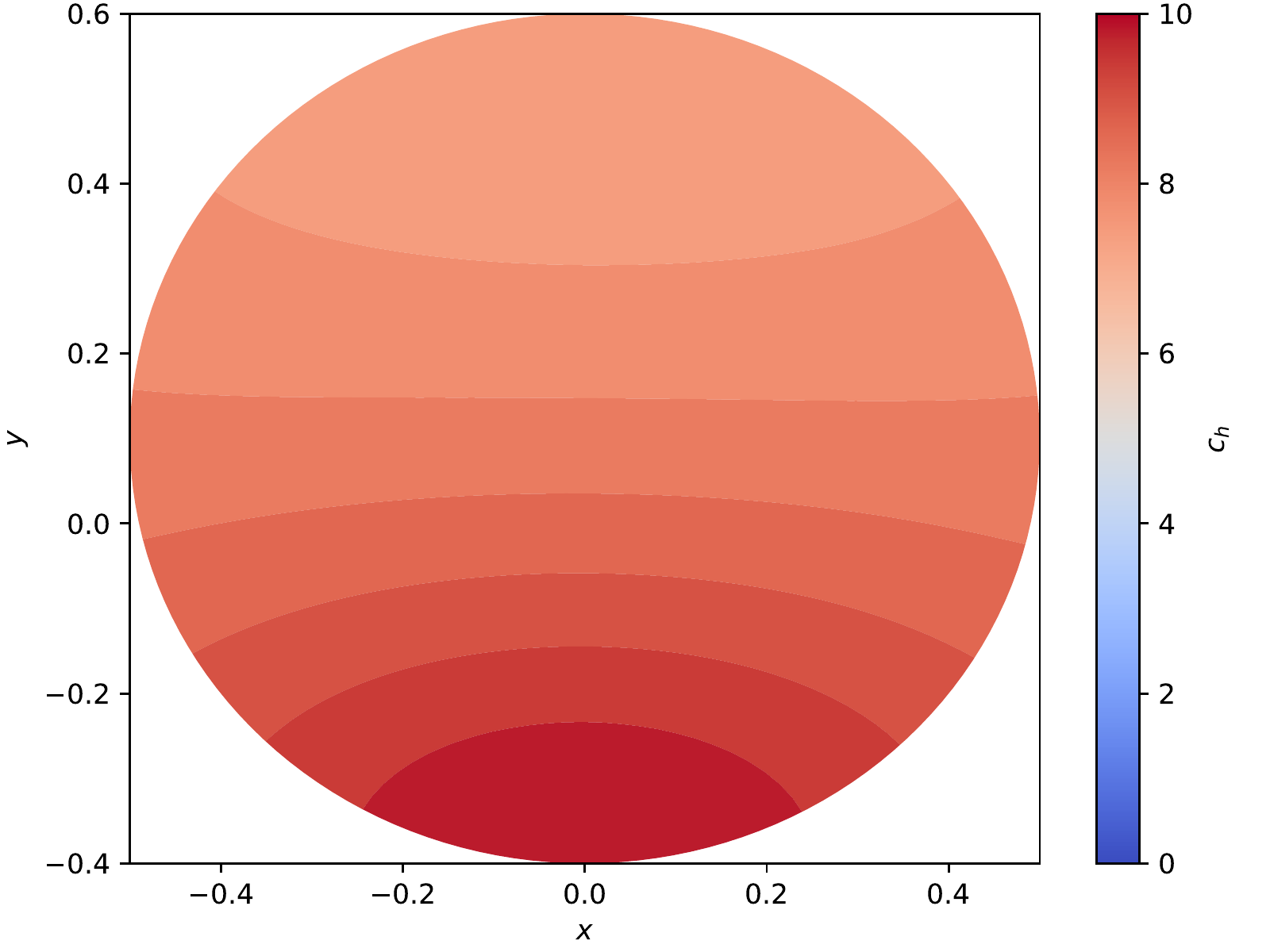}
    \end{subfigure}
    \begin{subfigure}[b]{0.23\textwidth}
        \centering
        \includegraphics[width=1.0\textwidth]{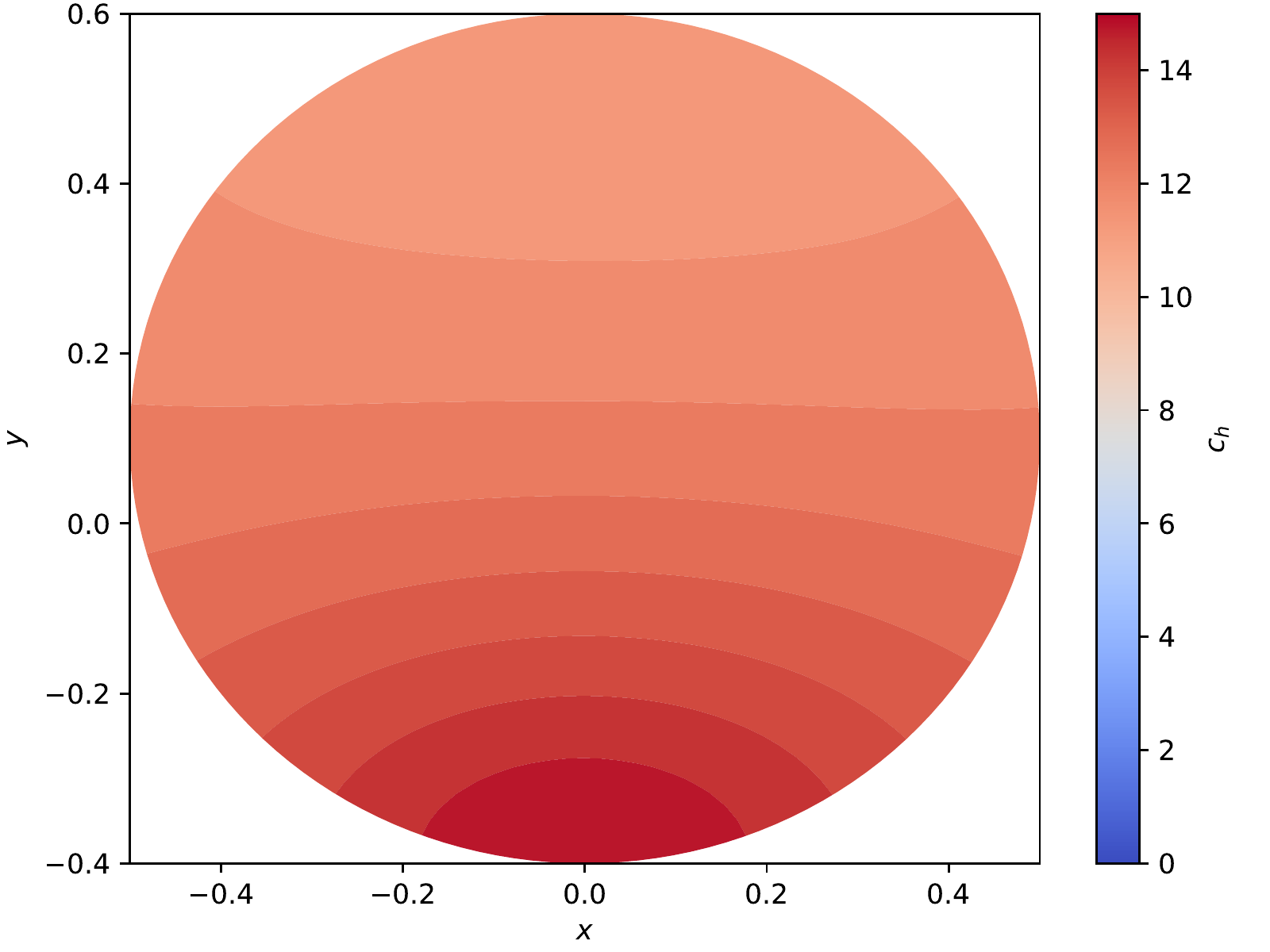}
    \end{subfigure}
    \caption{Snapshots at $t=0.02$, $t=0.04$, $1,$ and $10$ of $c_h$ for $\eta_0=350$}\label{Snapshots_350_ch}
            \end{figure}
    \begin{figure}
    \begin{subfigure}[b]{0.23\textwidth}
        \centering
        \includegraphics[width=1.0\textwidth]{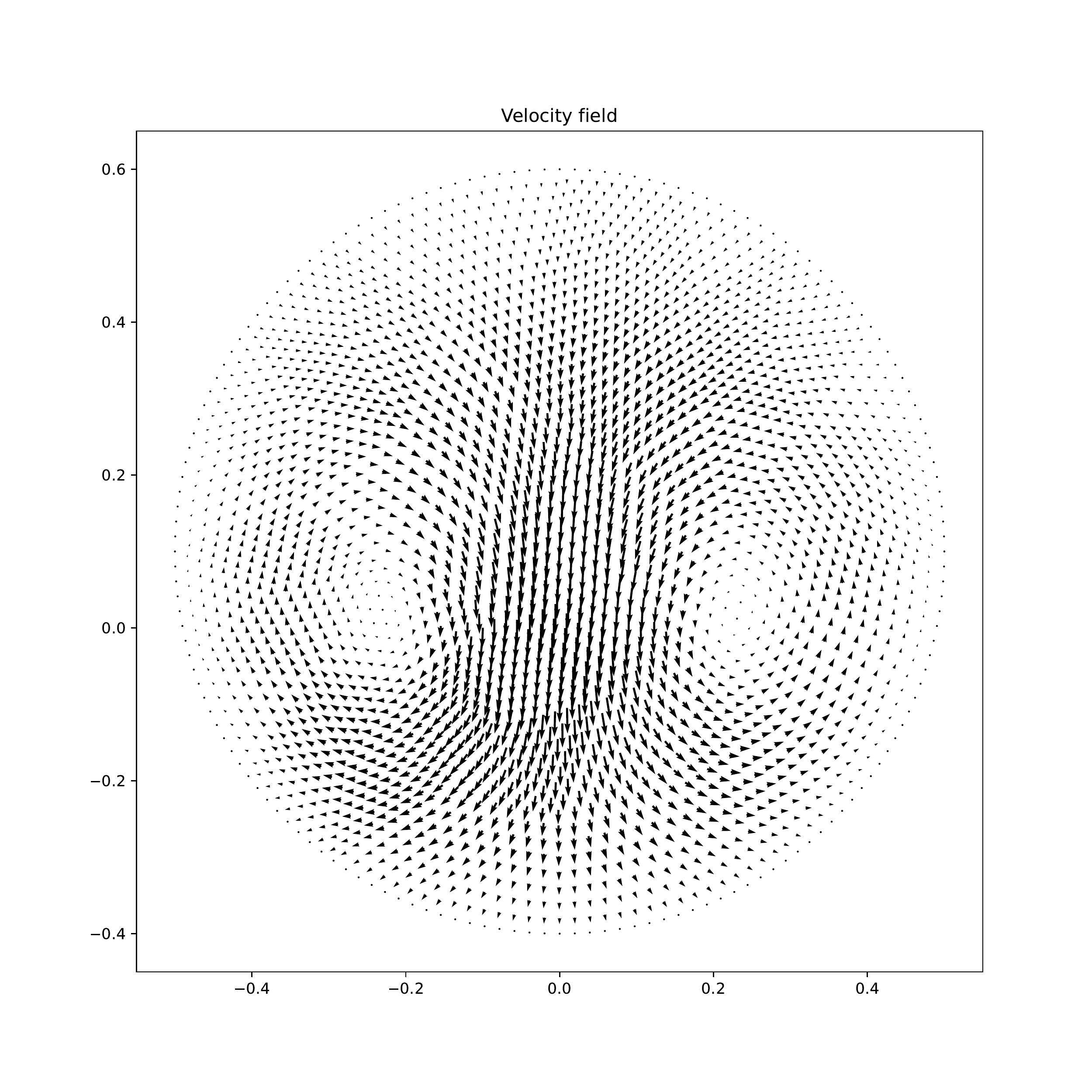}
    \end{subfigure}
    \begin{subfigure}[b]{0.23\textwidth}
        \centering
        \includegraphics[width=1.0\textwidth]{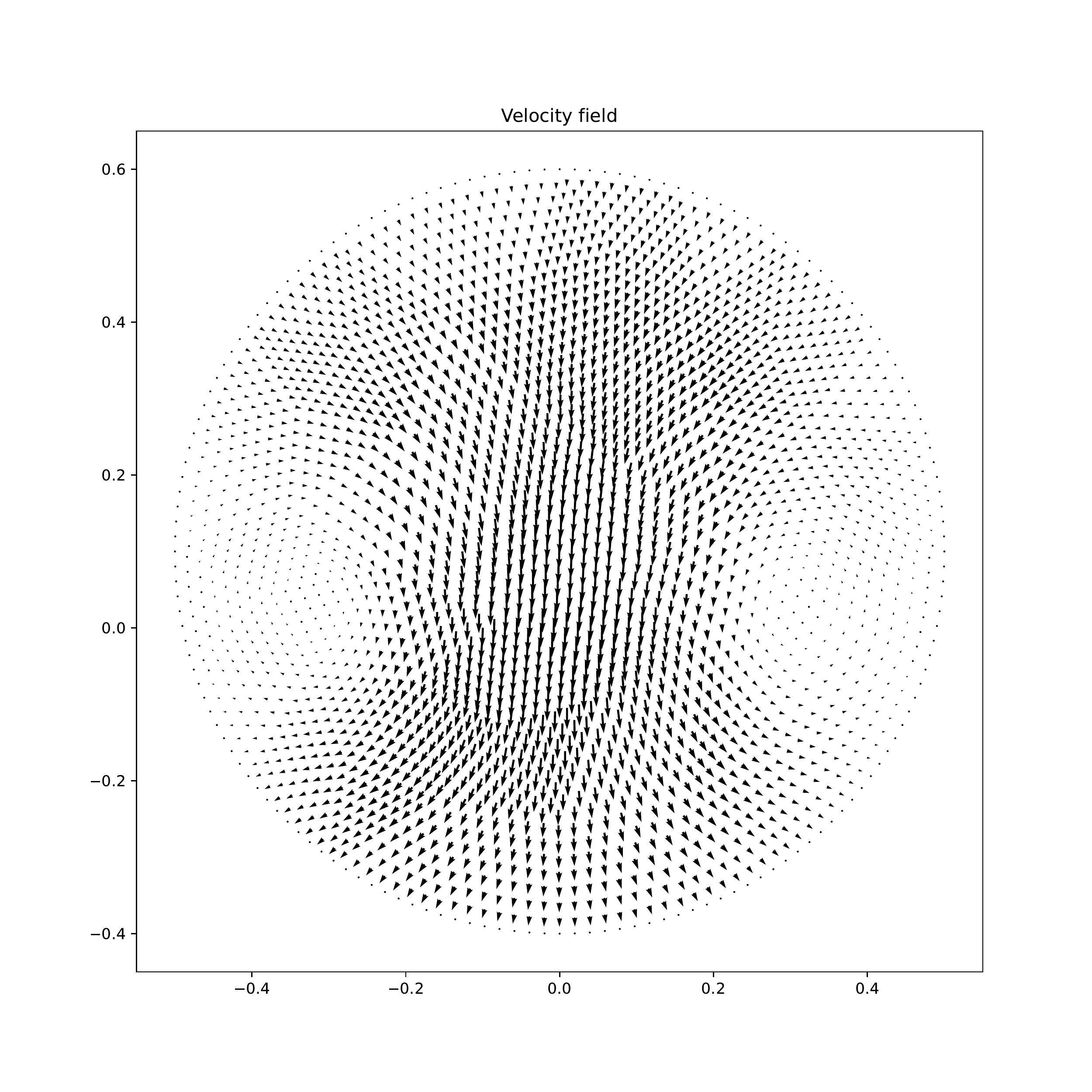}
    \end{subfigure}
    \begin{subfigure}[b]{0.23\textwidth}
        \centering
        \includegraphics[width=1.0\textwidth]{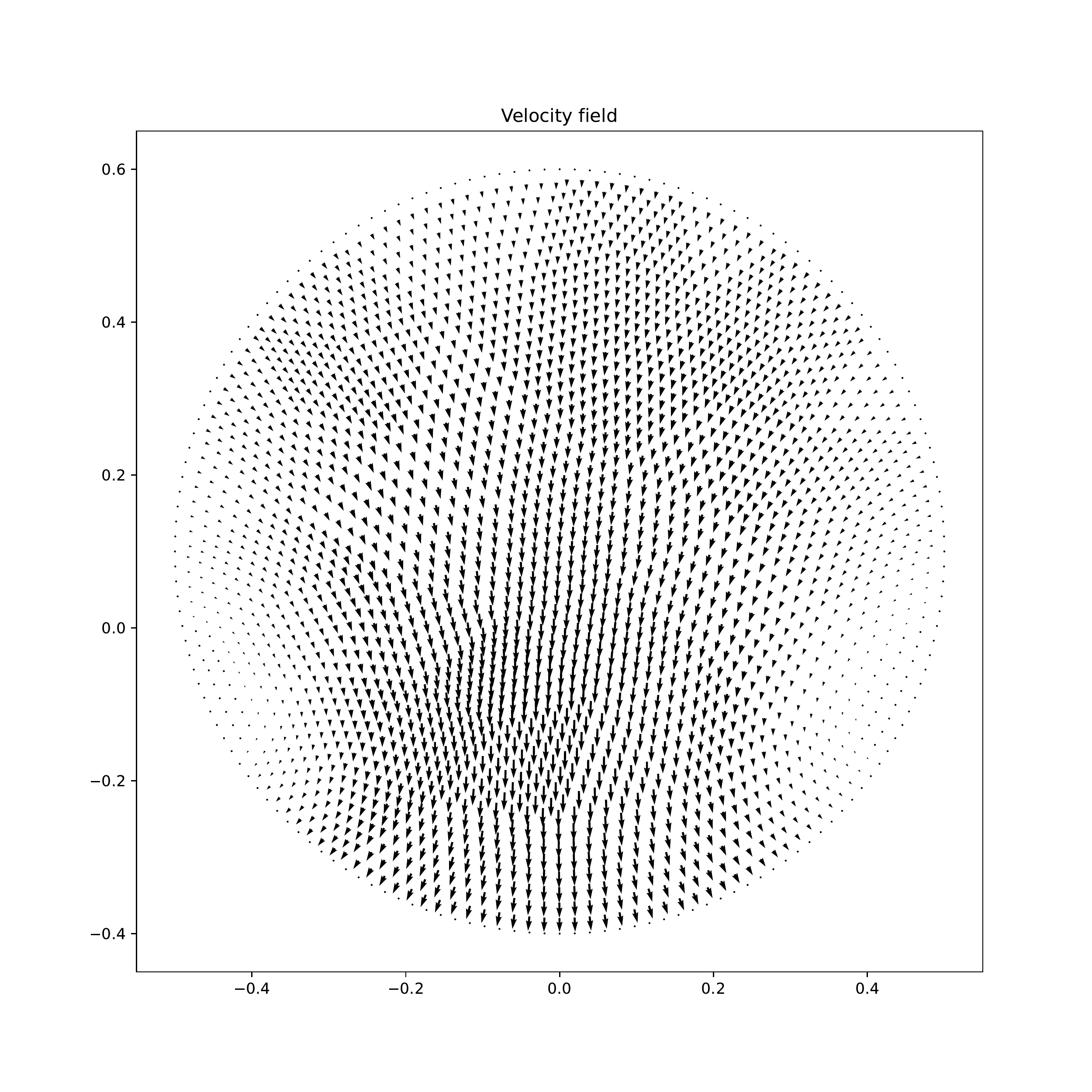}
    \end{subfigure}
    \begin{subfigure}[b]{0.23\textwidth}
        \centering
        \includegraphics[width=1.0\textwidth]{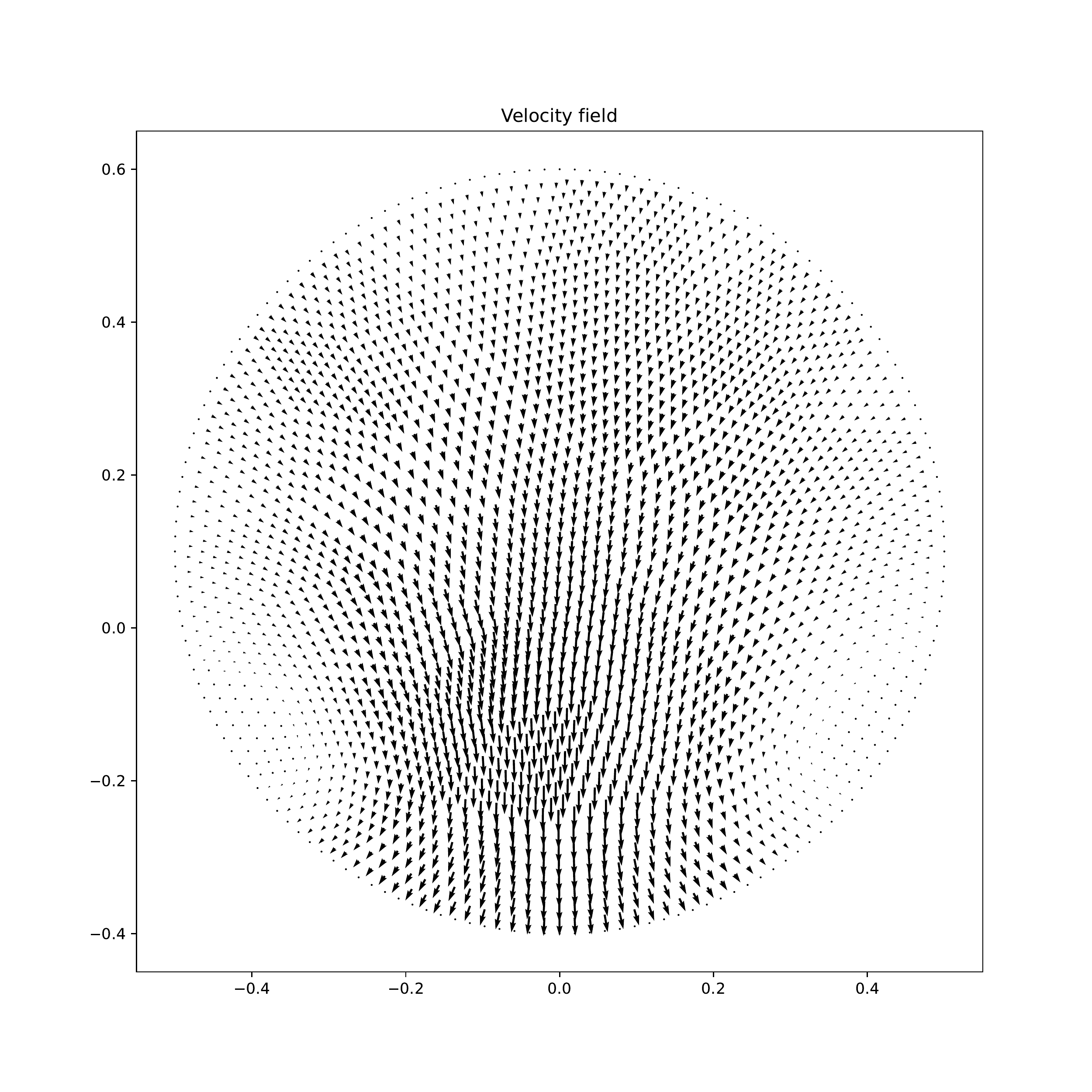}
    \end{subfigure}
    \caption{Snapshots at $t=0.02$, $t=0.04$, $1$, and $10$ of $\u_h$ for $\eta_0=350$}\label{Snapshots_350_uh}
\end{figure}
\begin{figure}
    \begin{subfigure}[b]{0.23\textwidth}
        \centering
        \includegraphics[width=1.0\textwidth]{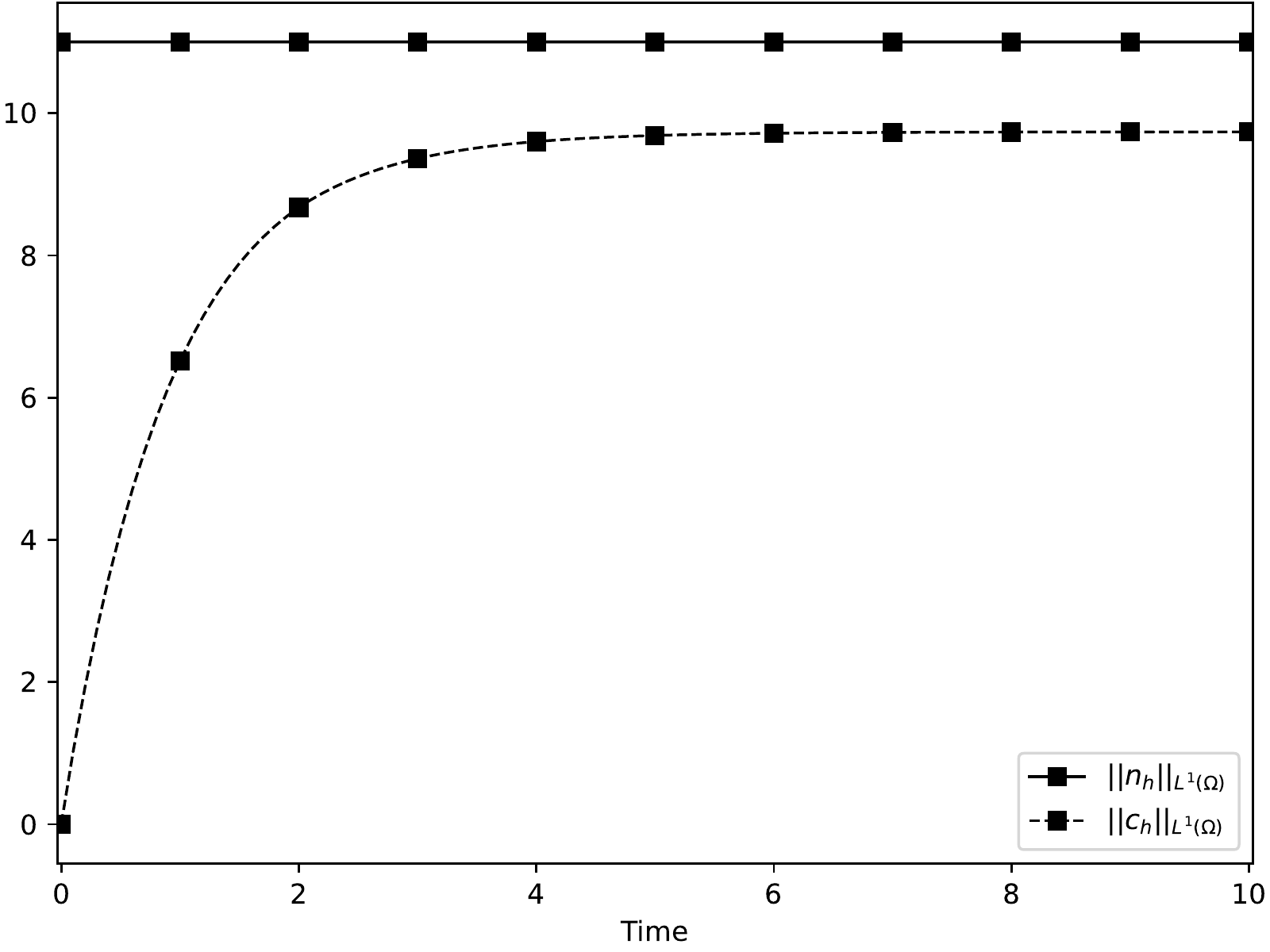}
    \end{subfigure}
    \begin{subfigure}[b]{0.23\textwidth}
        \centering
        \includegraphics[width=1.0\textwidth]{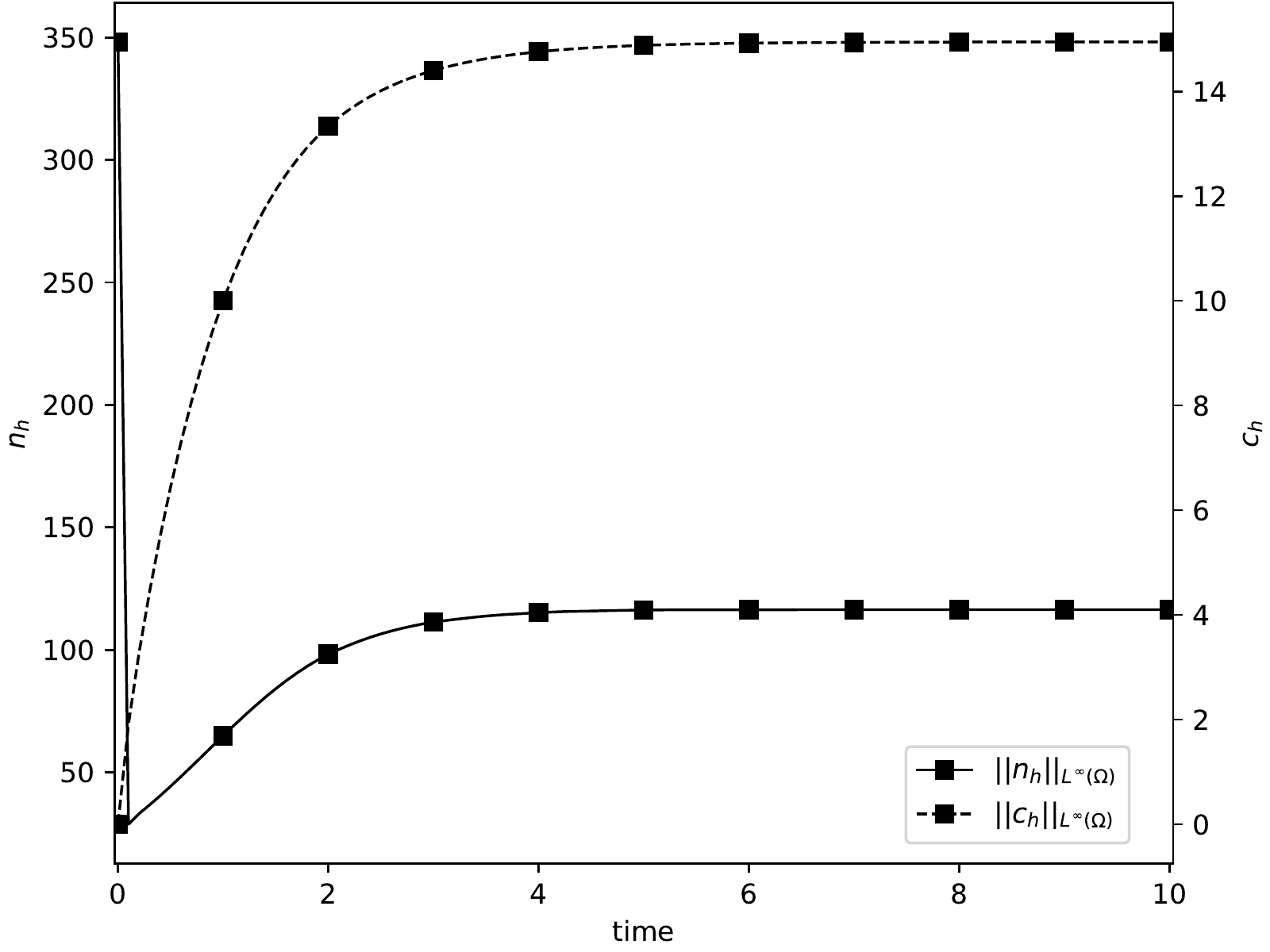}
    \end{subfigure}
    \begin{subfigure}[b]{0.23\textwidth}
        \centering
        \includegraphics[width=1.0\textwidth]{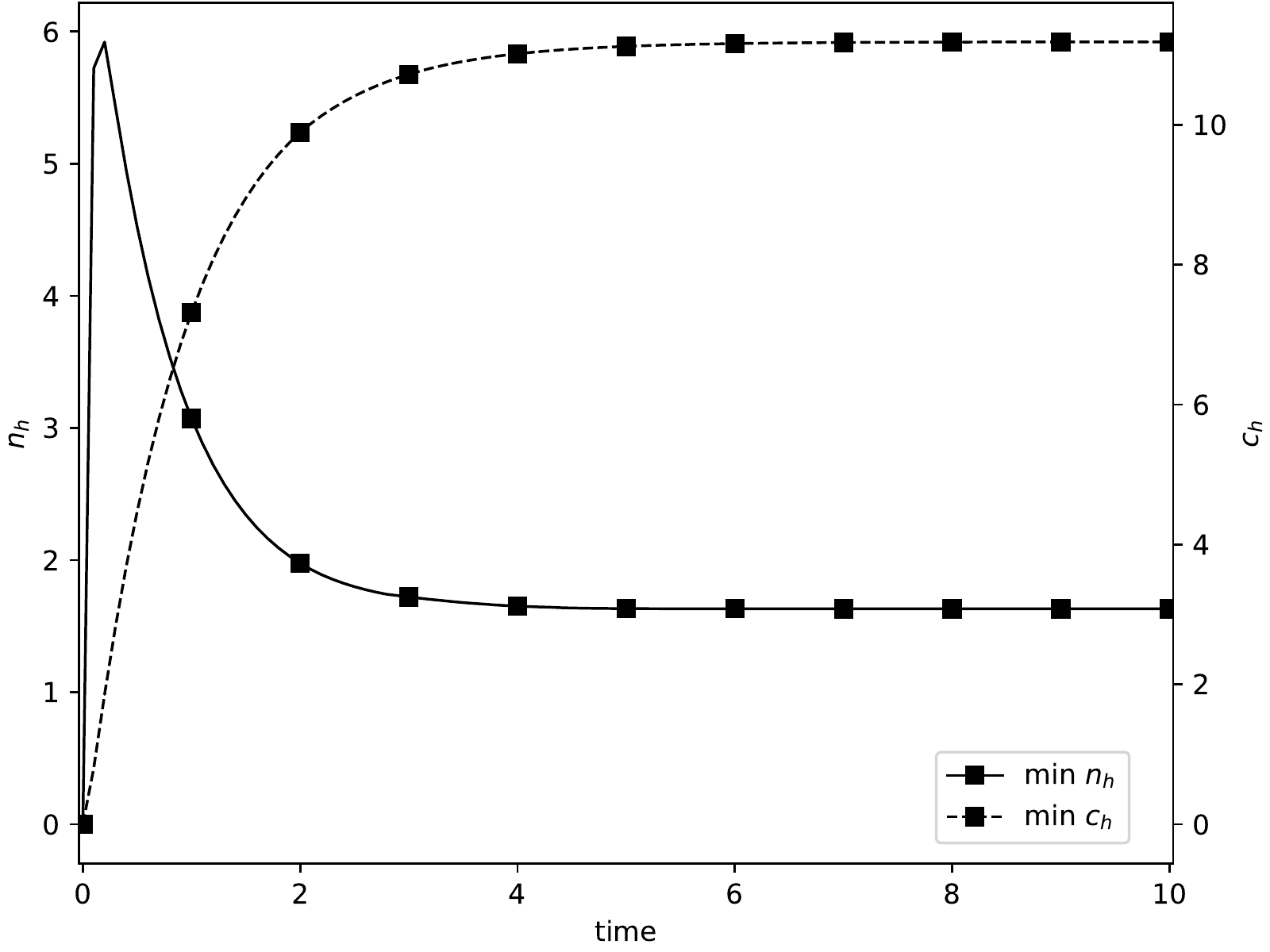}
    \end{subfigure}
            \caption{Plots of $L^1(\Omega)$-norm, maxima, minima and energy for $\eta=350$}\label{Graphs_350}
\end{figure}
\subsection{Case: \texorpdfstring{$\eta_0=400$}{Lg} and \texorpdfstring{$\Phi_0=10$}{Lg}} It will be assumed now that $\Phi_0=10$ and $\eta_0=400$. Accordingly this provides $\|n_0\|_{L^1(\Omega)}\approx12.5745\gtrsim4\pi$, whose value keeps for $\{n_h^m\}_m$ as before and $\{c_h^m\}_m$ gains mass being $\max_m\|c_h^m\|_{L^1(\Omega)}\approx 8.2289$ as shown in Figure \ref{Graphs_400} (left). In Figure \ref{Graphs_400} (middle) it is seen that maxima for $\{n_h^m\}_m$ and $\{c_h^m\}_m$ stabilise at $2115$ and $17.5$, respectively. In the evolution of minima we find a considerably important difference with regard to $\eta=350$, since minima for $\{n_h^m\}_m$ become almost null and those for $\{c_h^m\}_m$ exceed the value of $8$; cf. Figure \ref{Graphs_400} (right). In this case, the number of macroelements supporting the highest density values of $\{n_h^m\}_m$ are smaller than that for $\eta_0=350$ as can see from snapshots at $t=0.02$, $0.04$, $0.5$ and $1.78$ in Figure \ref{Snapshots_400_nh}. The dynamics of $\{c_h^m\}_m$ and $\{\u_h^m\}_m$ is displayed in Figures \ref{Snapshots_400_ch} and \ref{Snapshots_400_uh}, where it is observed that  the chemoattractant density distribution is less uniform than for $\eta=350$ and the velocity field is more active in the region of $(0, - 0.9)$; as a consequence of a higher concentration for $\{n_h^m\}_h$ at such a point. 
\begin{figure}
    \begin{subfigure}[b]{0.23\textwidth}
        \centering
        \includegraphics[width=1.0\textwidth]{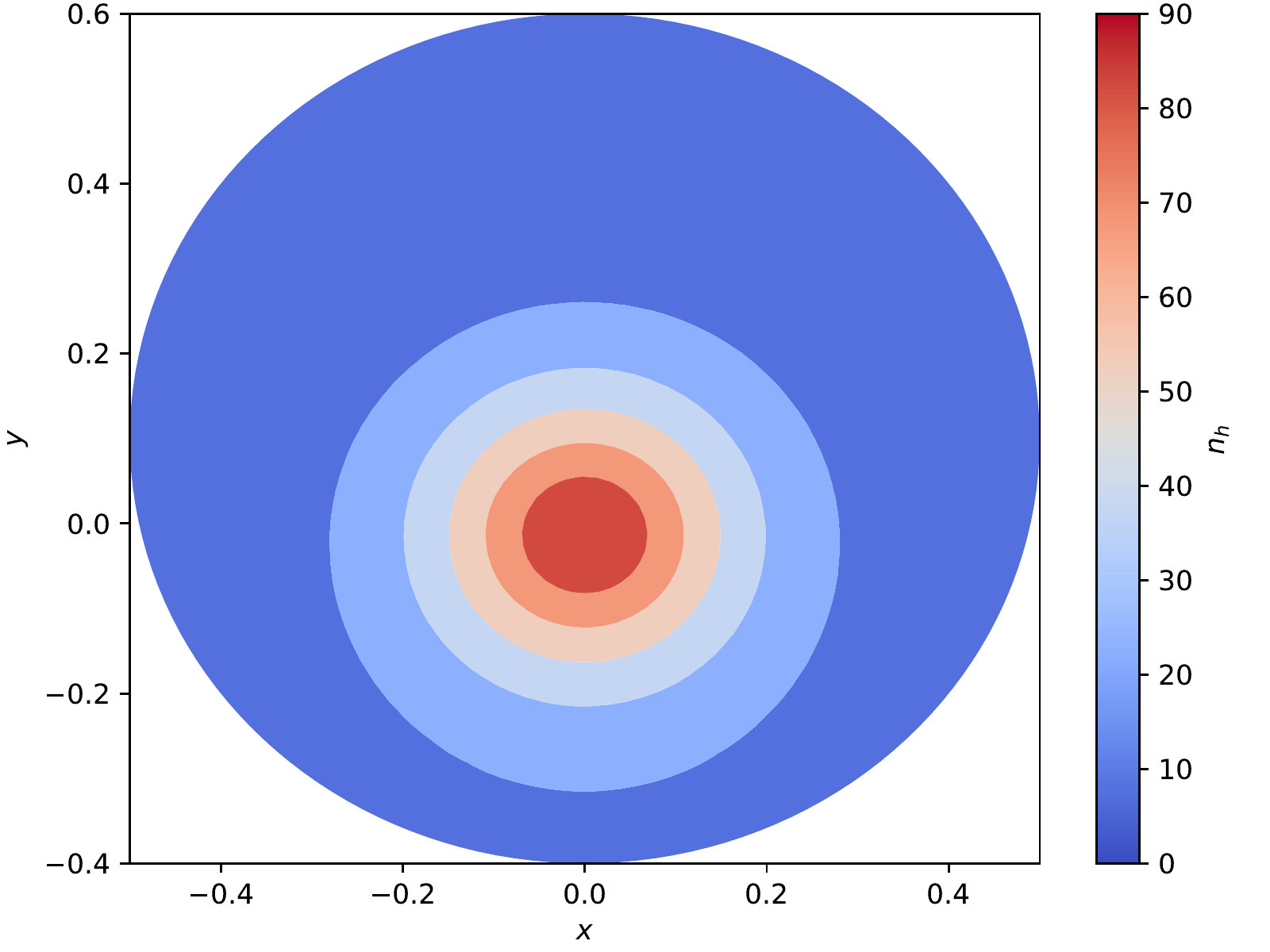}
    \end{subfigure}
    \begin{subfigure}[b]{0.23\textwidth}
        \centering
        \includegraphics[width=1.0\textwidth]{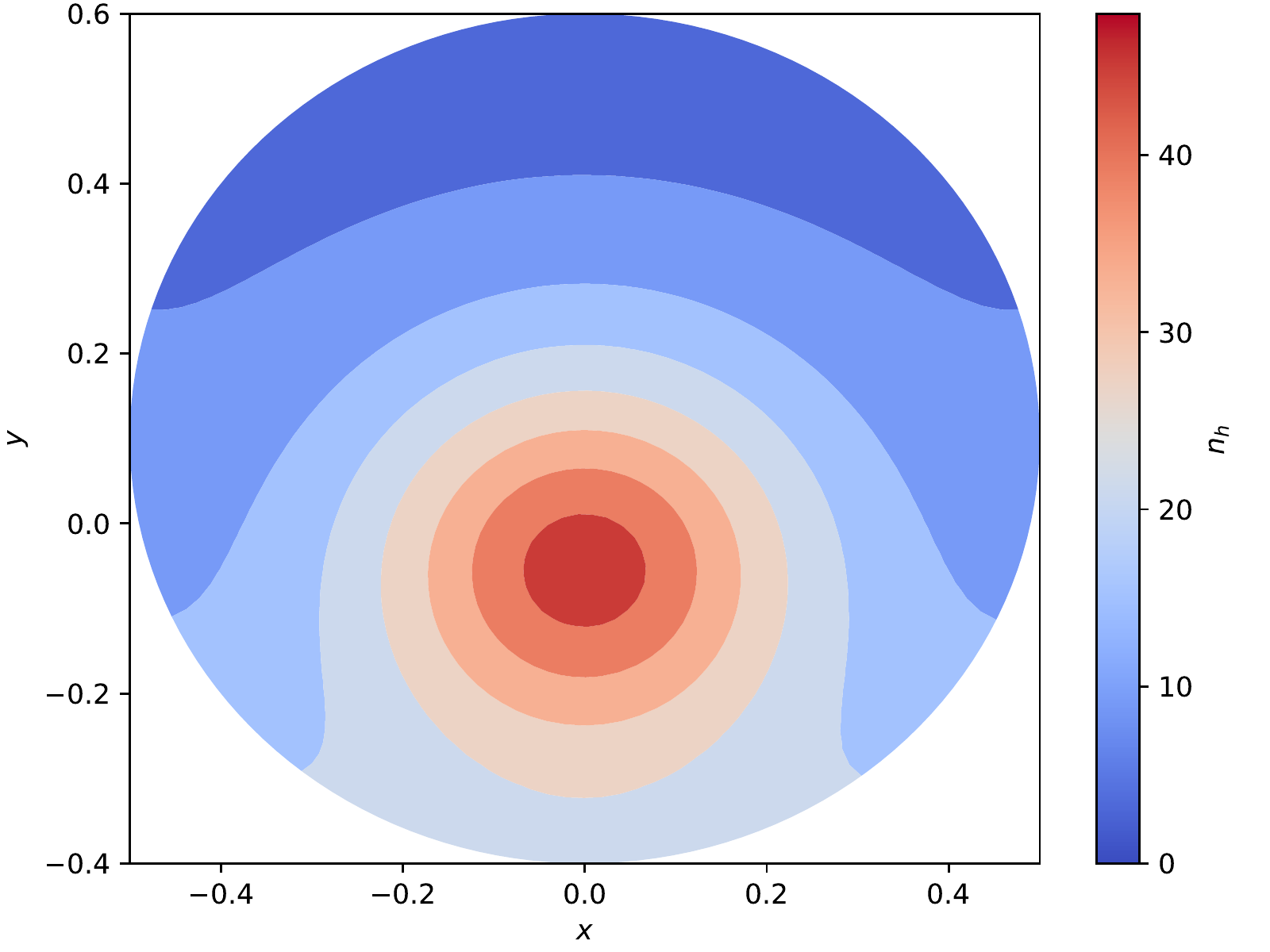}
    \end{subfigure}
    \begin{subfigure}[b]{0.23\textwidth}
        \centering
        \includegraphics[width=1.0\textwidth]{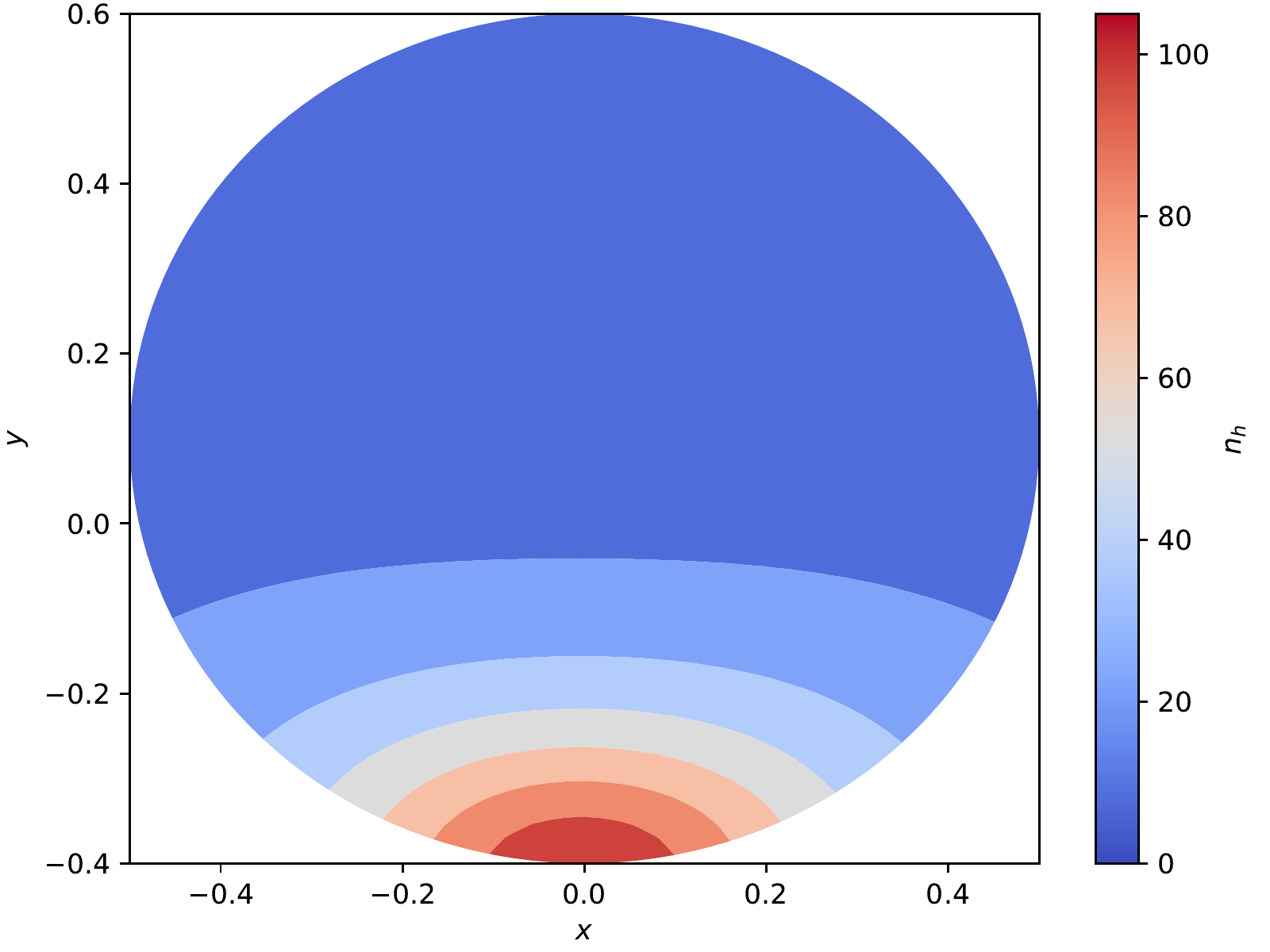}
    \end{subfigure}
    \begin{subfigure}[b]{0.23\textwidth}
        \centering
        \includegraphics[width=1.0\textwidth]{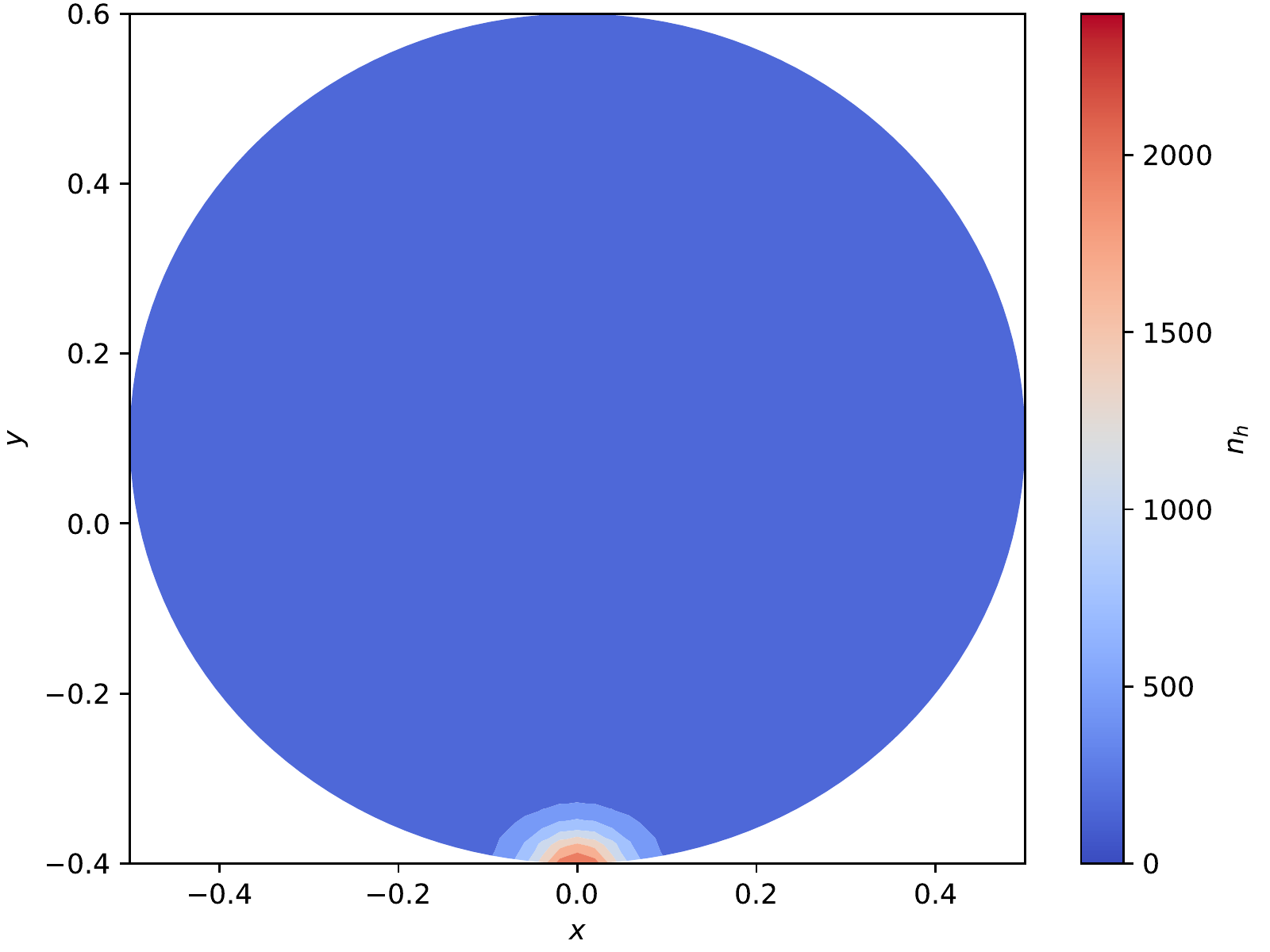}
    \end{subfigure}
            \caption{Snapshots at $t=0.02$, $t=0.04$, $0.5$ and $1.78$  of $\{n_h^m\}_m$ for $\eta_0=400$}\label{Snapshots_400_nh}
            \end{figure}
            \begin{figure}
    \begin{subfigure}[b]{0.23\textwidth}
        \centering
        \includegraphics[width=1.0\textwidth]{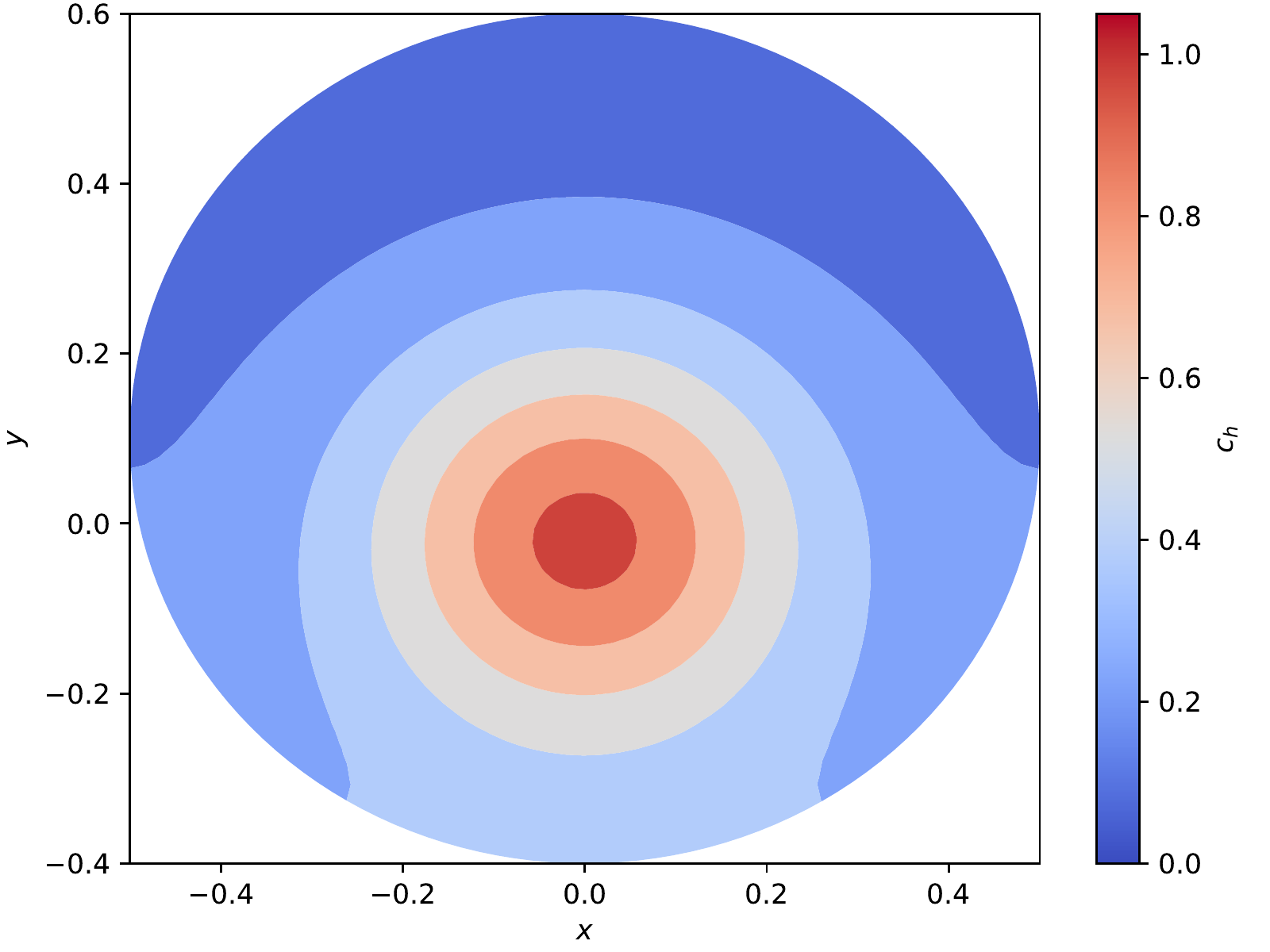}
    \end{subfigure}
    \begin{subfigure}[b]{0.23\textwidth}
        \centering
        \includegraphics[width=1.0\textwidth]{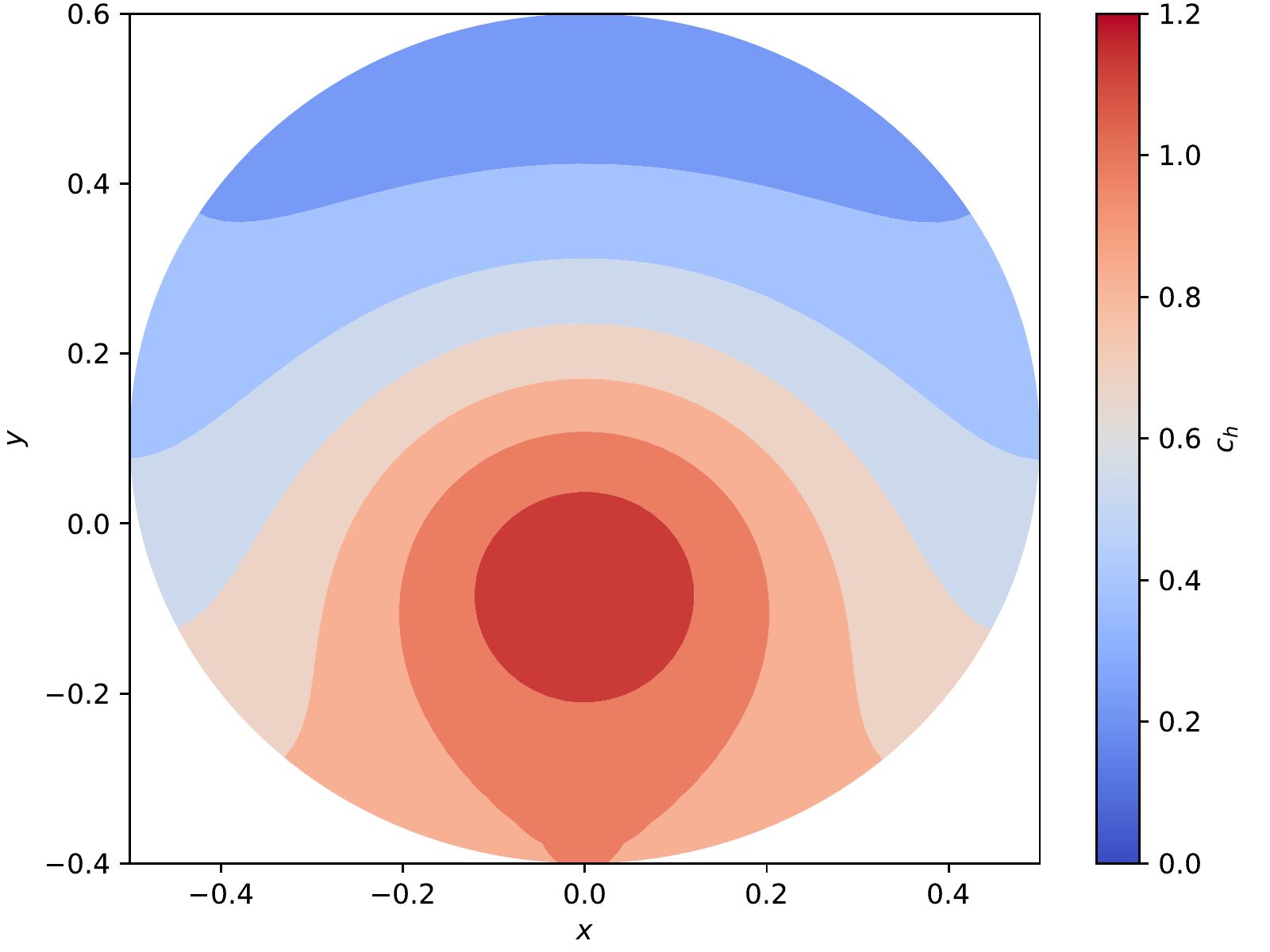}
    \end{subfigure}
    \begin{subfigure}[b]{0.23\textwidth}
        \centering
        \includegraphics[width=1.0\textwidth]{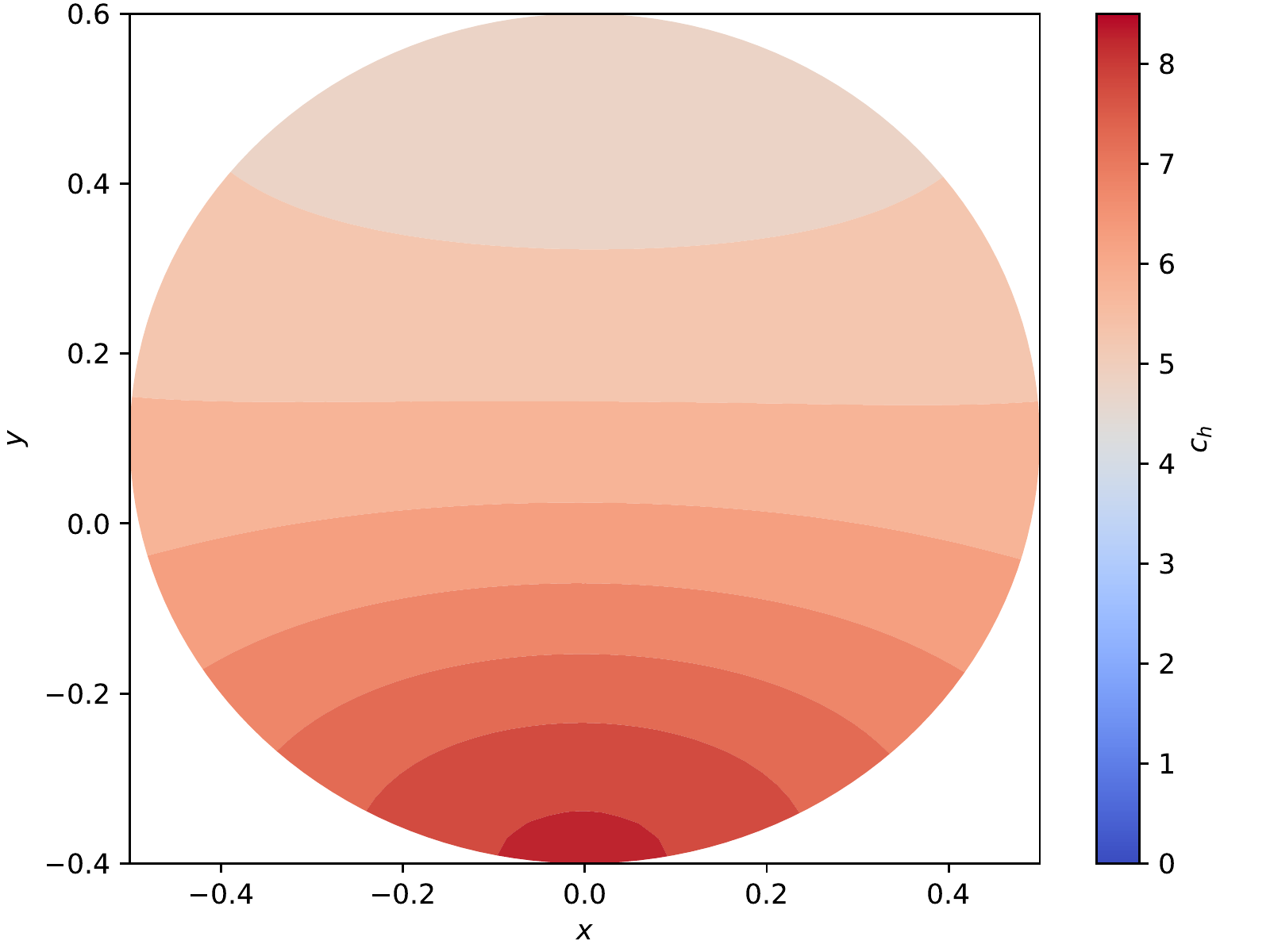}
    \end{subfigure}
    \begin{subfigure}[b]{0.23\textwidth}
        \centering
        \includegraphics[width=1.0\textwidth]{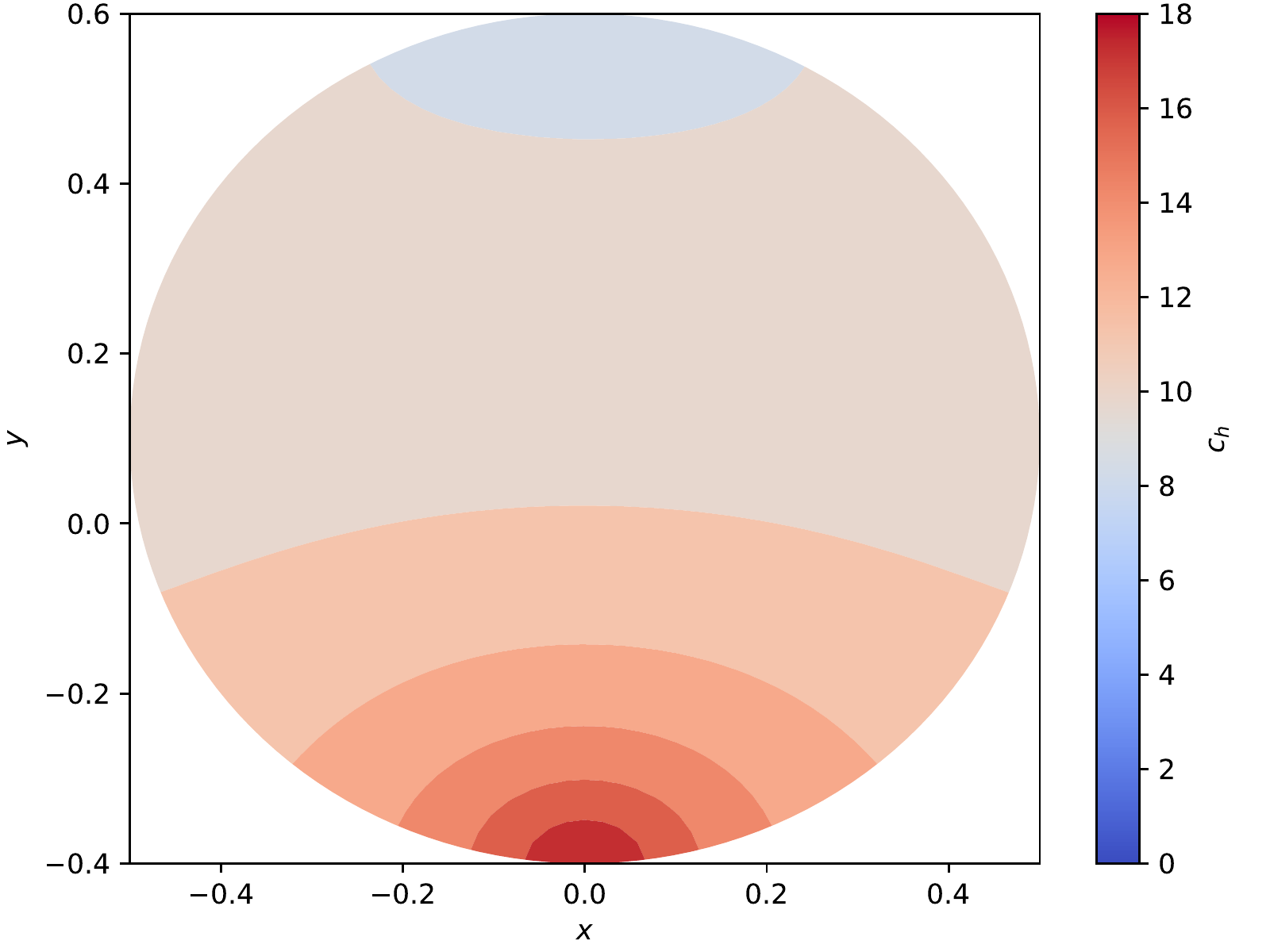}
    \end{subfigure}
                \caption{Snapshots at $t=0.02$, $t=0.04$, $0.5$ and $1.78$  of $\{c_h^m\}_m$ for $\eta_0=400$}\label{Snapshots_400_ch}
            \end{figure}
   \begin{figure}
    \begin{subfigure}[b]{0.23\textwidth}
        \centering
        \includegraphics[width=1.0\textwidth]{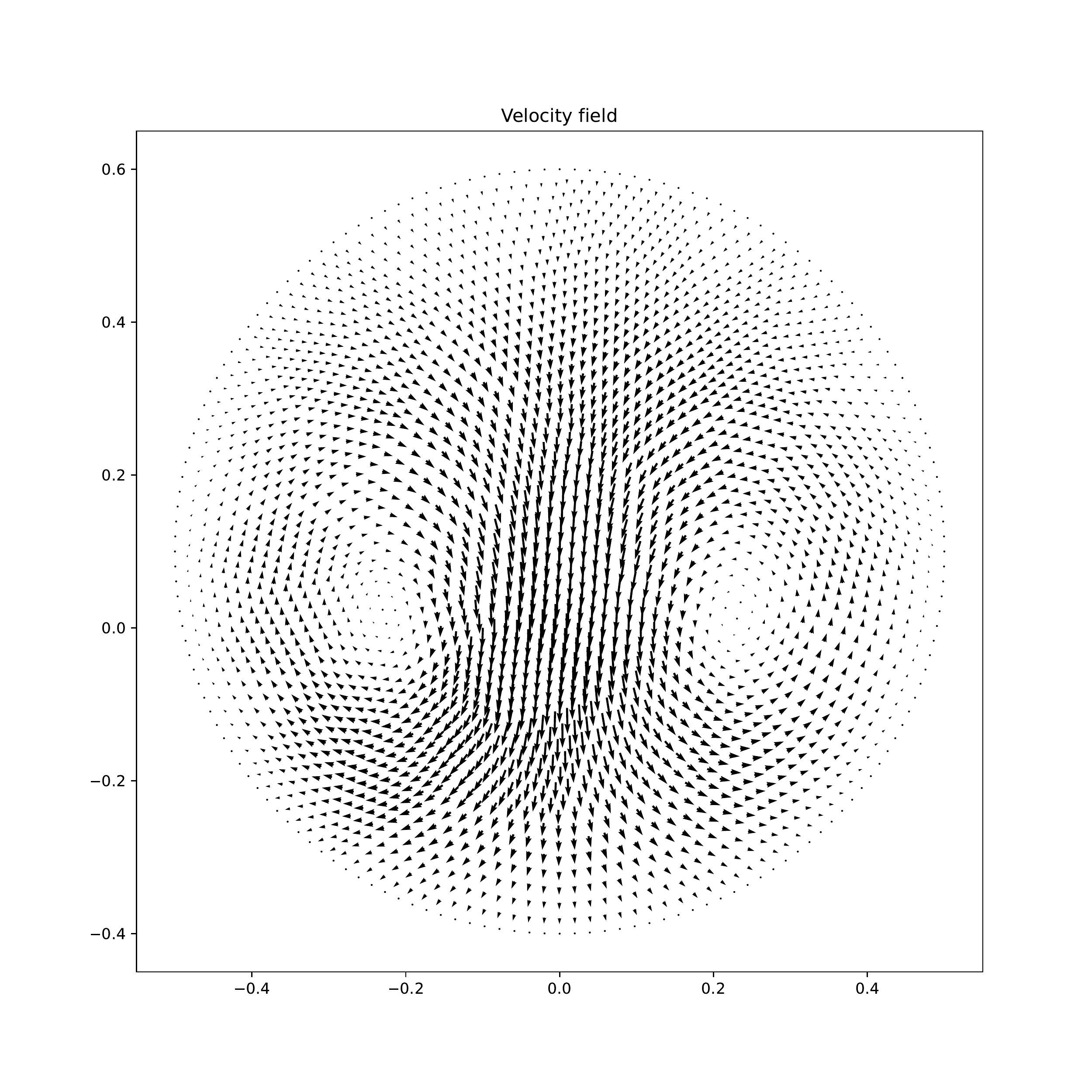}
    \end{subfigure}
    \begin{subfigure}[b]{0.23\textwidth}
        \centering
        \includegraphics[width=1.0\textwidth]{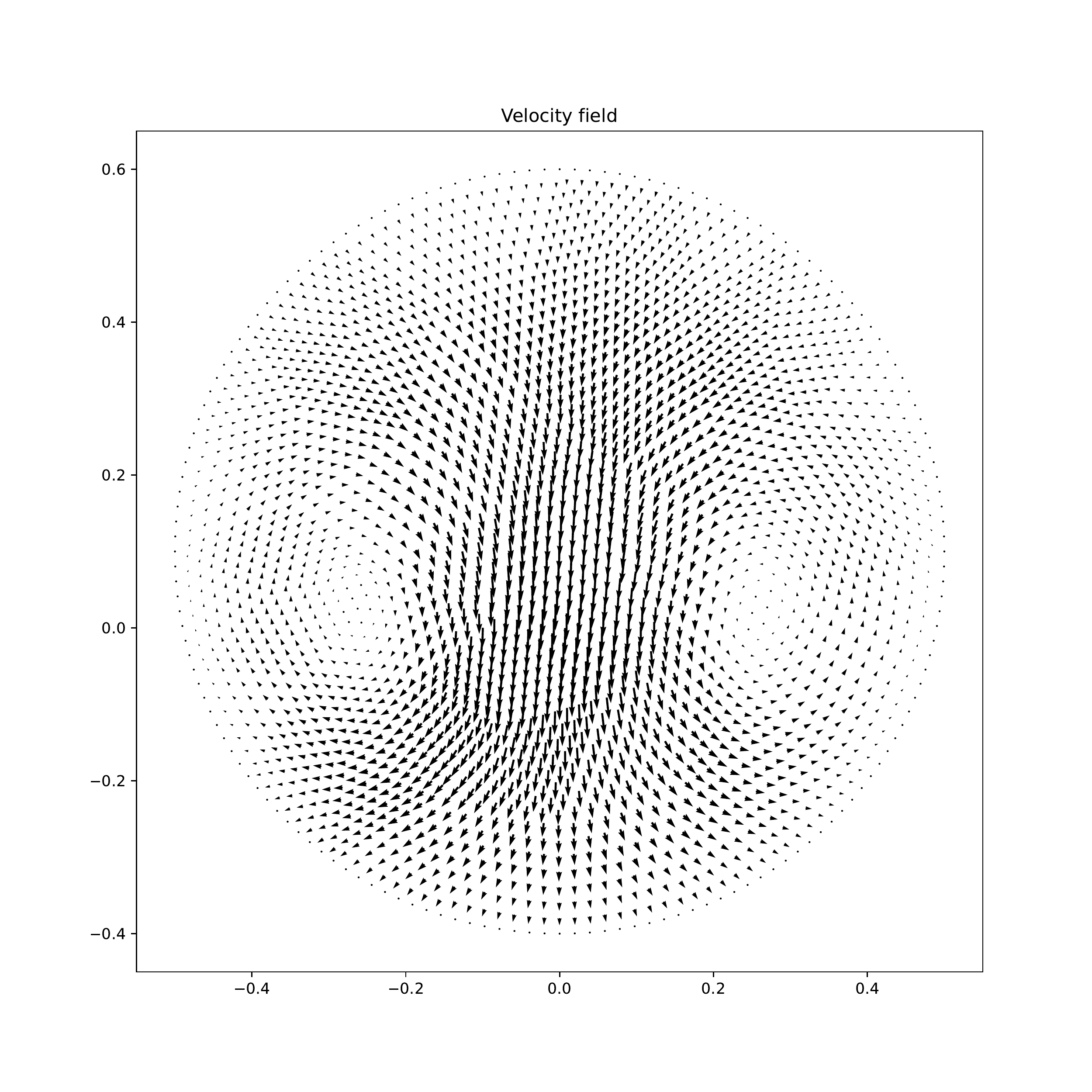}
    \end{subfigure}
    \begin{subfigure}[b]{0.23\textwidth}
        \centering
        \includegraphics[width=1.0\textwidth]{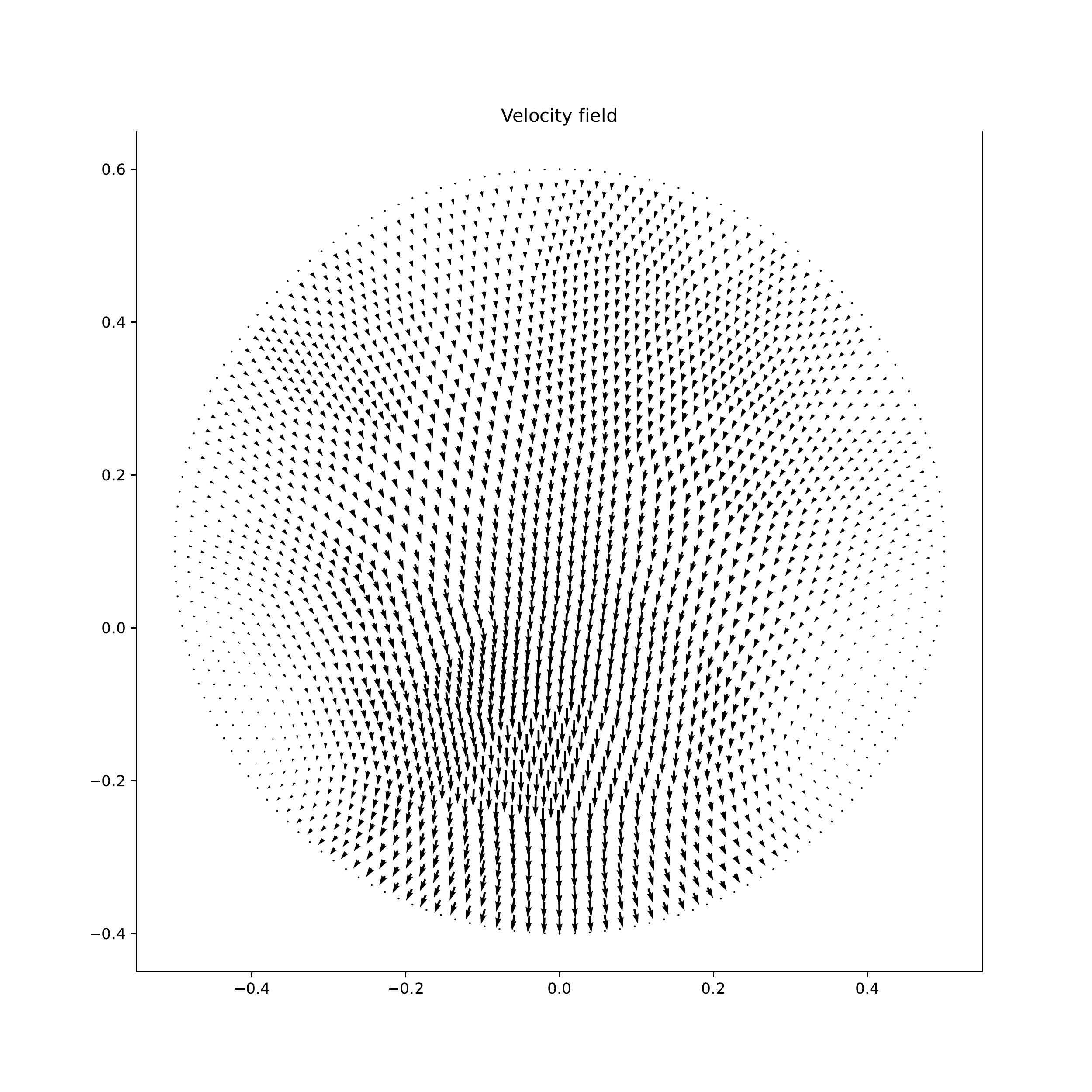}
    \end{subfigure}
    \begin{subfigure}[b]{0.23\textwidth}
        \centering
        \includegraphics[width=1.0\textwidth]{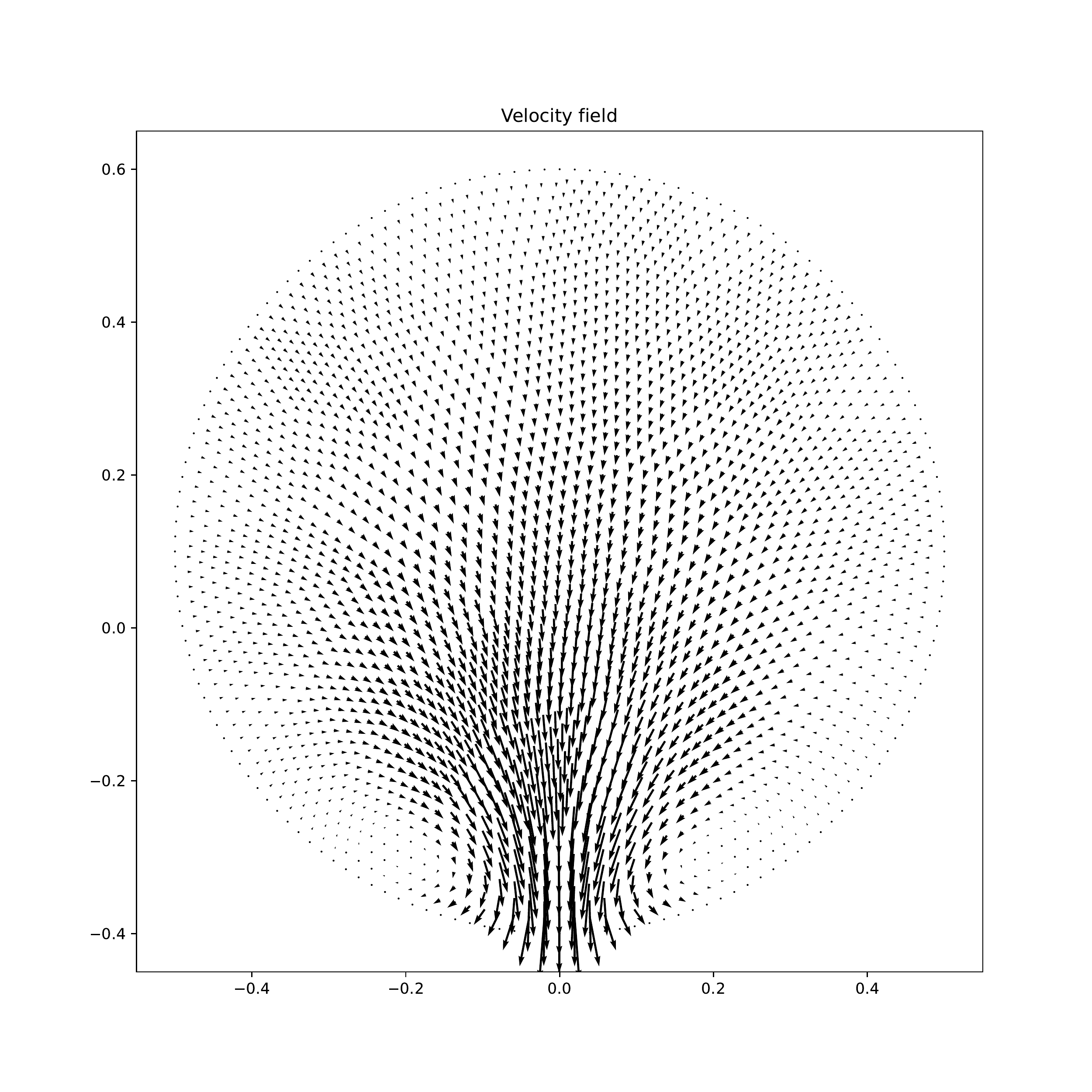}
    \end{subfigure}
    \caption{Snapshots at $t=0.02$, $t=0.04$, $0.5$ and $1.78$  of $\{\u_h^m\}_m$ for $\eta_0=400$}\label{Snapshots_400_uh}
\end{figure}
\begin{figure}
    \begin{subfigure}[b]{0.23\textwidth}
        \centering
        \includegraphics[width=1.0\textwidth]{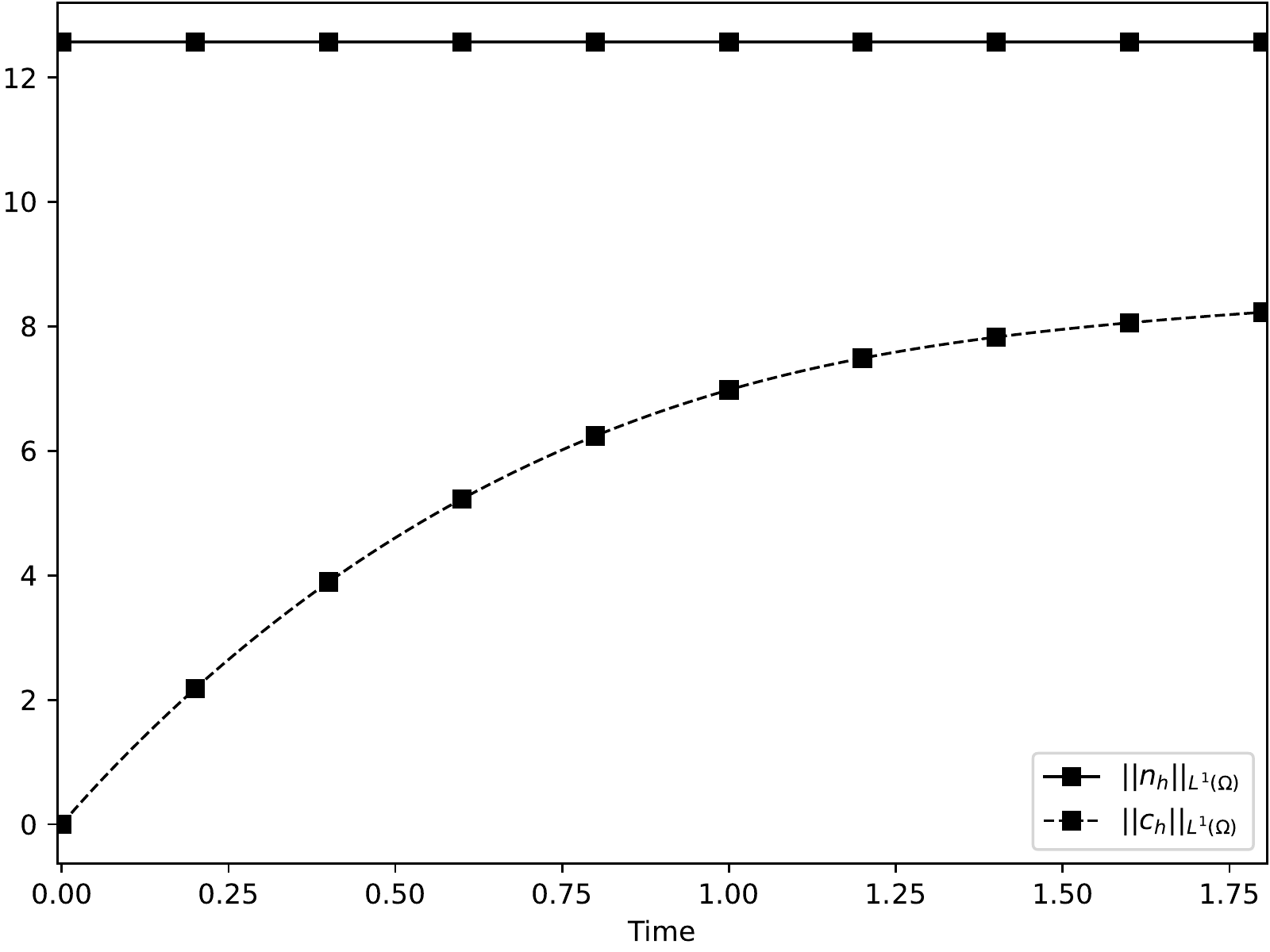}
    \end{subfigure}
    \begin{subfigure}[b]{0.23\textwidth}
        \centering
        \includegraphics[width=1.0\textwidth]{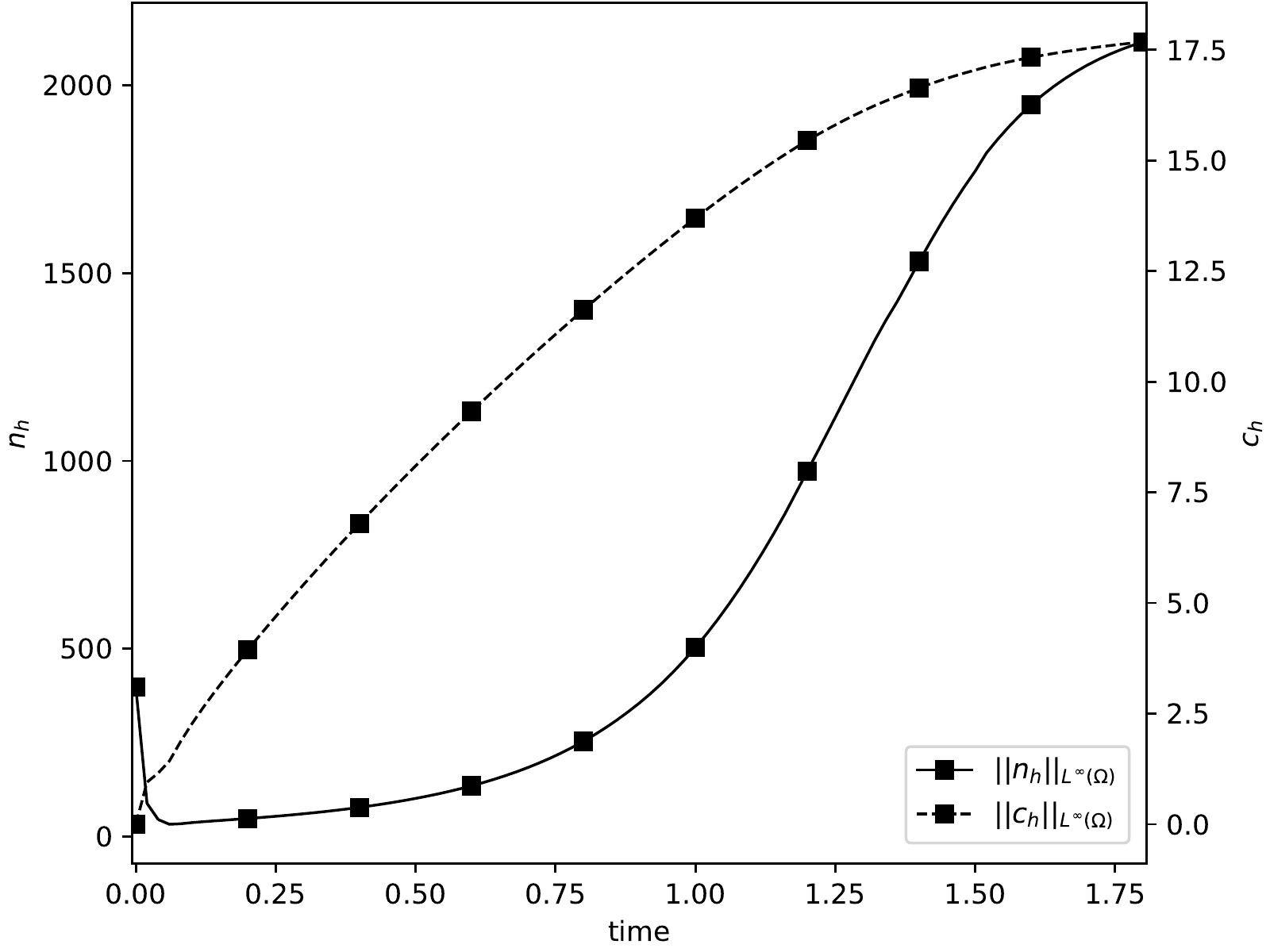}
    \end{subfigure}
    \begin{subfigure}[b]{0.23\textwidth}
        \centering
        \includegraphics[width=1.0\textwidth]{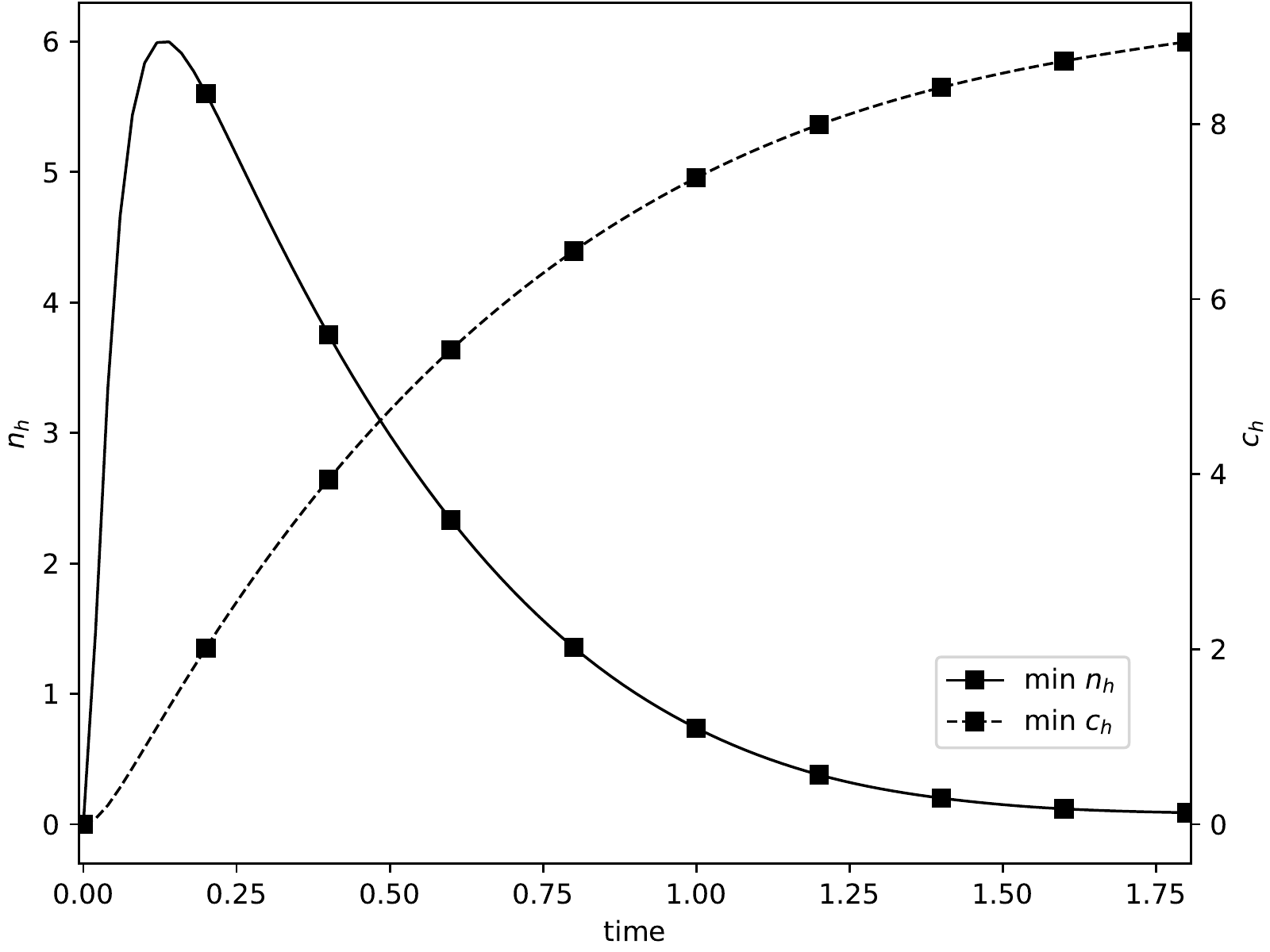}
    \end{subfigure}
            \caption{Evolution of total mass, maxima and minima $\{n_h^m\}$ and $\{c_h^m\}_m$ for $\eta=400$}\label{Graphs_400}.
\end{figure}

\subsection{Case: \texorpdfstring{$\eta_0=450$}{Lg} and \texorpdfstring{$\Phi_0=10$}{Lg}} We next continue selecting $\Phi_0=10$ and $\eta_0=450$, which leads to $\|n_0\|_{L^1(\Omega)}=14.1463$. For such a value we find a potential blowup formation for $\{n_h^m\}_m$ (Figure \ref{Snapshots_450_nh}), in favorable concordance with the theoretical results. This Dirac-type singularity is spontaneously formed at a time smaller than $t=0.7$ as displayed in Figure \ref{Graphs_450} (left) for maxima of $\{n_h^m\}_m$, where $\max_m \|n^m_h\|_{L^\infty}(\Omega)\approx 6\cdot 10^4$. As a result of the finite-time singularity  development, lower density areas for $\{n_h^m\}_m$ (Figure \ref{Snapshots_450_nh}) are dredged as chemotaxis dominates diffusion until being carried close to $0$ as seen in Figure \ref{Graphs_450} (middle); on the contrary, Figure \ref{Graphs_450} (middle) shows minina for $\{c_h^m\}_m$, which get smaller values than  $\eta_0=400$. Snapshots of $\{n_h^m\}_m$, $\{c_h^m\}_m$ and $\{\u_n^m\}_m$ at $t=0.04$, $0.1$, $0.4$ and $0.8$ are illustrated in Figures \ref{Snapshots_450_nh}, \ref{Snapshots_450_ch} and \ref{Snapshots_450_uh}, which are evidently different from those of  $\eta_0=400$.       
\begin{figure}
    \begin{subfigure}[b]{0.23\textwidth}
        \centering
        \includegraphics[width=1.0\textwidth]{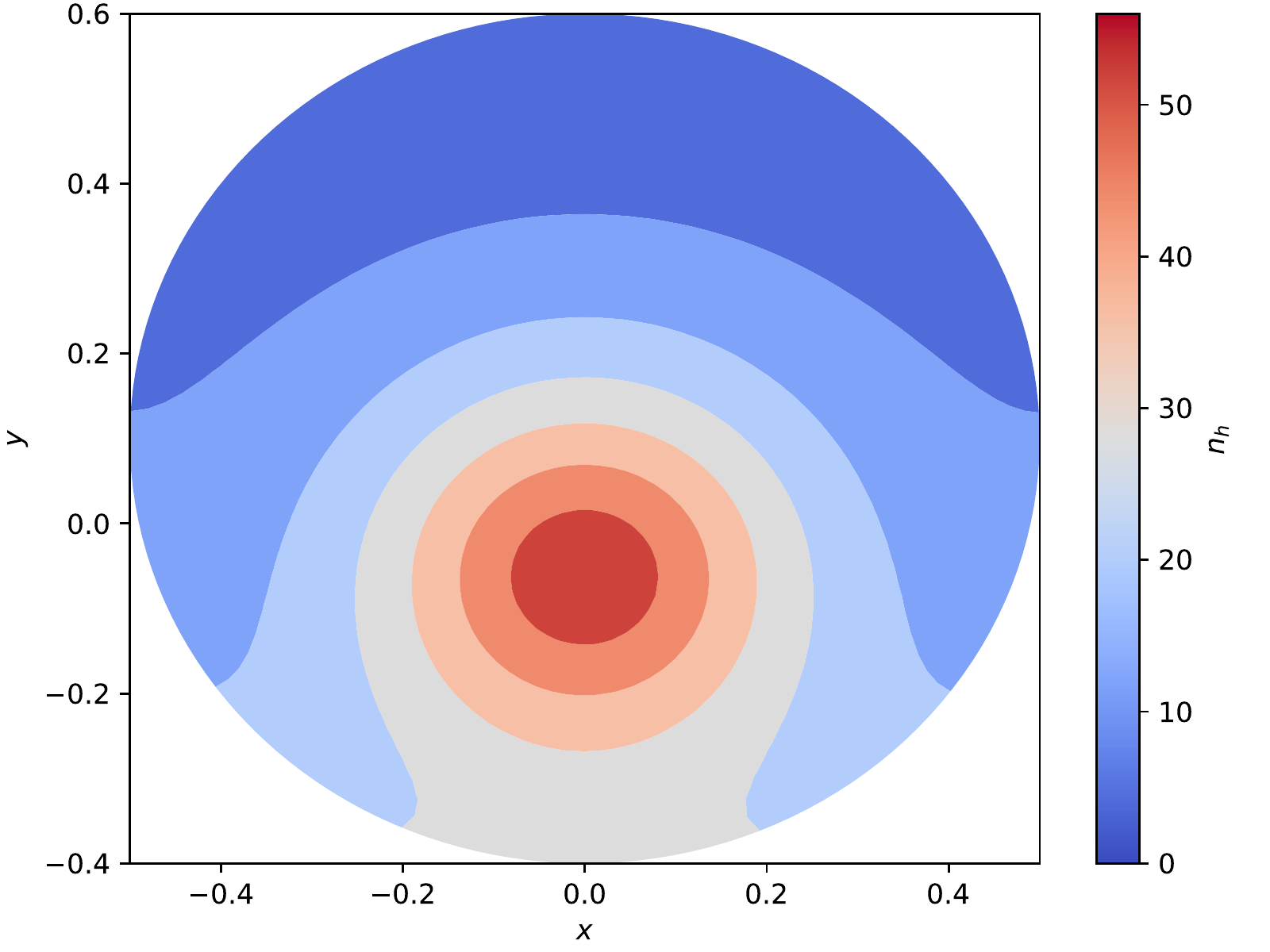}
    \end{subfigure}
    \begin{subfigure}[b]{0.23\textwidth}
        \centering
        \includegraphics[width=1.0\textwidth]{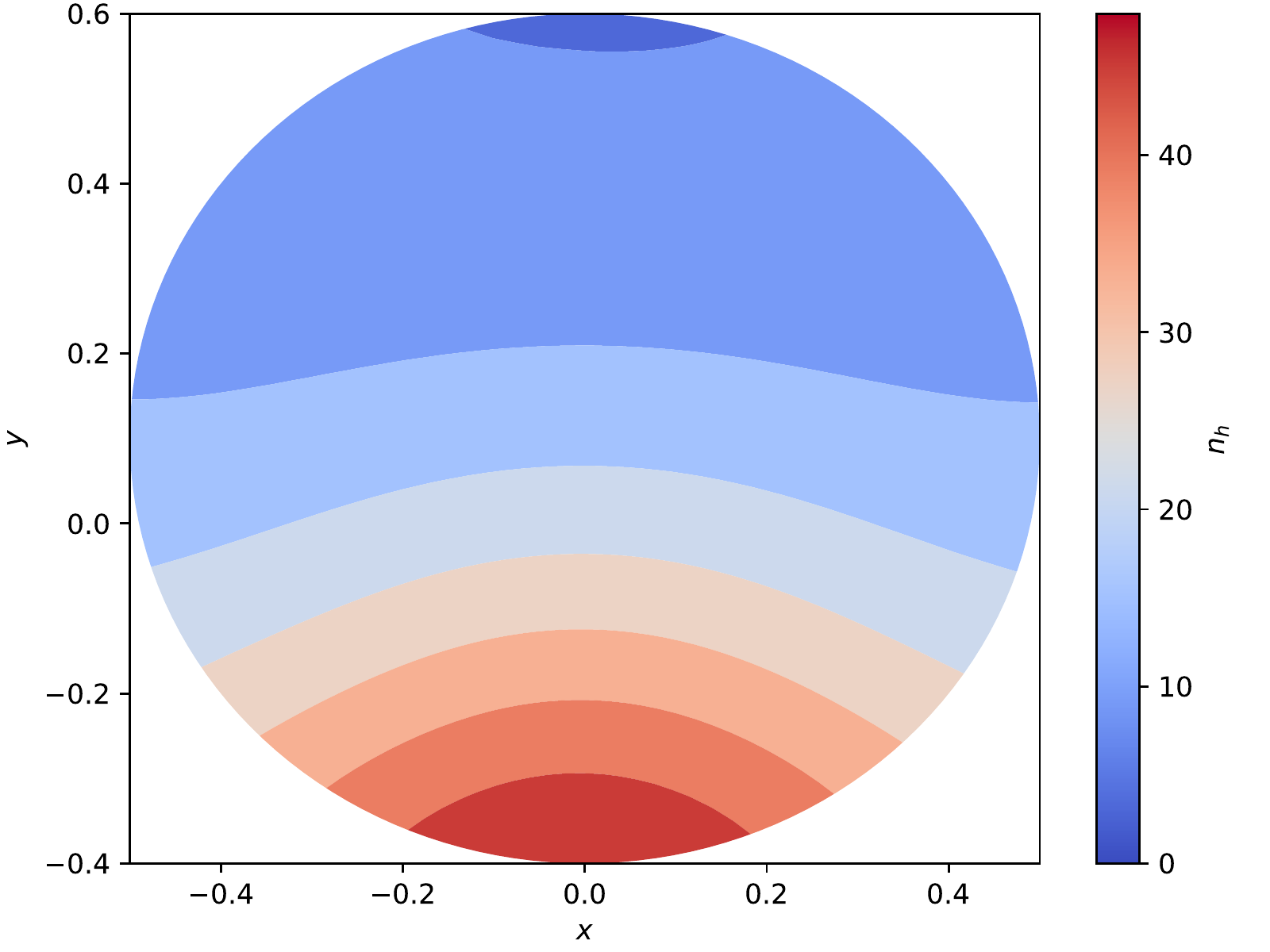}
    \end{subfigure}
    \begin{subfigure}[b]{0.23\textwidth}
        \centering
        \includegraphics[width=1.0\textwidth]{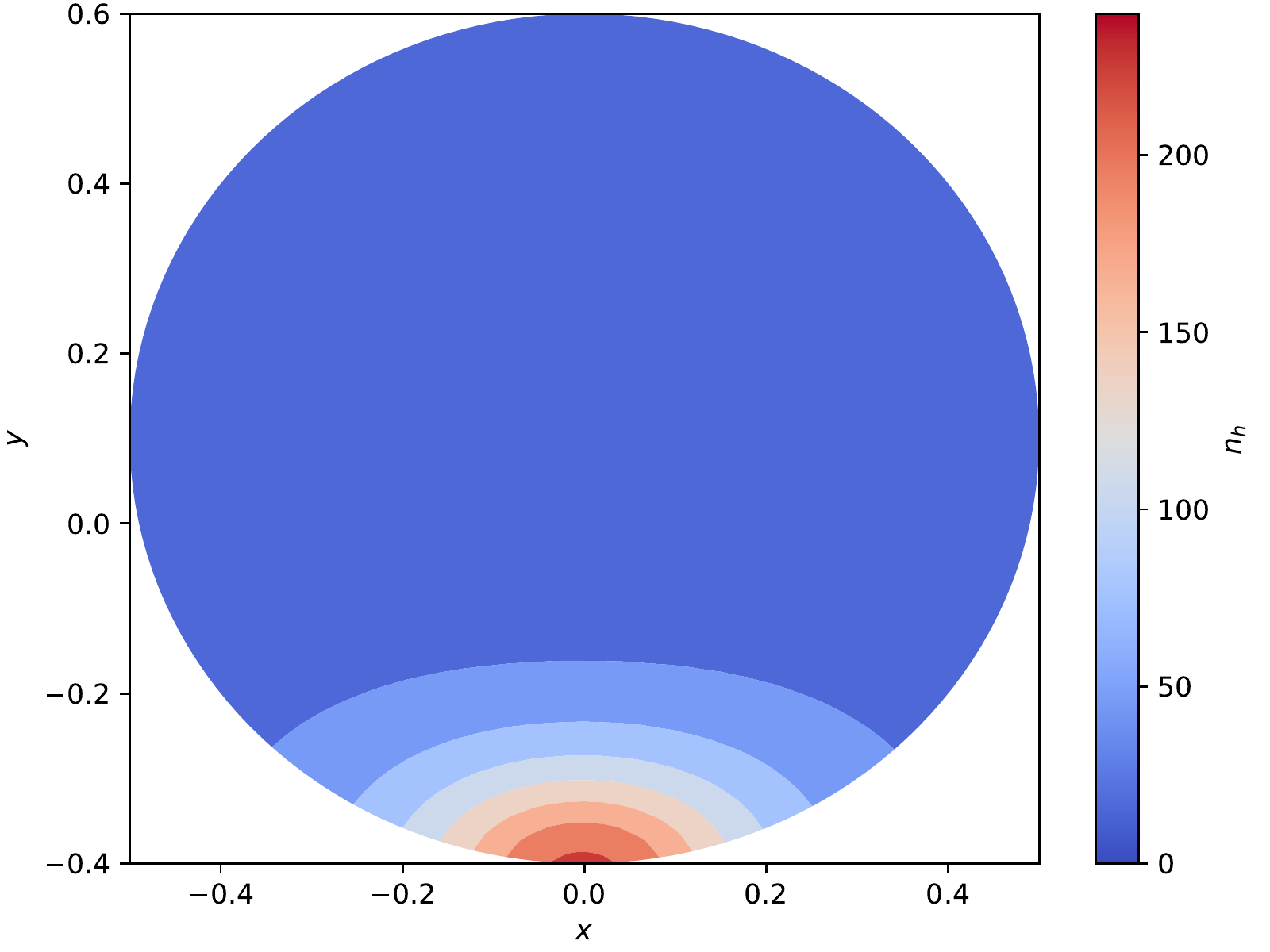}
    \end{subfigure}
    \begin{subfigure}[b]{0.23\textwidth}
        \centering
        \includegraphics[width=1.0\textwidth]{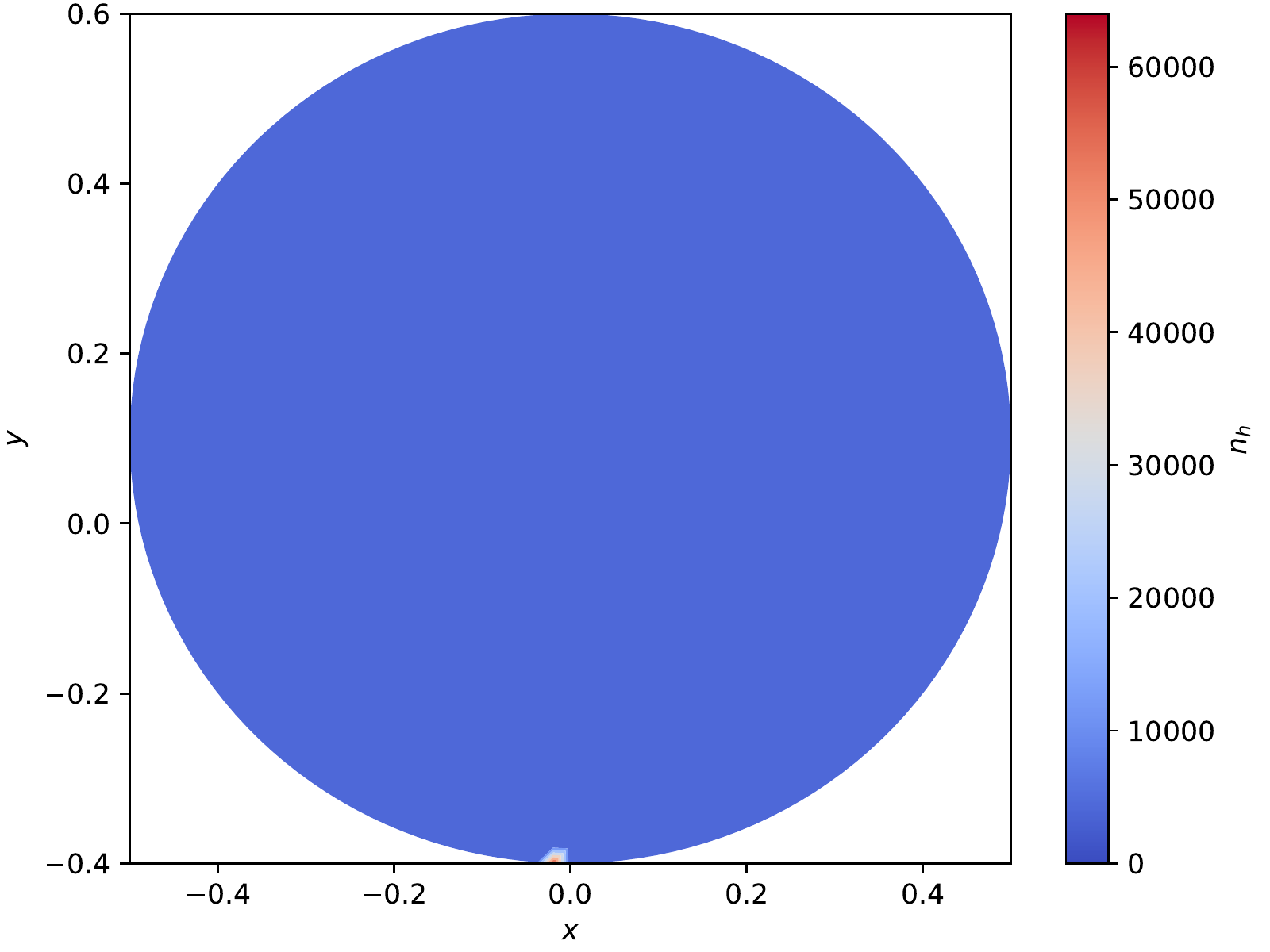}
    \end{subfigure}
            \caption{Snapshots at $t=0.04$, $0.1$, $0.4$ and $0.8$ of $\{n_h^m\}_m$ for $\eta_0=450$ on $\mathcal{T}_h^*$.}\label{Snapshots_450_nh}
\end{figure}
\begin{figure}
    \begin{subfigure}[b]{0.23\textwidth}
        \centering
        \includegraphics[width=1.0\textwidth]{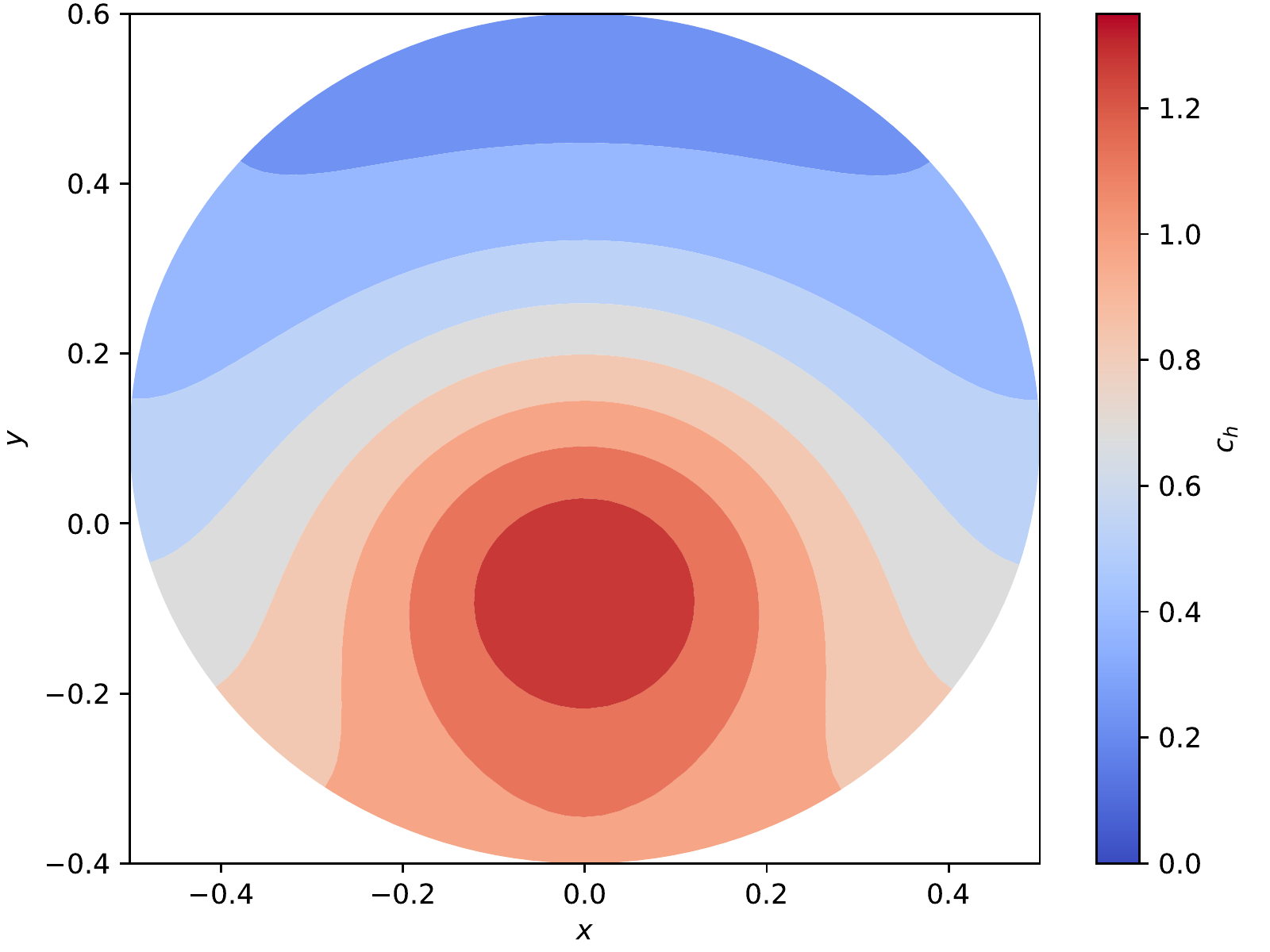}
    \end{subfigure}
    \begin{subfigure}[b]{0.23\textwidth}
        \centering
        \includegraphics[width=1.0\textwidth]{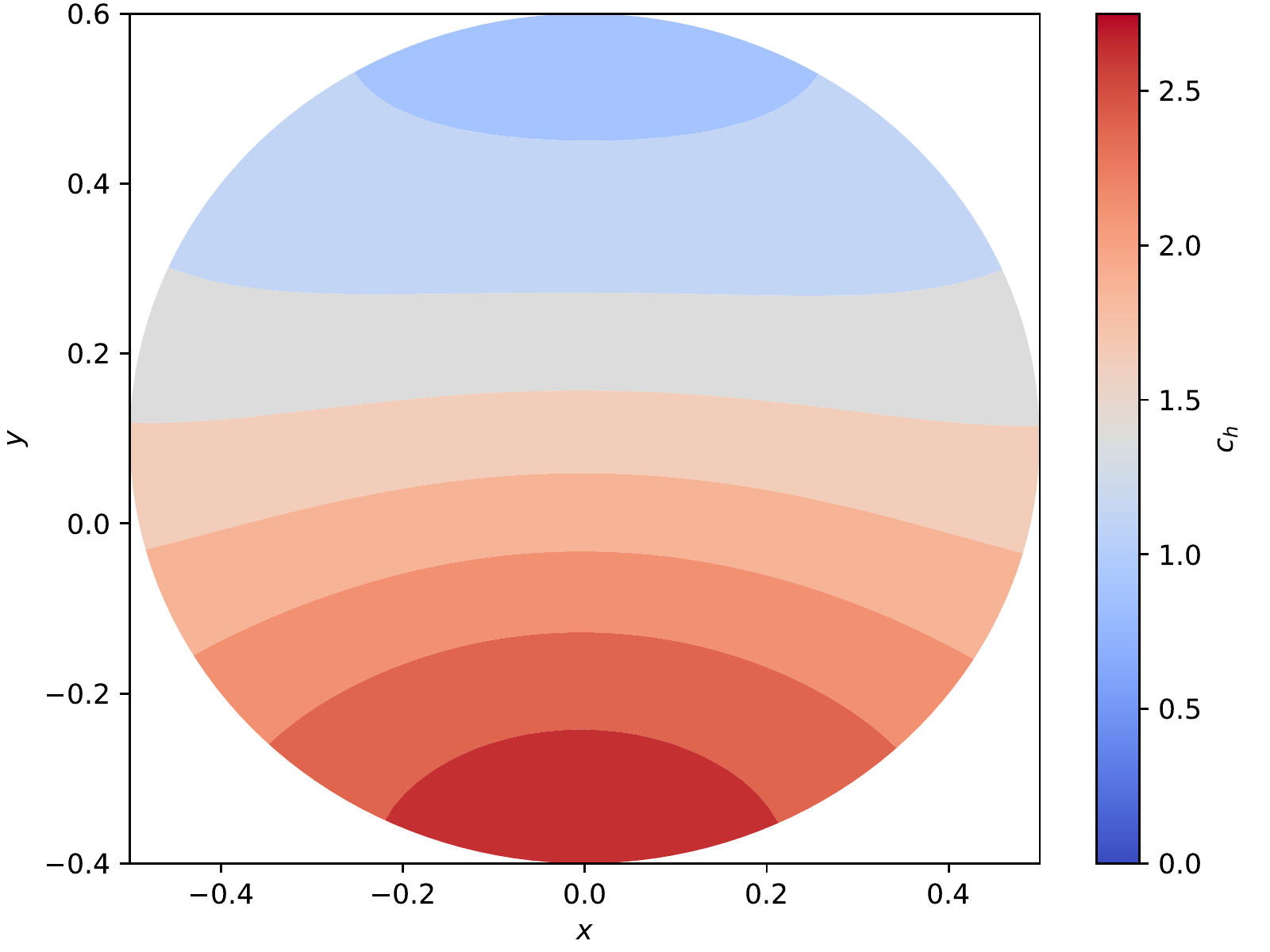}
    \end{subfigure}
    \begin{subfigure}[b]{0.23\textwidth}
        \centering
        \includegraphics[width=1.0\textwidth]{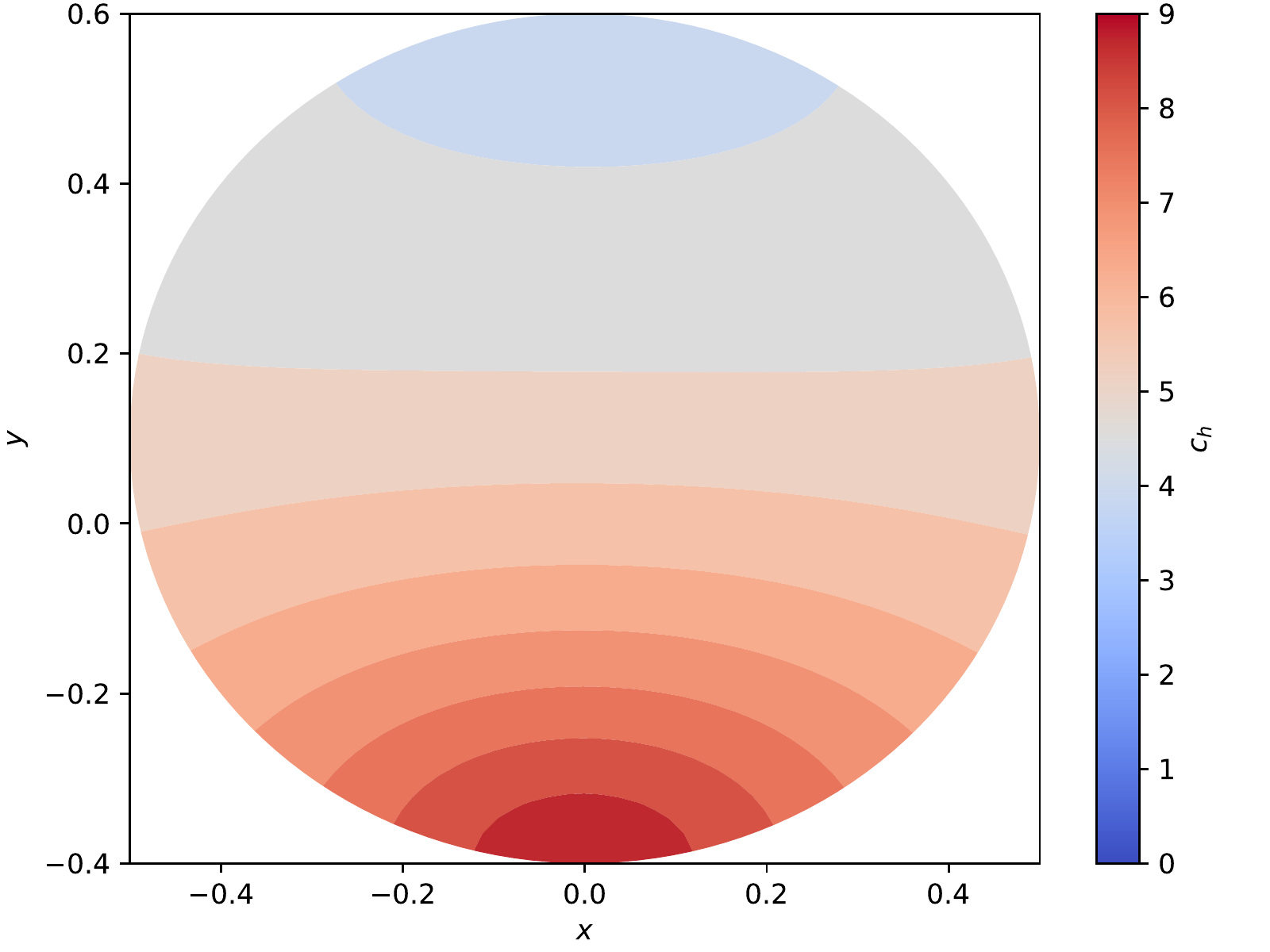}
    \end{subfigure}
    \begin{subfigure}[b]{0.23\textwidth}
        \centering
        \includegraphics[width=1.0\textwidth]{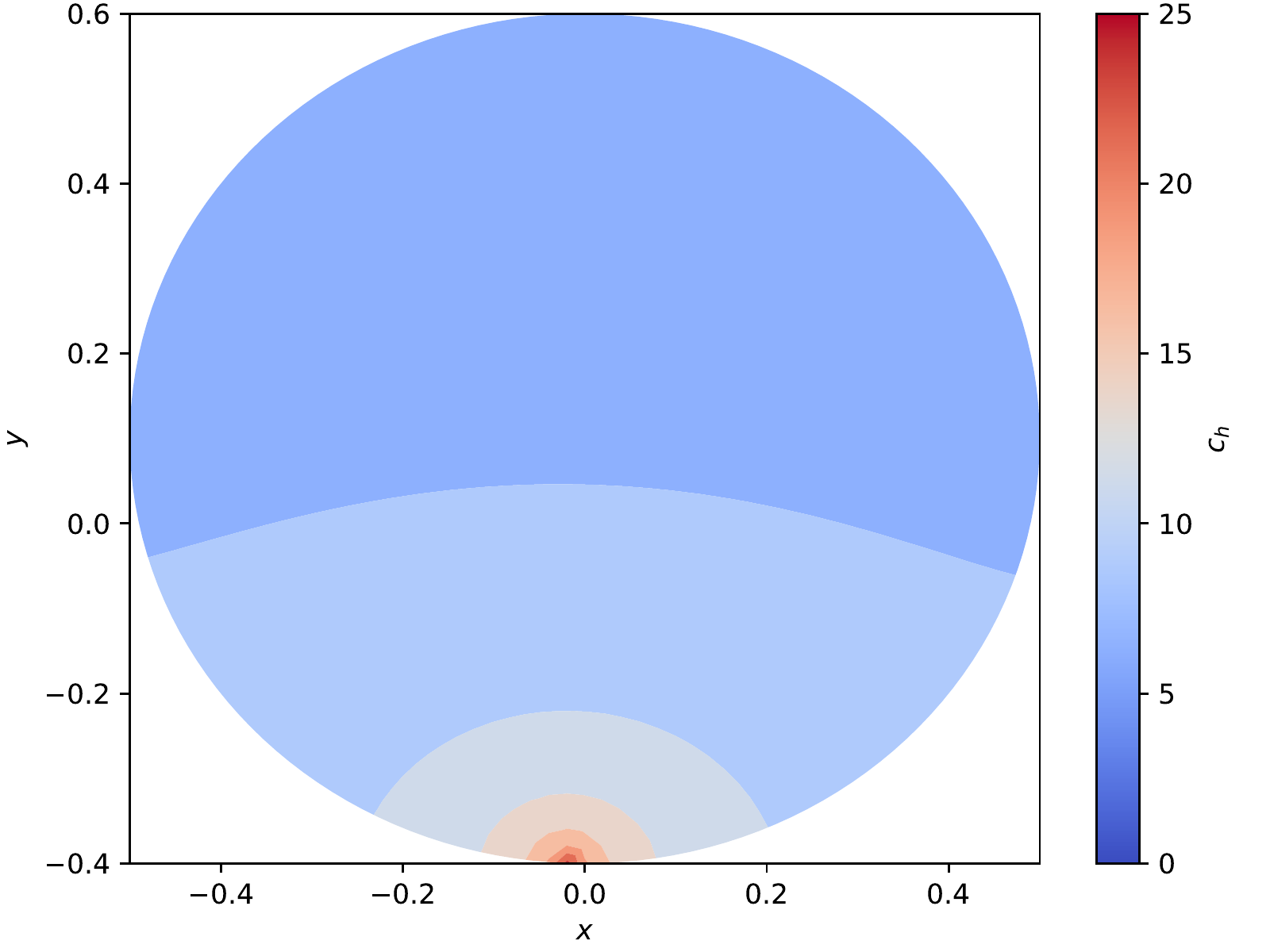}
    \end{subfigure}
\caption{Snapshots at $t=0.04$, $0.1$, $0.4$ and $0.8$ of $\{c_h^m\}_m$ for $\eta_0=450$ on $\mathcal{T}_h^*$.}\label{Snapshots_450_ch}
\end{figure}
\begin{figure}
    \begin{subfigure}[b]{0.23\textwidth}
        \centering
        \includegraphics[width=1.0\textwidth]{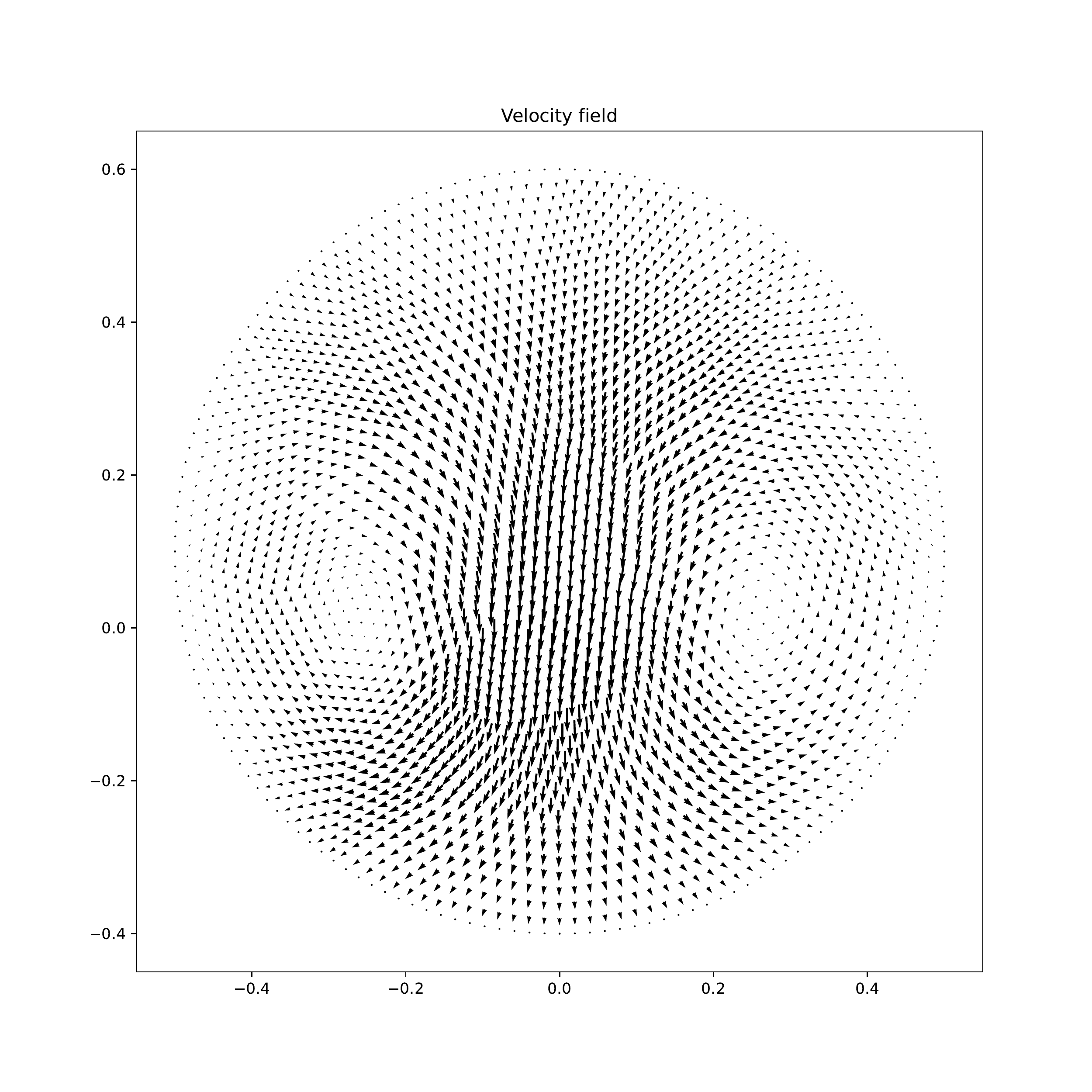}
    \end{subfigure}
    \begin{subfigure}[b]{0.23\textwidth}
        \centering
        \includegraphics[width=1.0\textwidth]{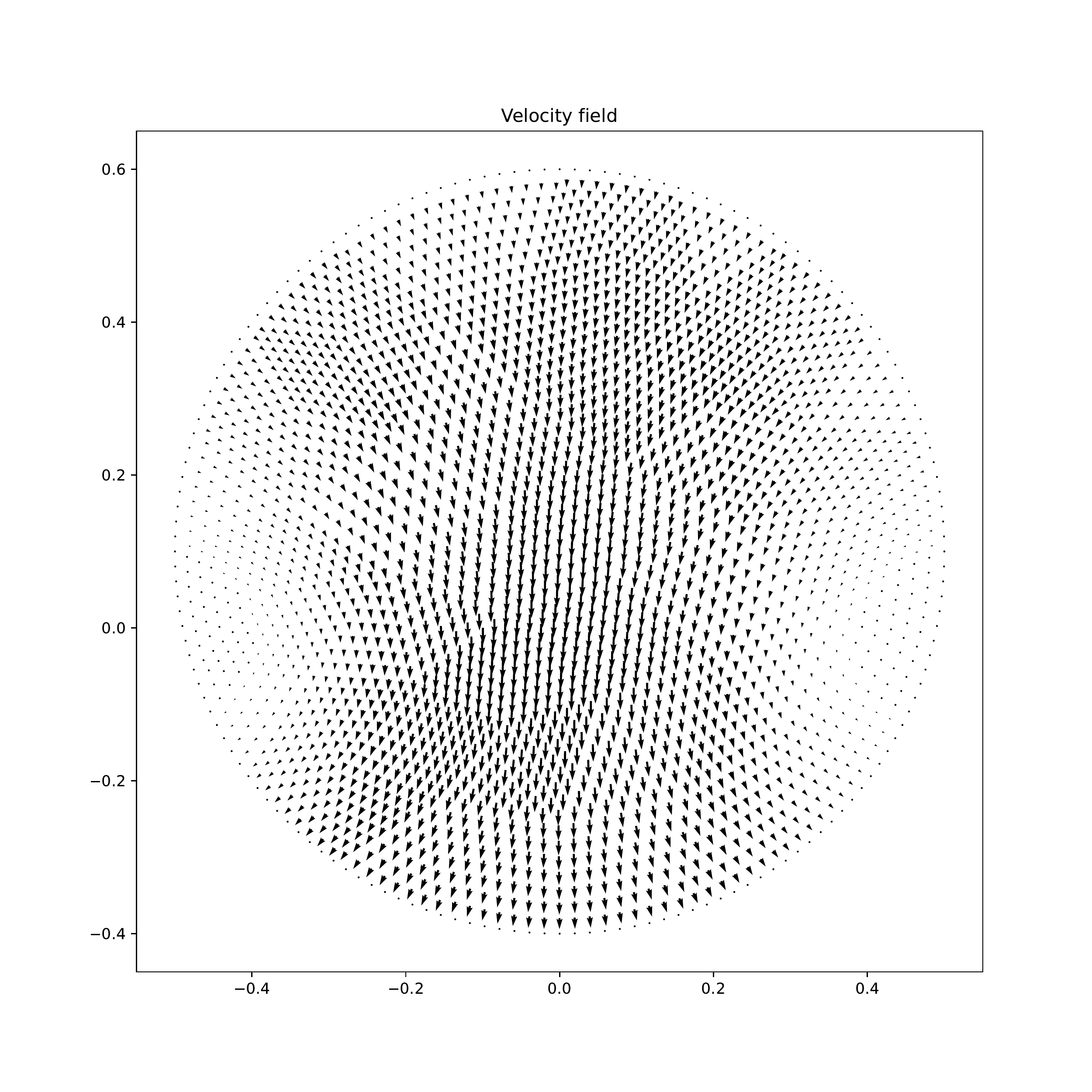}
    \end{subfigure}
    \begin{subfigure}[b]{0.23\textwidth}
        \centering
        \includegraphics[width=1.0\textwidth]{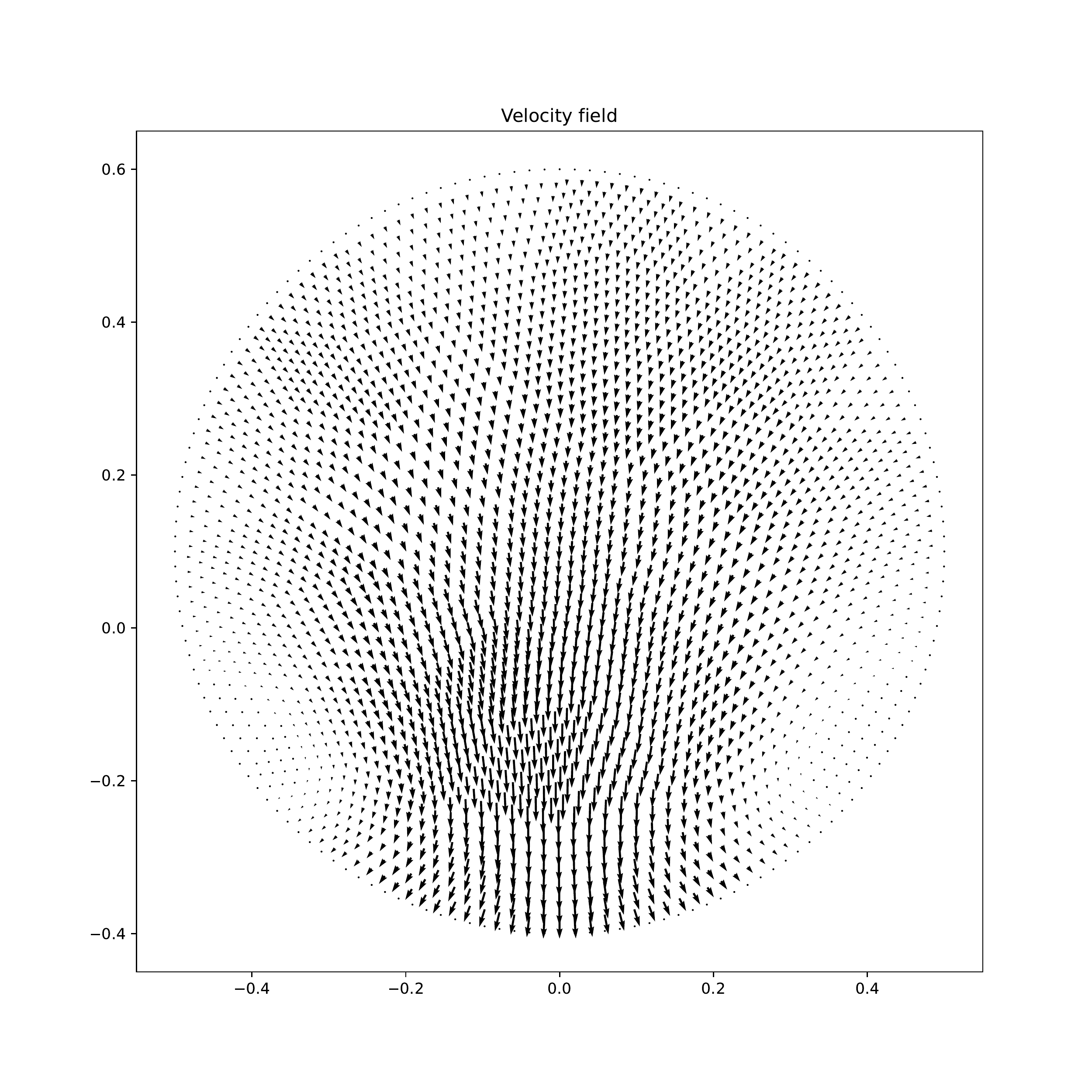}
    \end{subfigure}
    \begin{subfigure}[b]{0.23\textwidth}
        \centering
        \includegraphics[width=1.0\textwidth]{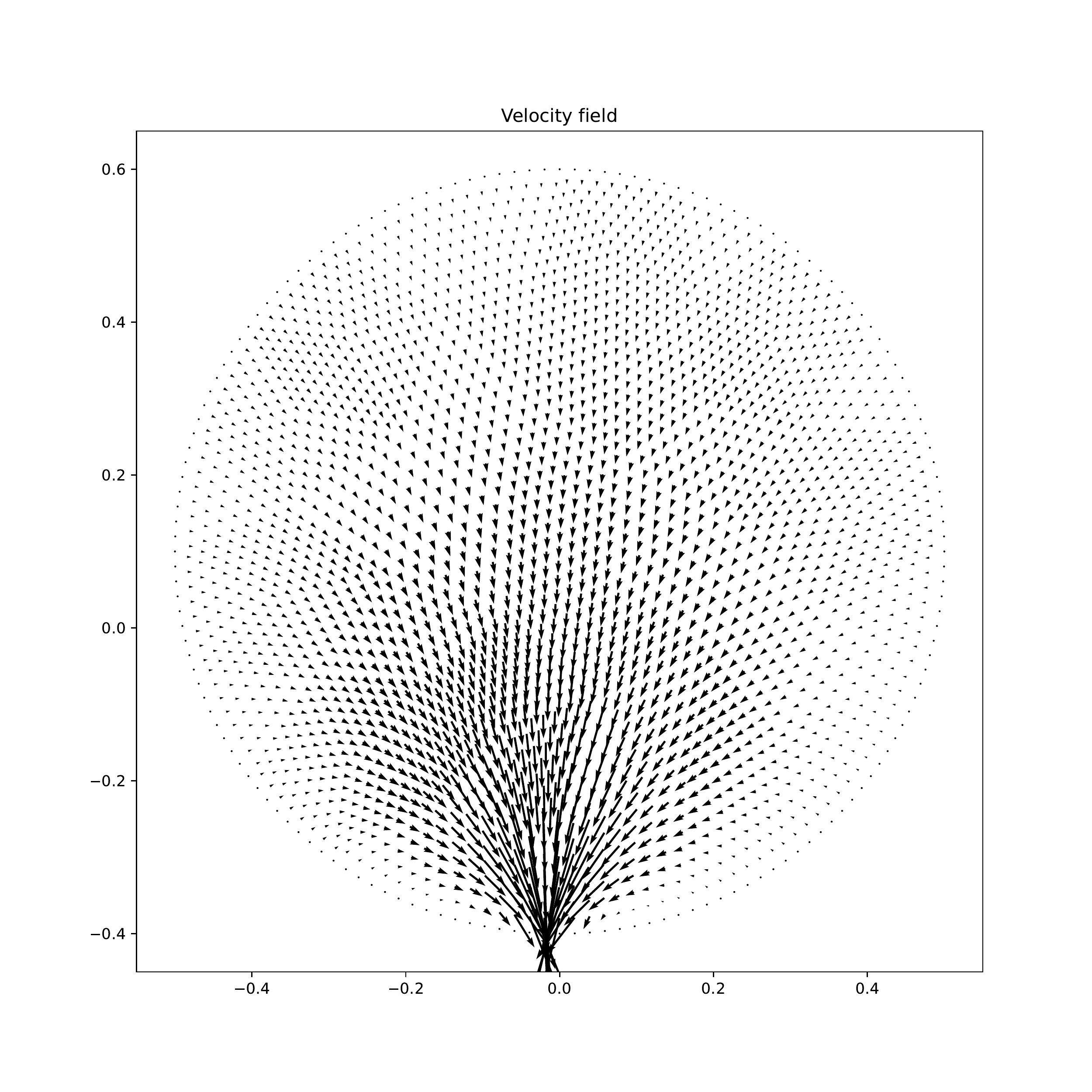}
    \end{subfigure}
    \caption{Snapshots at $t=0.04$, $0.1$, $0.4$ and $0.8$ of $\{\u_h^m\}_m$ for $\eta_0=450$ on $\mathcal{T}_h^*$.}\label{Snapshots_450_uh}
\end{figure}
\begin{figure}
    \begin{subfigure}[b]{0.23\textwidth}
        \centering
        \includegraphics[width=1.0\textwidth]{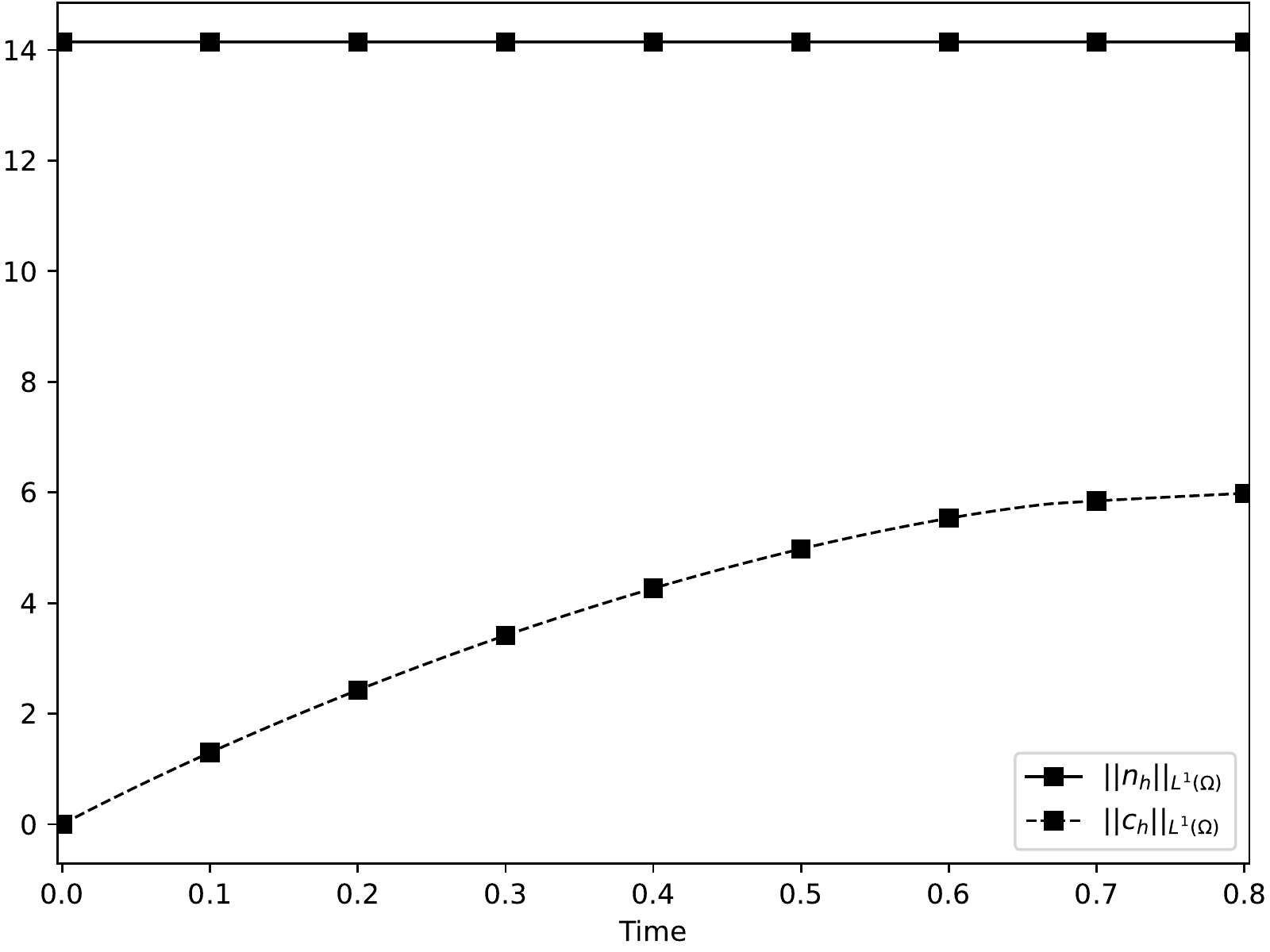}
    \end{subfigure}
    \begin{subfigure}[b]{0.23\textwidth}
        \centering
        \includegraphics[width=1.0\textwidth]{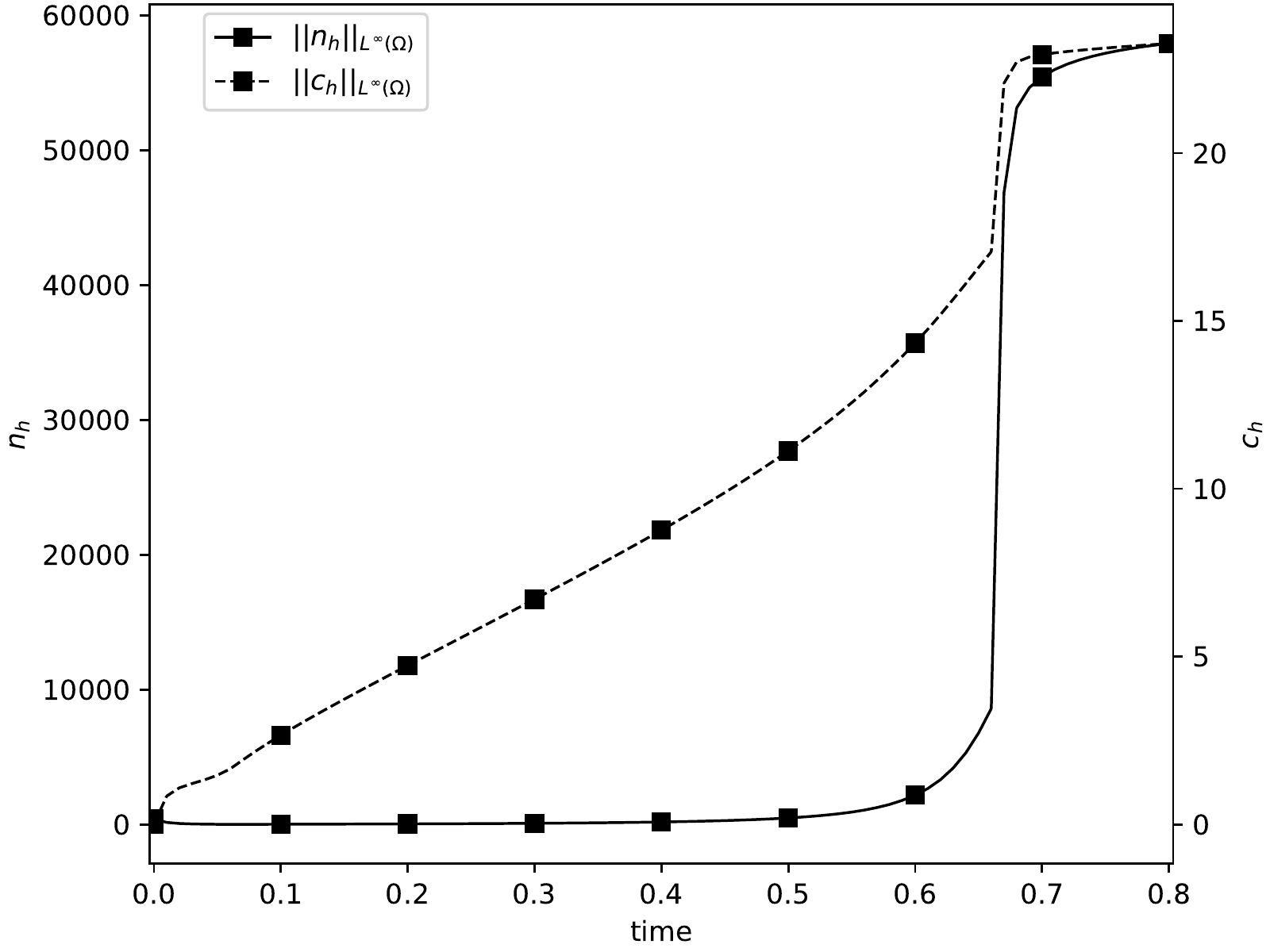}
    \end{subfigure}
    \begin{subfigure}[b]{0.23\textwidth}
        \centering
        \includegraphics[width=1.0\textwidth]{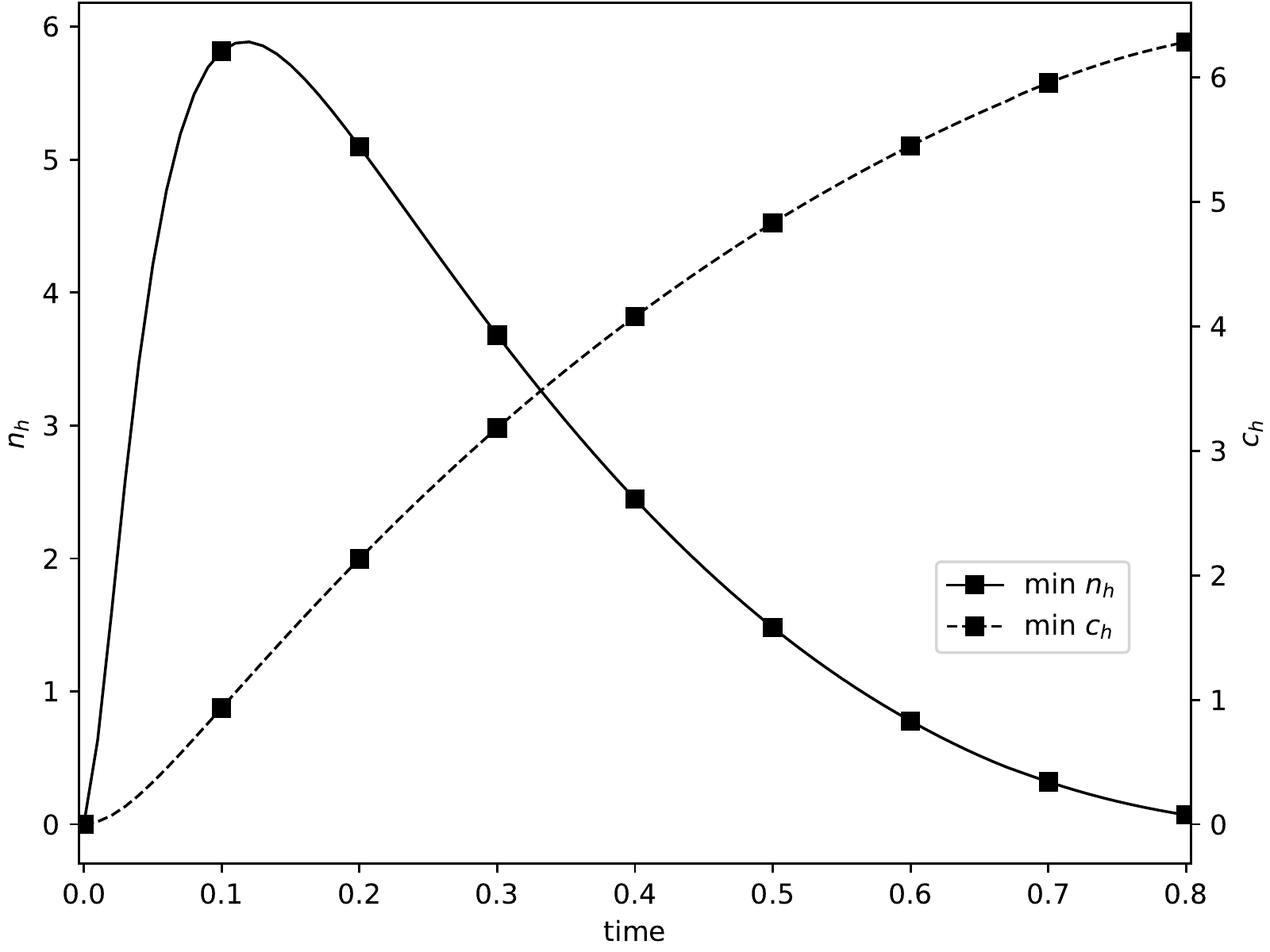}
    \end{subfigure}
            \caption{Plots of total mass, maxima and minima of $\{n_h^m\}_m$ and $\{c_h^m\}_m$ for $\eta_0=450$ on $\mathcal{T}_h^*$.}\label{Graphs_450}
\end{figure}

To make the singularity formation far more manifest, we use the partially refined mesh $\mathcal{T}_h^*$ displayed in Figure \ref{meshes} (right). The first remarkable aspect in Figure \ref{Graphs_450_finer} is concerned with maxima of $\{n_h^m\}_m$, which reaches higher values; $\max_{m}\|n_h^{m}\|\approx 4\cdot 10^{5}$.  As for the singularity-formation time, it is smaller being $t\approx0.5$. The qualitative evolution of $\{n_h^m\}_m$, $\{c_h^m\}_m$ and $\{\u_h^m\}_m$ illustrated in Figures \ref{Snapshots_450_finer_nh}, \ref{Snapshots_450_finer_ch} and \ref{Snapshots_450_finer_nh}  does not differ from that computed on $\mathcal{T}_h$.       
\begin{figure}
    \begin{subfigure}[b]{0.23\textwidth}
        \centering
        \includegraphics[width=1.0\textwidth]{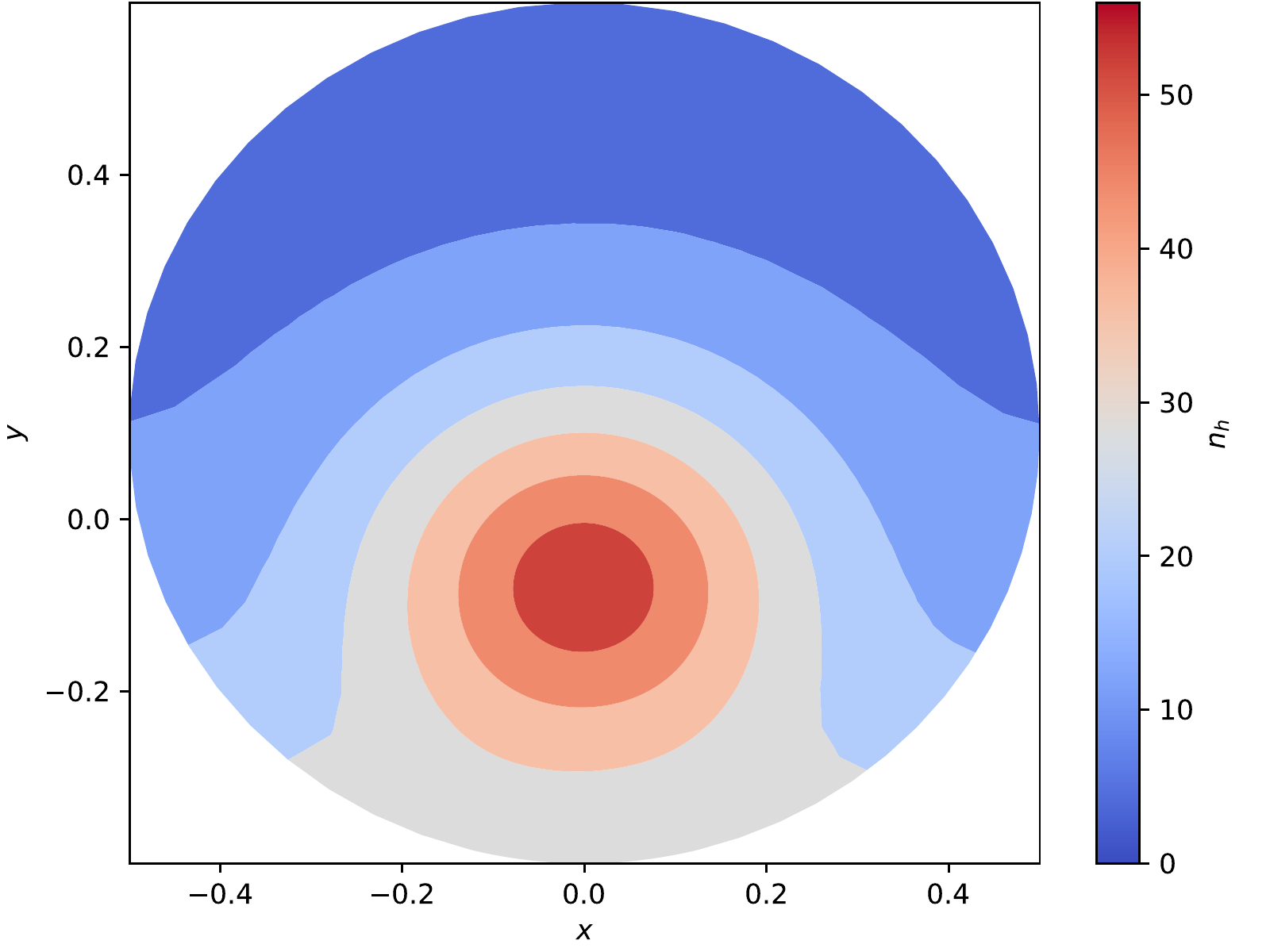}
    \end{subfigure}
    \begin{subfigure}[b]{0.23\textwidth}
        \centering
        \includegraphics[width=1.0\textwidth]{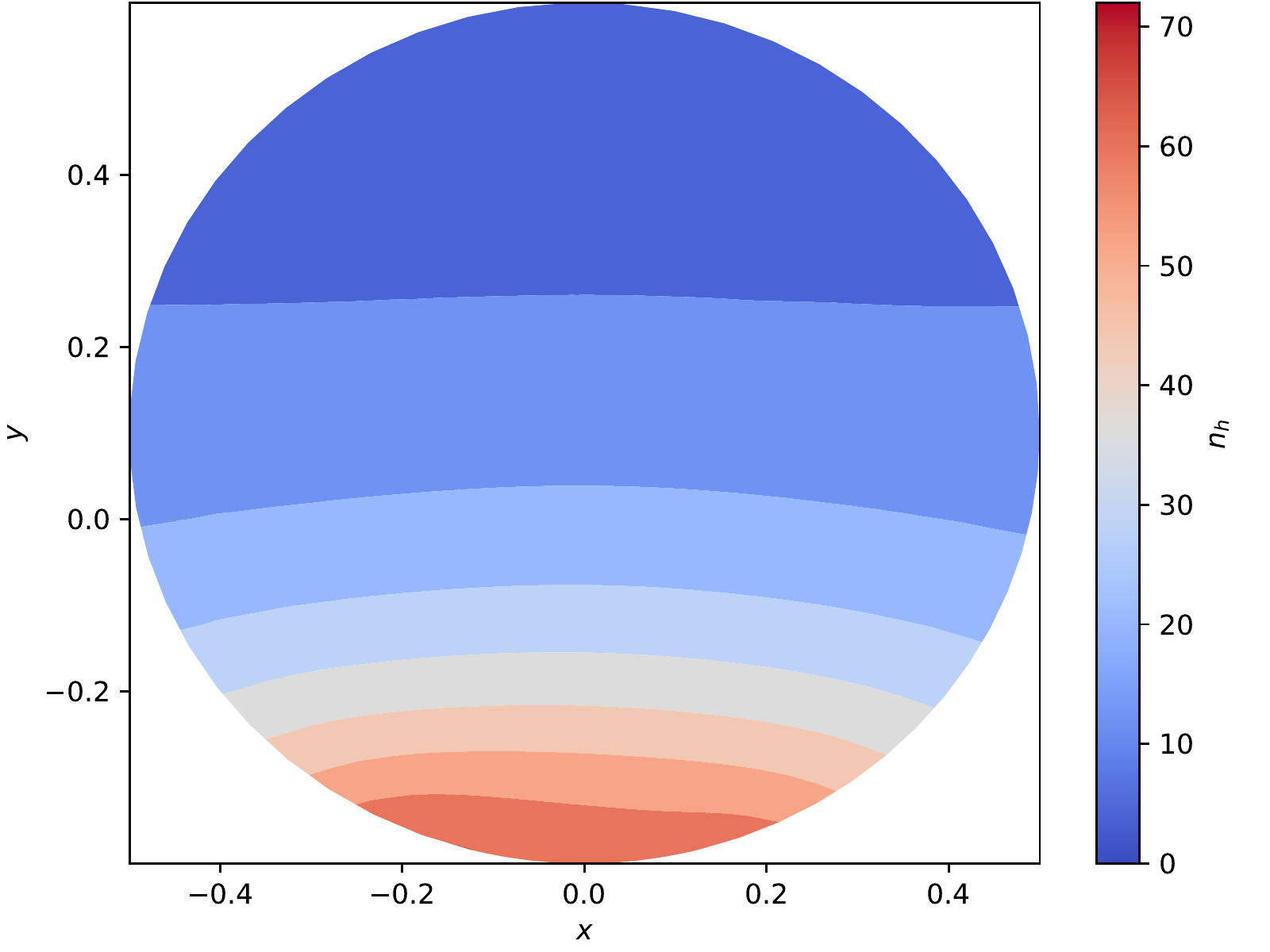}
    \end{subfigure}
    \begin{subfigure}[b]{0.23\textwidth}
        \centering
        \includegraphics[width=1.0\textwidth]{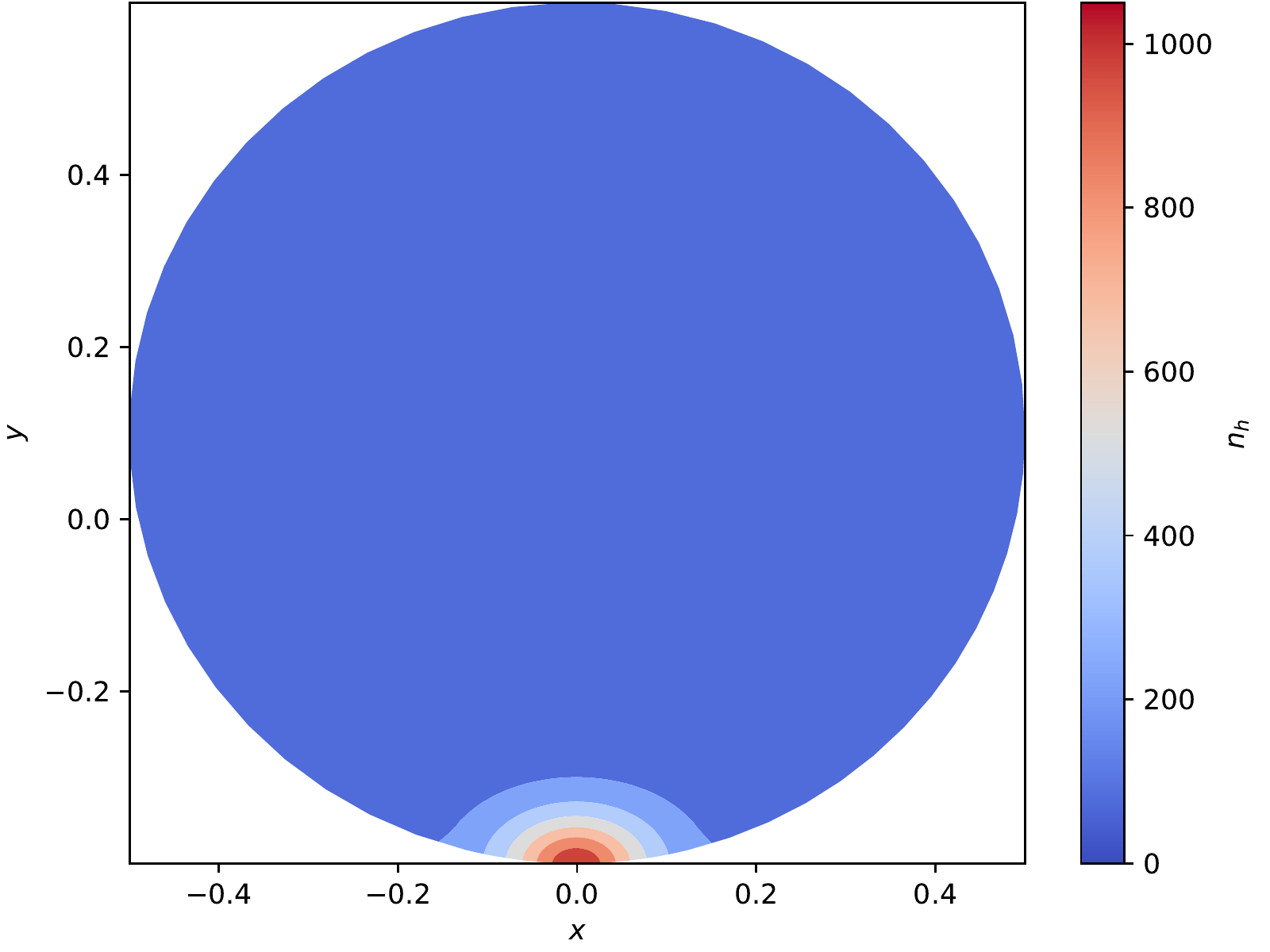}
    \end{subfigure}
    \begin{subfigure}[b]{0.23\textwidth}
        \centering
        \includegraphics[width=1.0\textwidth]{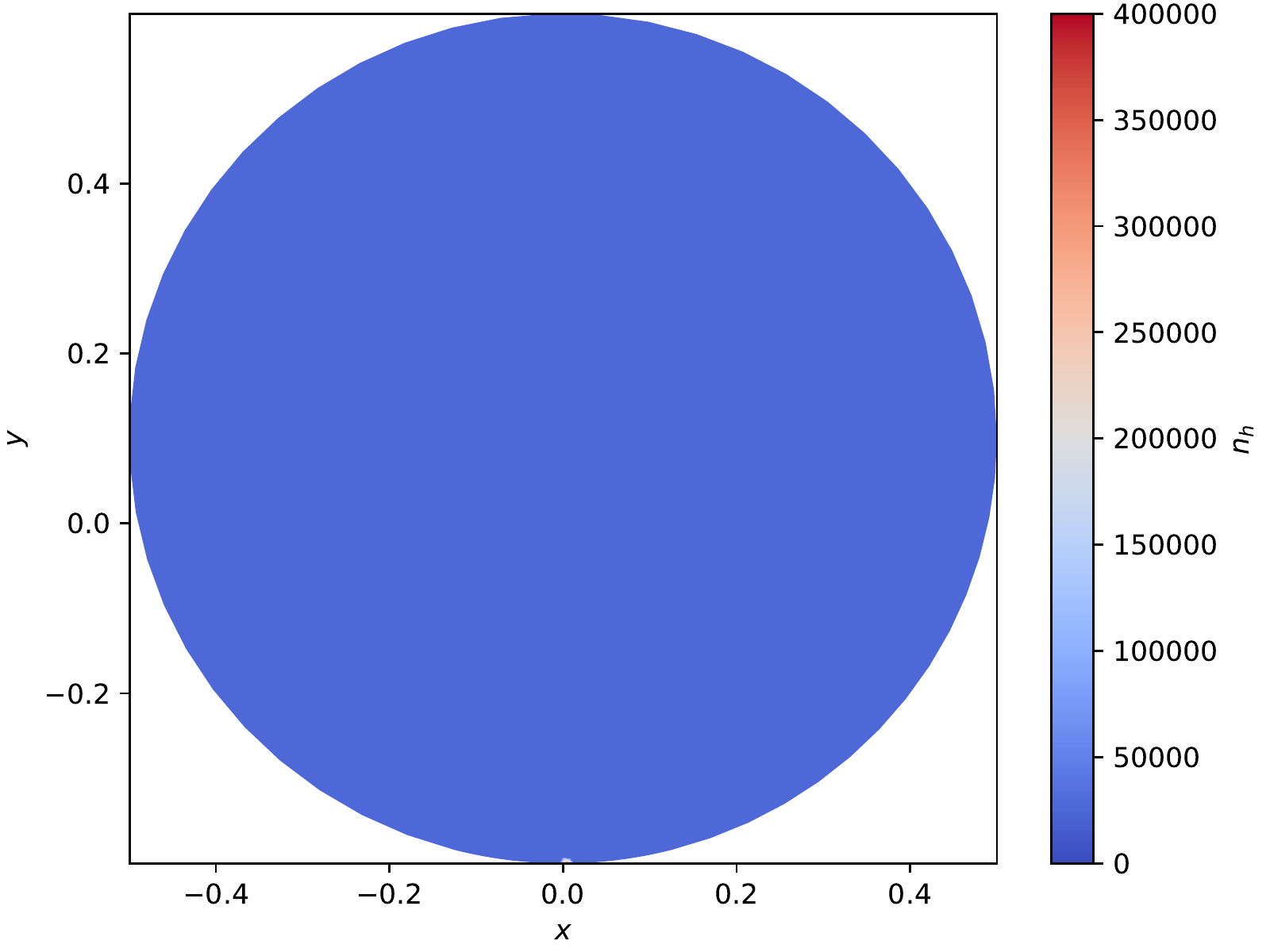}
    \end{subfigure}
            \caption{Snapshots at $t=0.04$, $0.1$, $0.4$ and $0.7$ of $\{n_h^m\}_m$ for $\eta_0=450$}\label{Snapshots_450_finer_nh}
\end{figure}
\begin{figure}
    \begin{subfigure}[b]{0.23\textwidth}
        \centering
        \includegraphics[width=1.0\textwidth]{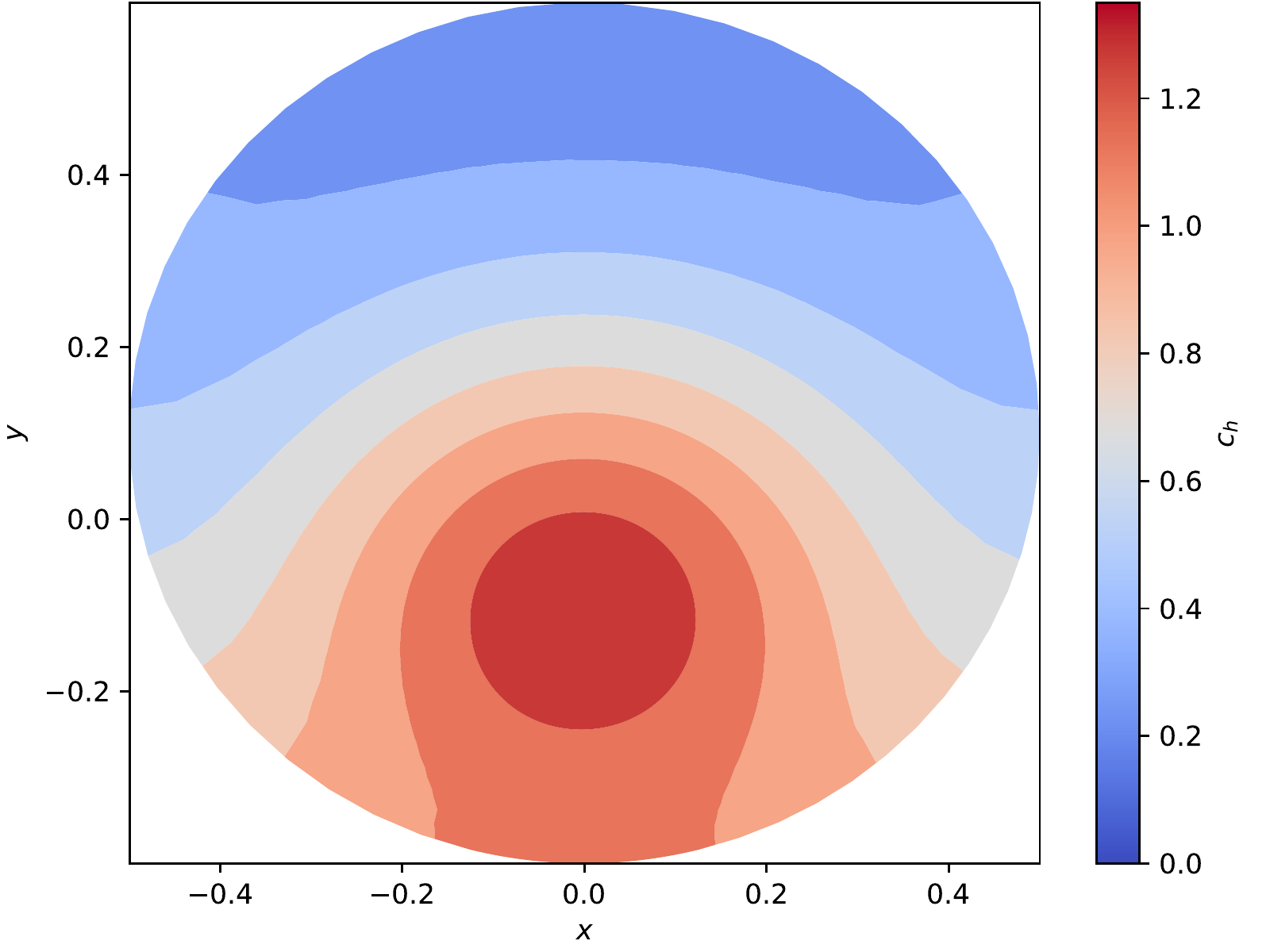}
    \end{subfigure}
    \begin{subfigure}[b]{0.23\textwidth}
        \centering
        \includegraphics[width=1.0\textwidth]{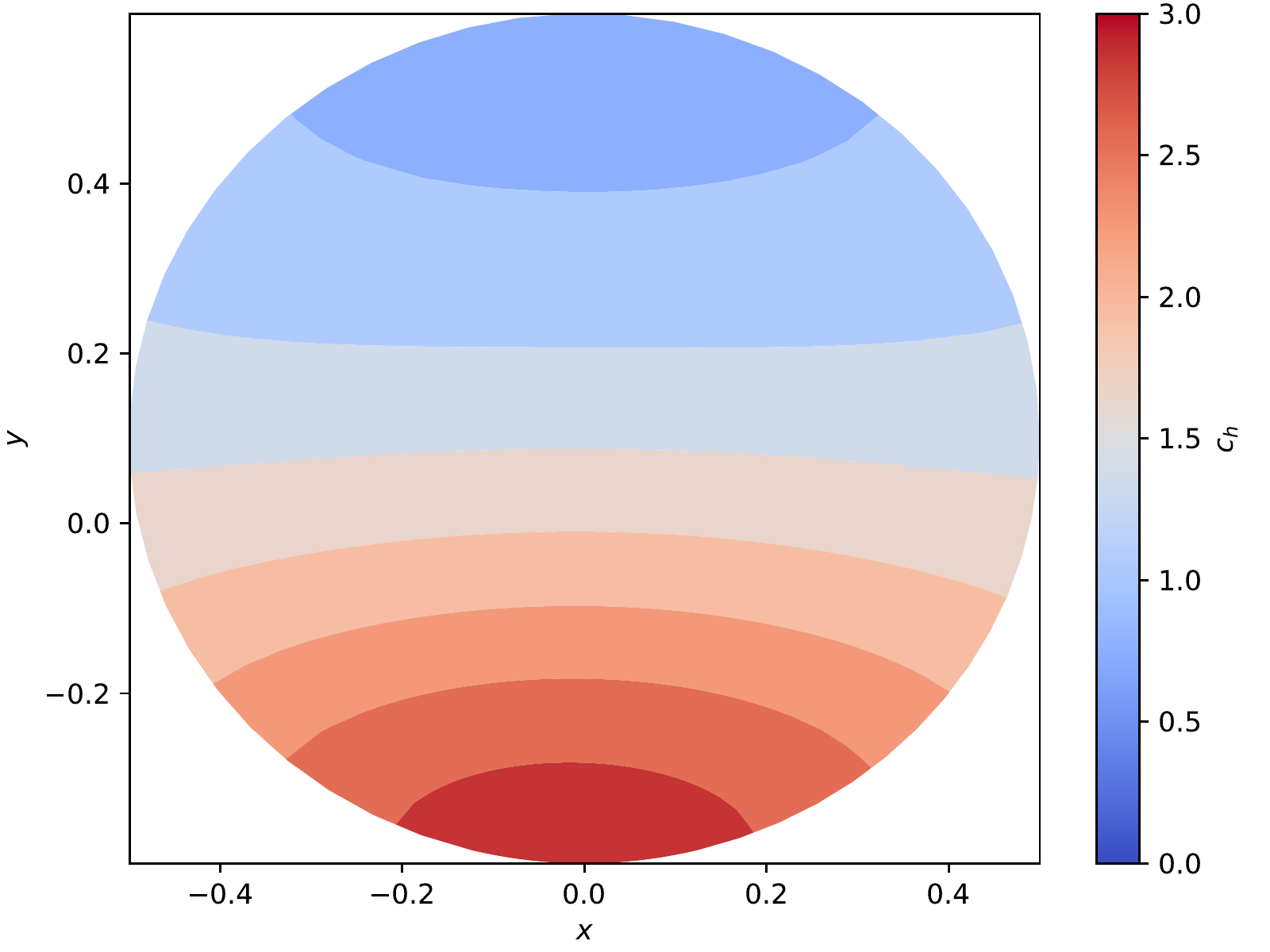}
    \end{subfigure}
    \begin{subfigure}[b]{0.23\textwidth}
        \centering
        \includegraphics[width=1.0\textwidth]{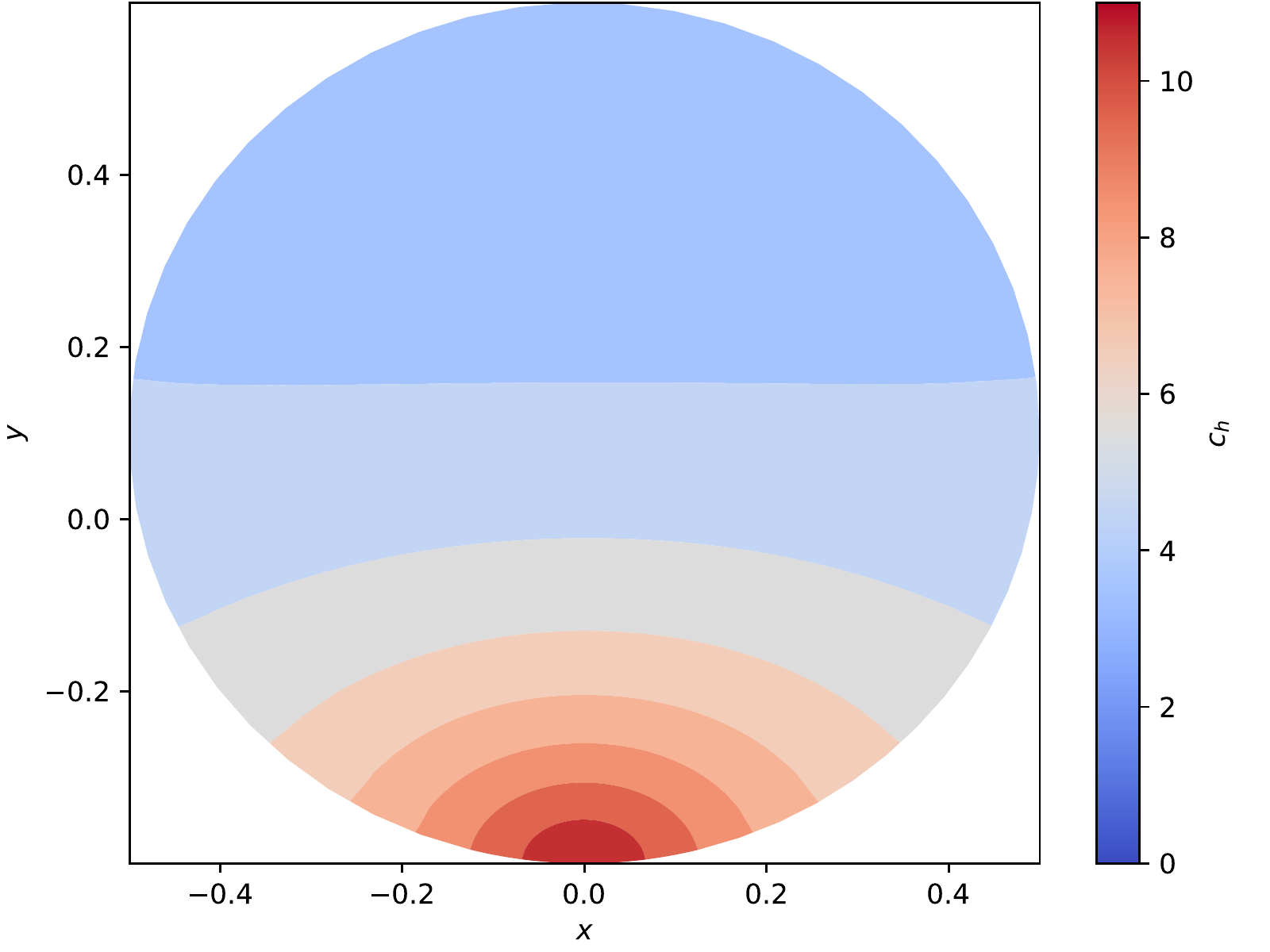}
    \end{subfigure}
    \begin{subfigure}[b]{0.23\textwidth}
        \centering
        \includegraphics[width=1.0\textwidth]{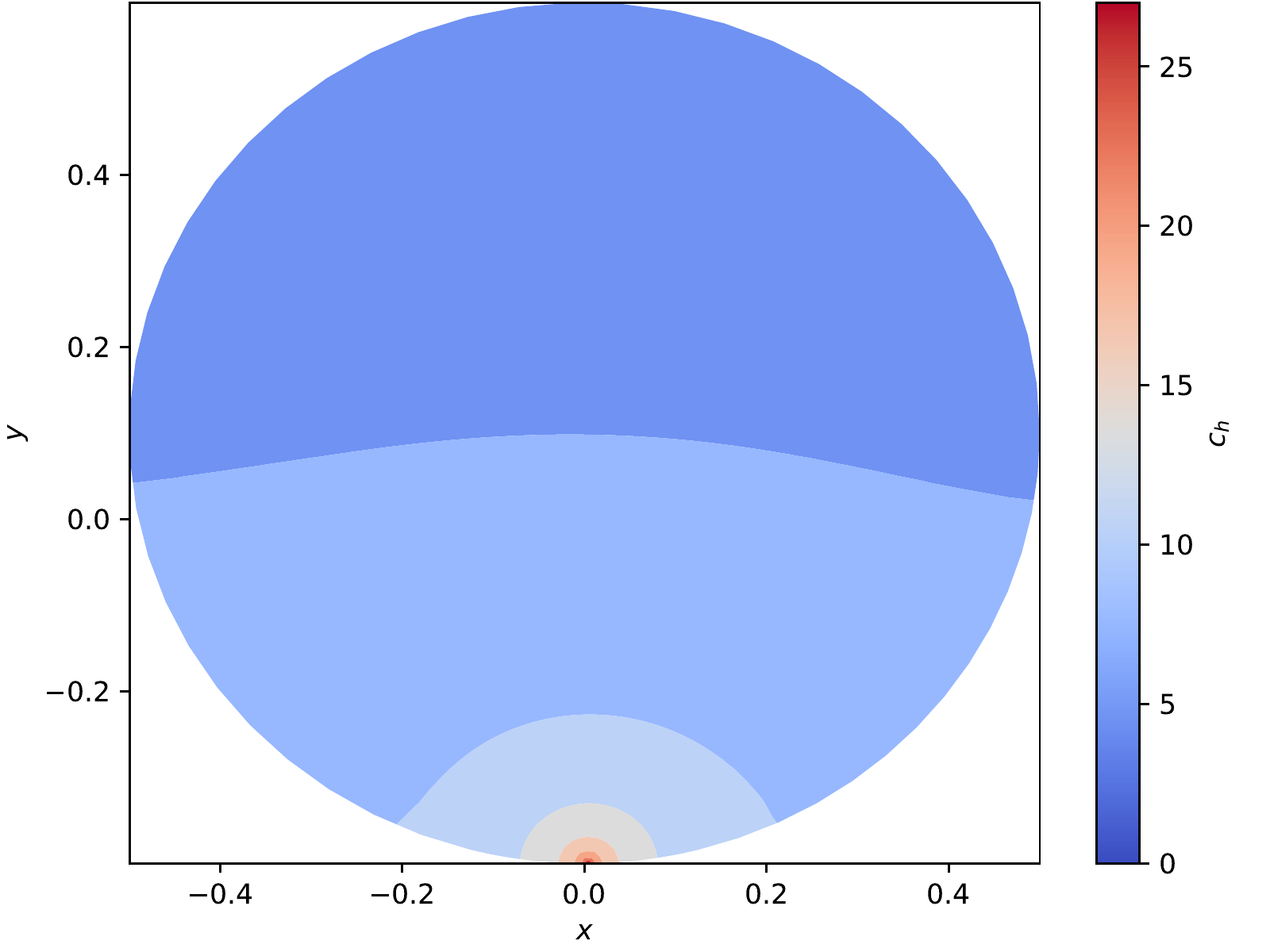}
    \end{subfigure}
\caption{Snapshots at $t=0.04$, $0.1$, $0.4$ and $0.7$ of $\{c_h^m\}_m$ for $\eta_0=450$.}\label{Snapshots_450_finer_ch}
\end{figure}
\begin{figure}
    \begin{subfigure}[b]{0.23\textwidth}
        \centering
        \includegraphics[width=1.0\textwidth]{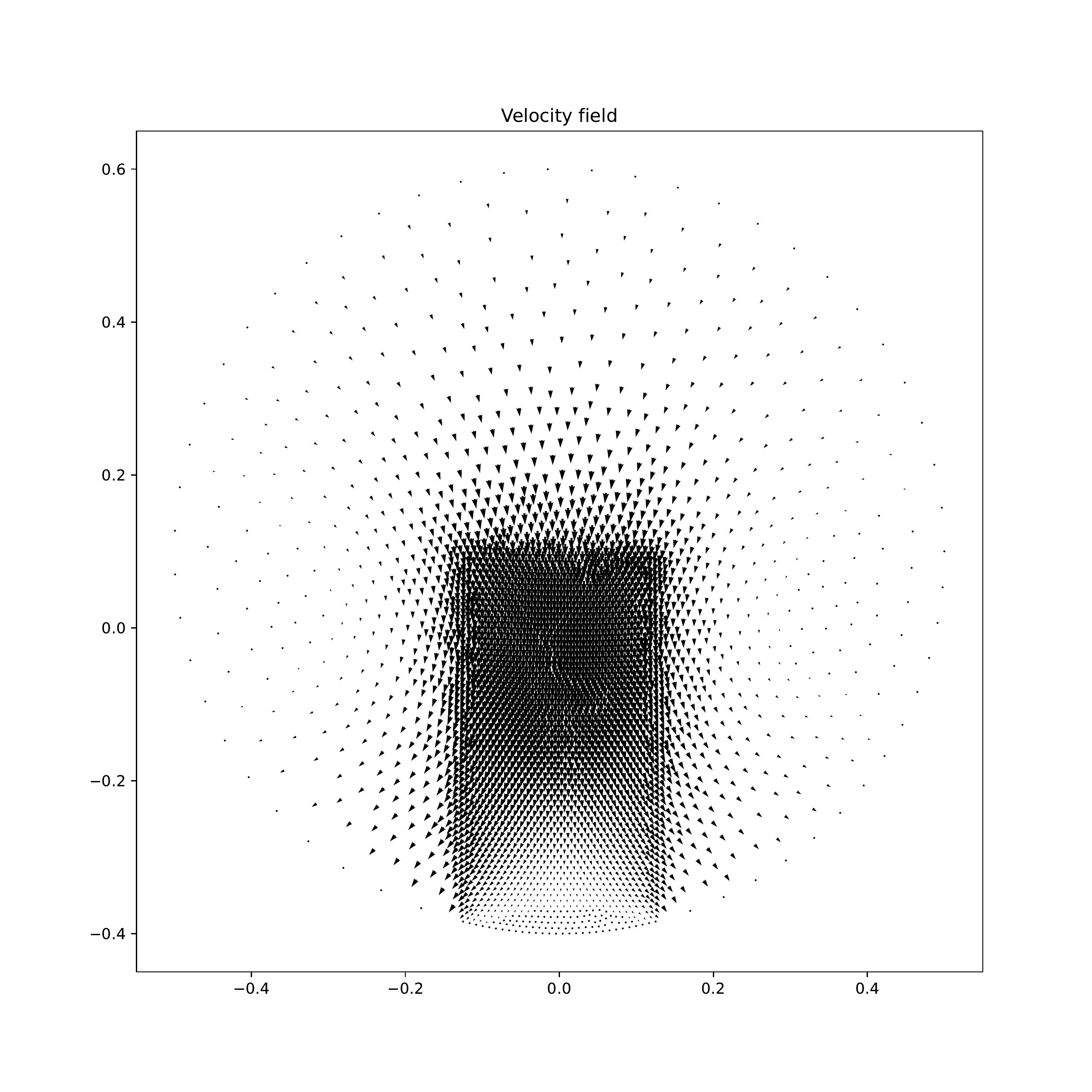}
    \end{subfigure}
    \begin{subfigure}[b]{0.23\textwidth}
        \centering
        \includegraphics[width=1.0\textwidth]{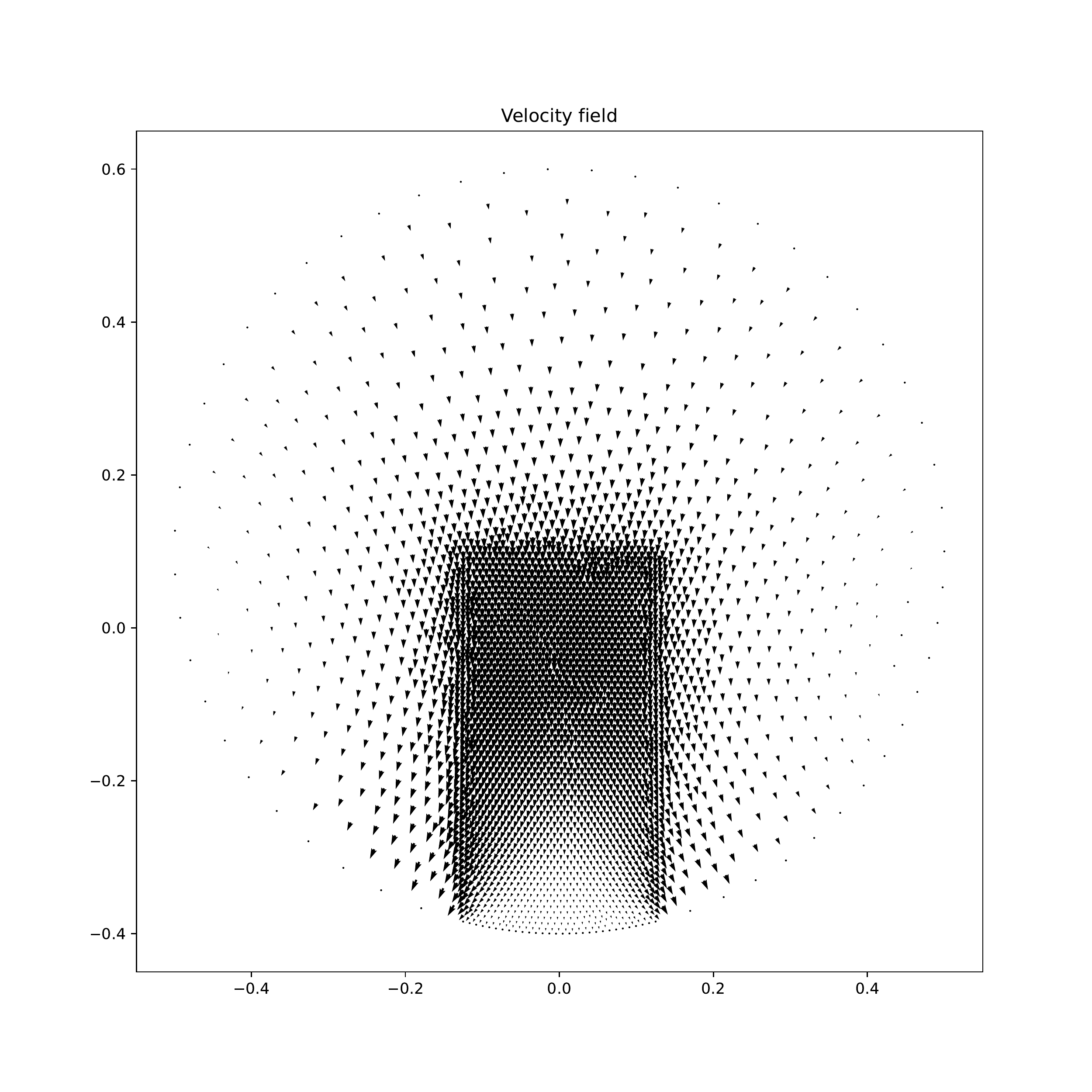}
    \end{subfigure}
    \begin{subfigure}[b]{0.23\textwidth}
        \centering
        \includegraphics[width=1.0\textwidth]{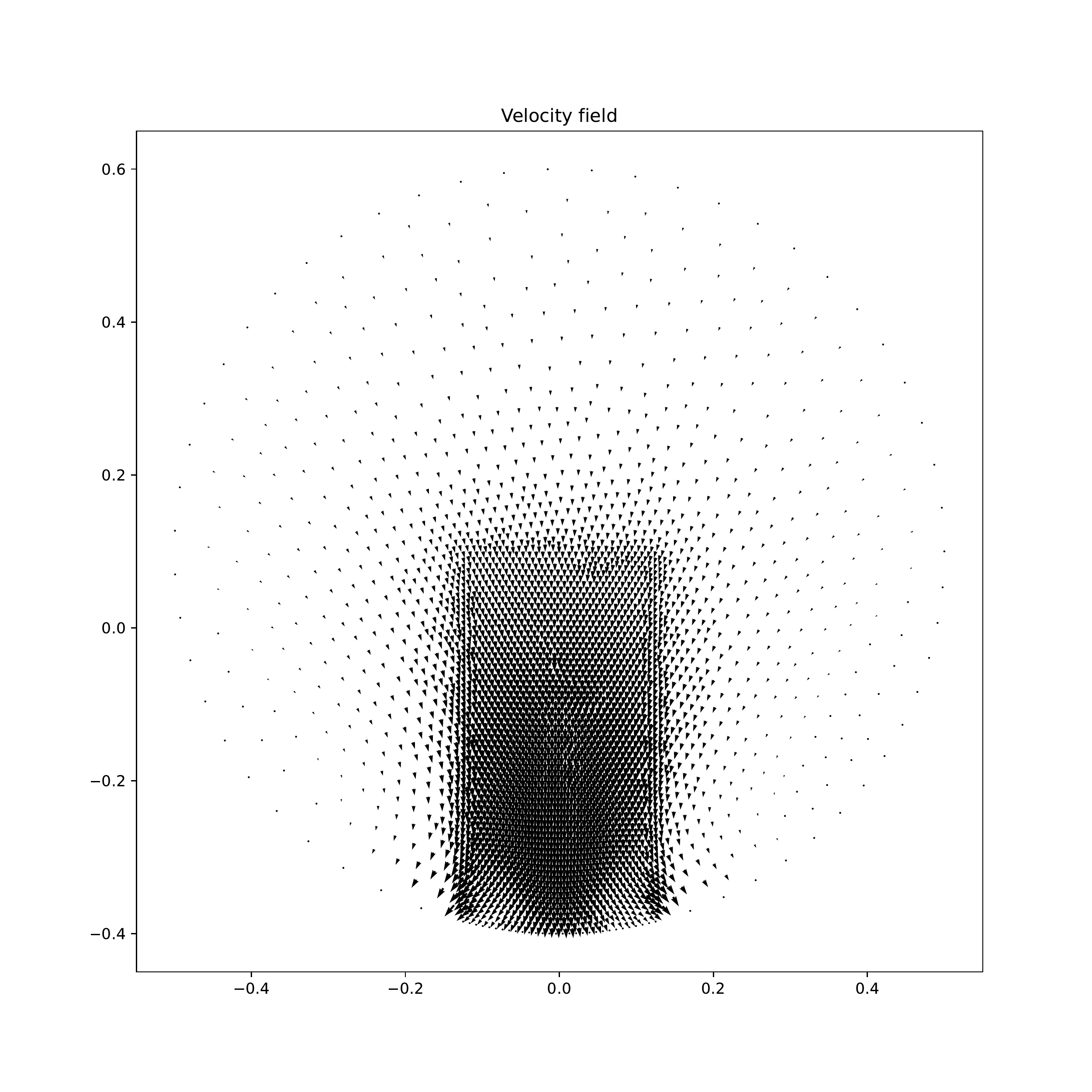}
    \end{subfigure}
    \begin{subfigure}[b]{0.23\textwidth}
        \centering
        \includegraphics[width=1.0\textwidth]{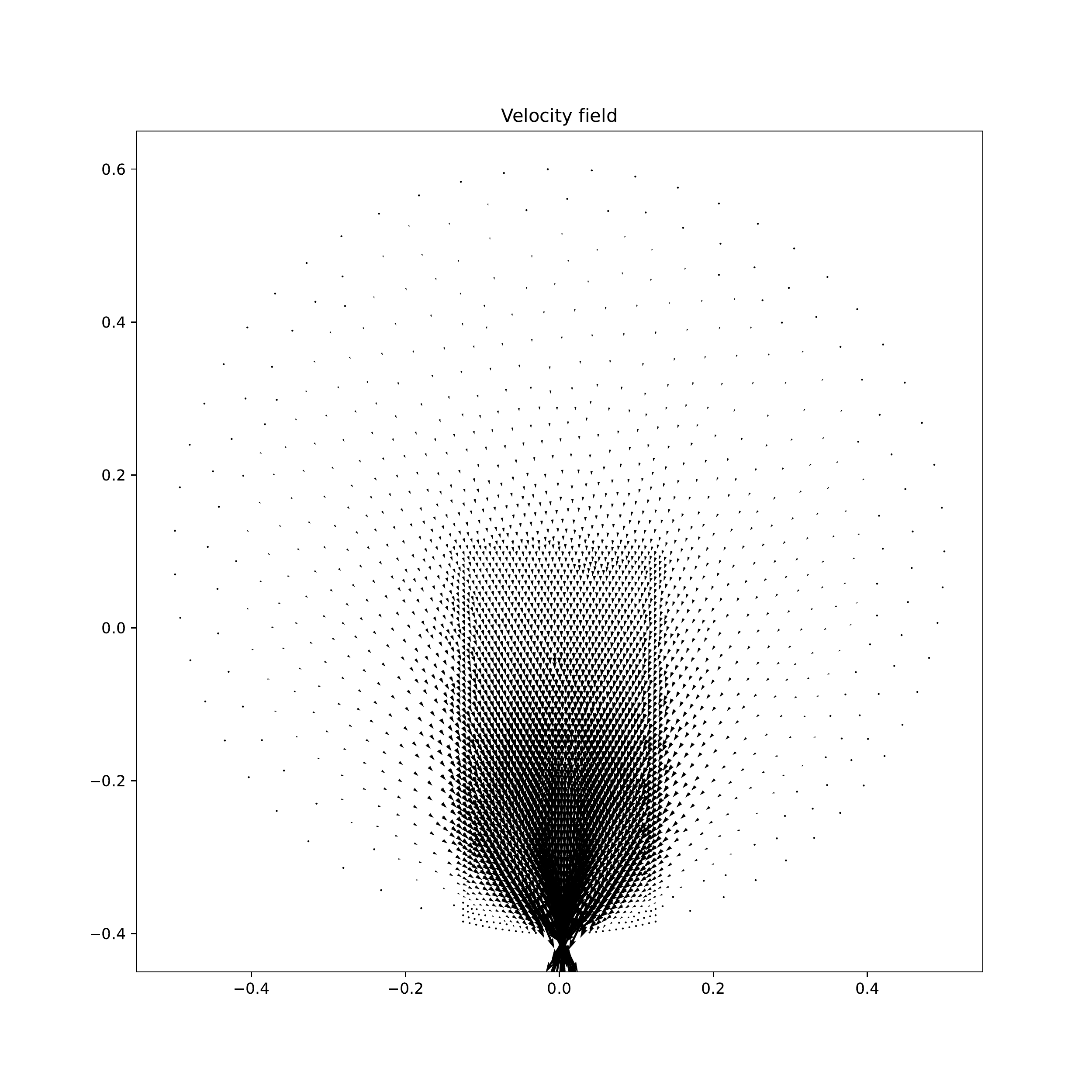}
    \end{subfigure}
    \caption{Snapshots at $t=0.04$, $0.1$, $0.4$ and $0.7$ of $\{\u_h^m\}_m$ for $\eta_0=450$.}\label{Snapshots_450_finer_uh}
\end{figure}

\begin{figure}
    \begin{subfigure}[b]{0.23\textwidth}
        \centering
        \includegraphics[width=1.0\textwidth]{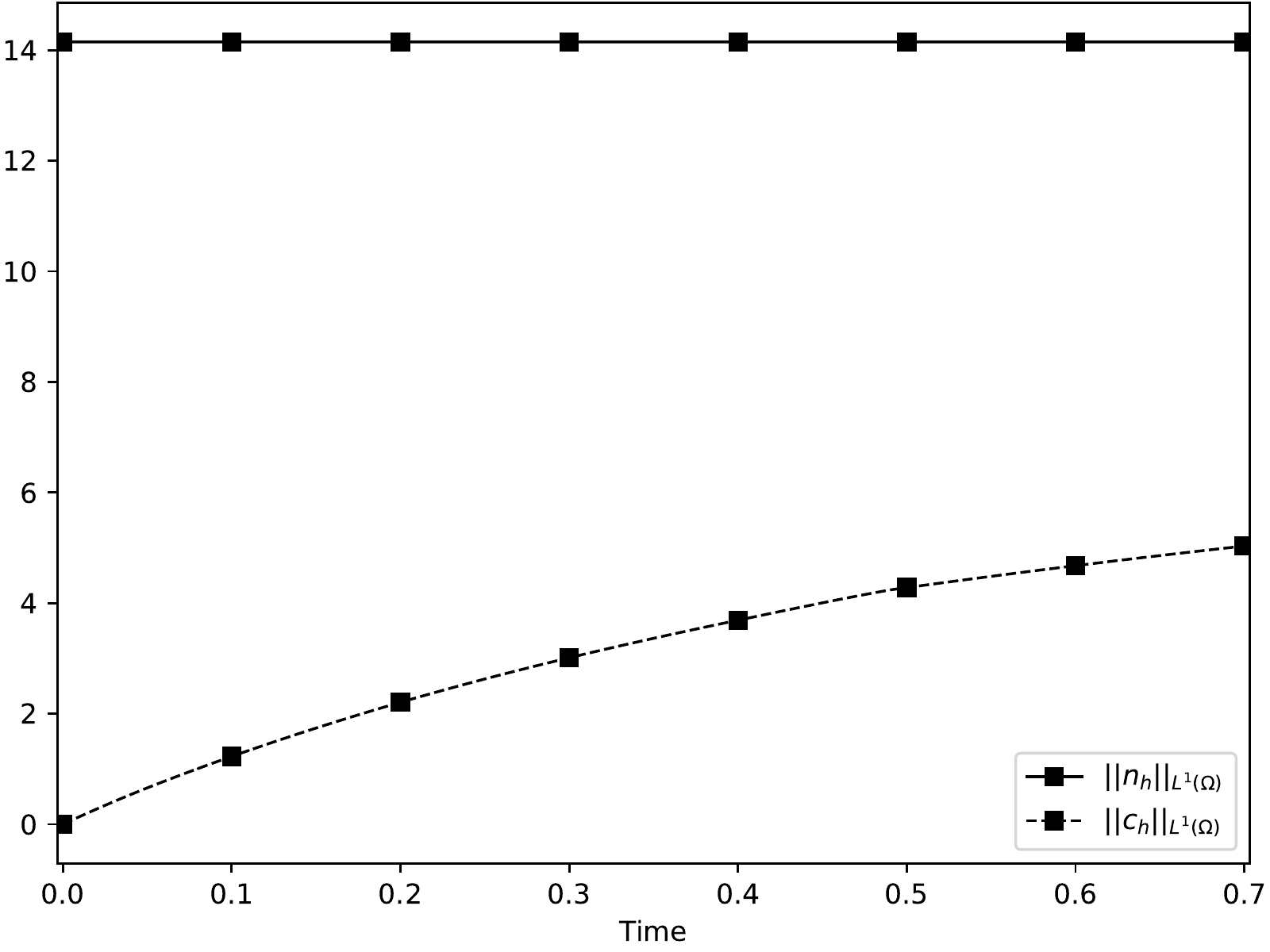}
    \end{subfigure}
    \begin{subfigure}[b]{0.23\textwidth}
        \centering
        \includegraphics[width=1.0\textwidth]{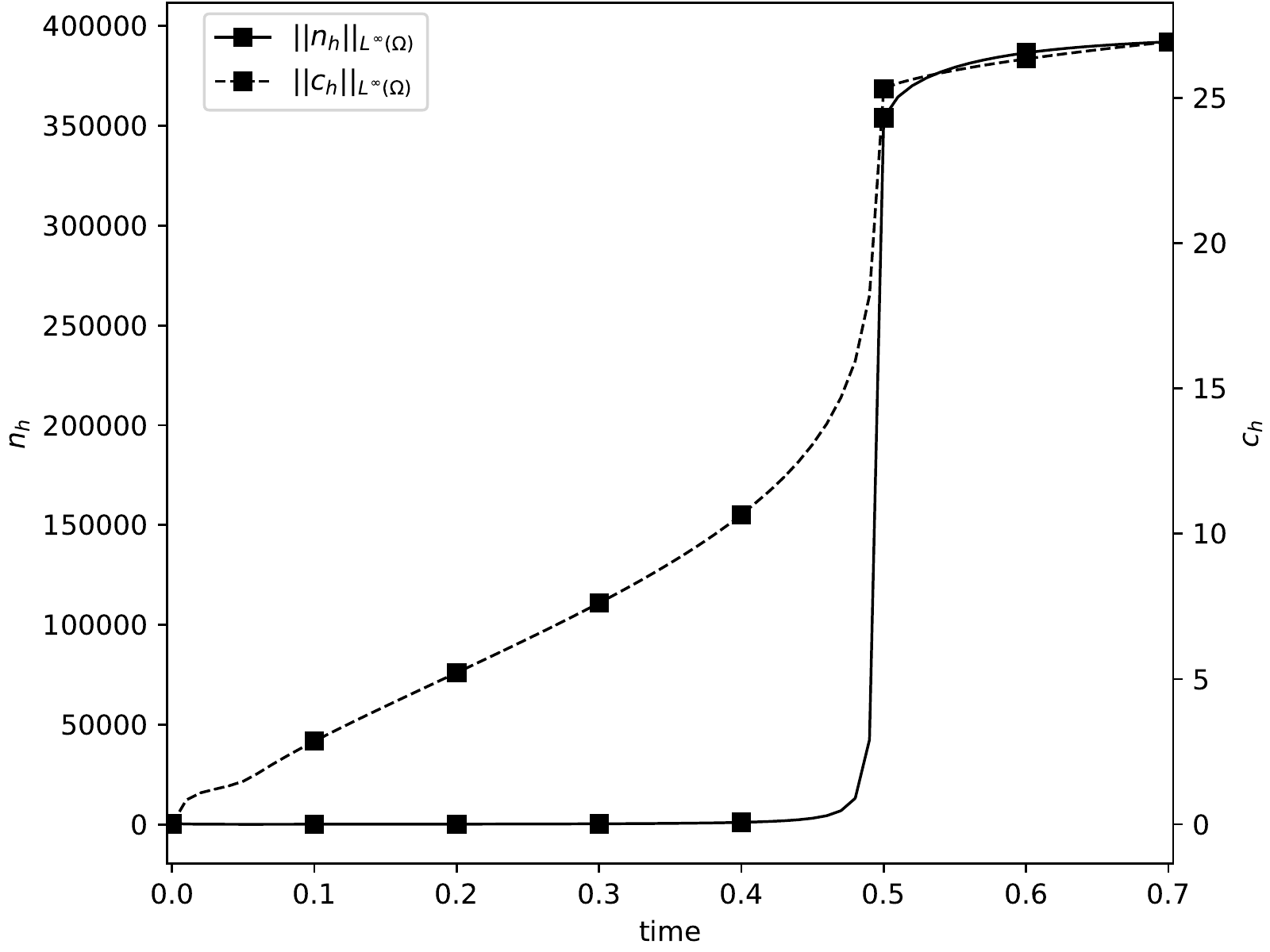}
    \end{subfigure}
    \begin{subfigure}[b]{0.23\textwidth}
        \centering
        \includegraphics[width=1.0\textwidth]{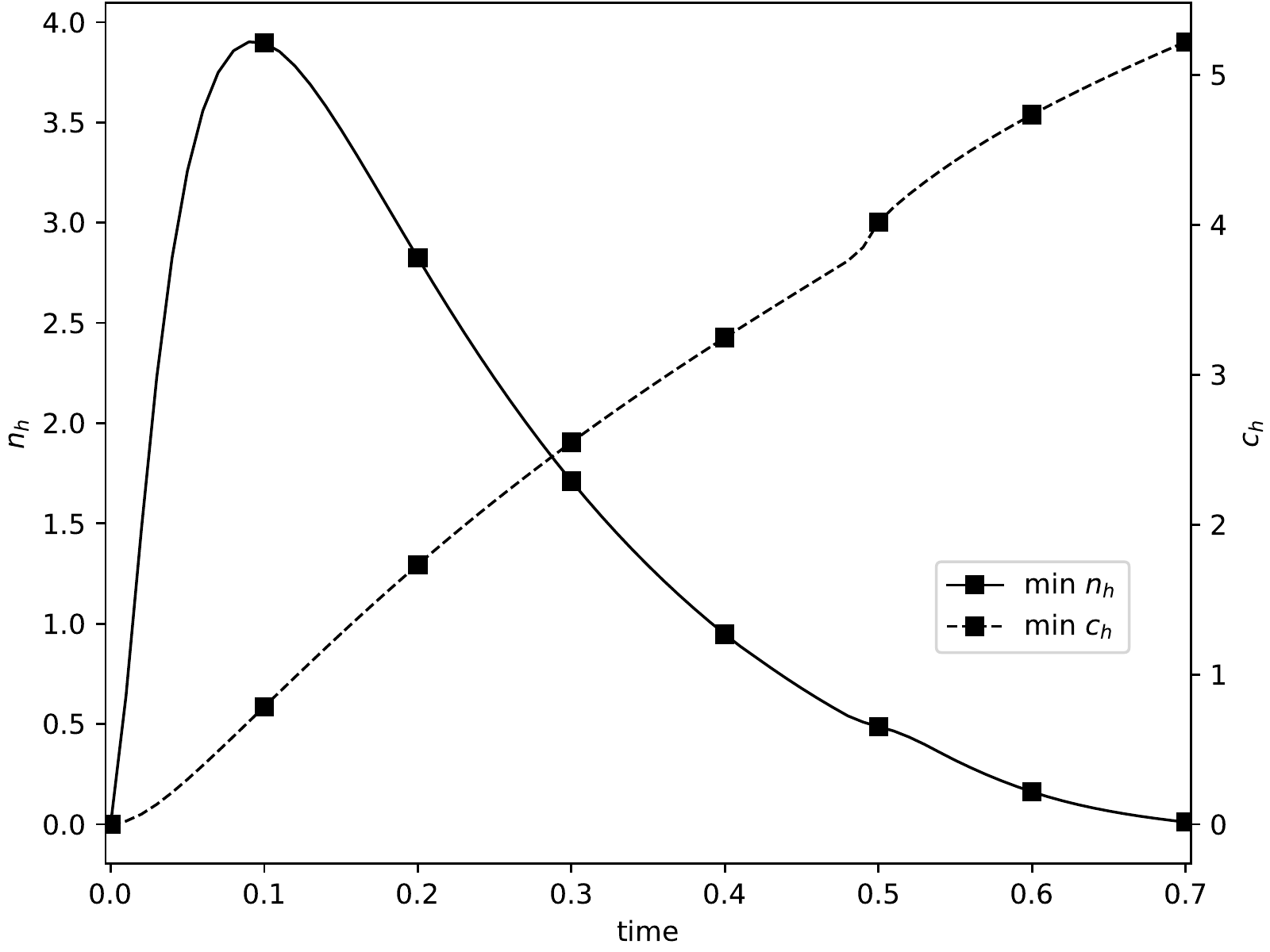}
    \end{subfigure}
            \caption{Plots of total mass, maxima and minima  of $\{n^m_h\}_m$ and $\{c_h^m\}_m$ for $\eta_0=450$.}\label{Graphs_450_finer}
\end{figure}
\subsection{Case: \texorpdfstring{$\eta_0=450$}{Lg} and \texorpdfstring{$\Phi_0=50$}{Lg}}
Now that it is known that there is a finite-time singularity at least numerically for $\eta_0=450$. We ask ourselves whether or not the fluid flow may modify this configuration. In doing so, we take $\Phi_0=50$ to speed up the fluid velocity. Surprisingly as indicated in Figure \ref{Graphs_450_50} (middle) for maxima of $\{n_h^m\}_m$ there seems that the fluid velocity kills the singularity formation, since they do not grow beyond $1200$. Furthermore, minima of $\{n_h^m\}_m$ in Figure \ref{Graphs_450_50} (right) became constant over time, but far away from $0$. For this reason chemotaxis mechanism cannot force the dredging at a single point as a consequence of the growth in velocity induced by $\Phi_0=50$. In addition, convection introduces diffusion in the system. This phenomenological description is shown in Figures \ref{Snapshots_450_50_nh}, \ref{Snapshots_450_50_ch} and \ref{Snapshots_450_50_uh}. 
\begin{figure}
    \begin{subfigure}[b]{0.23\textwidth}
        \centering
        \includegraphics[width=1.0\textwidth]{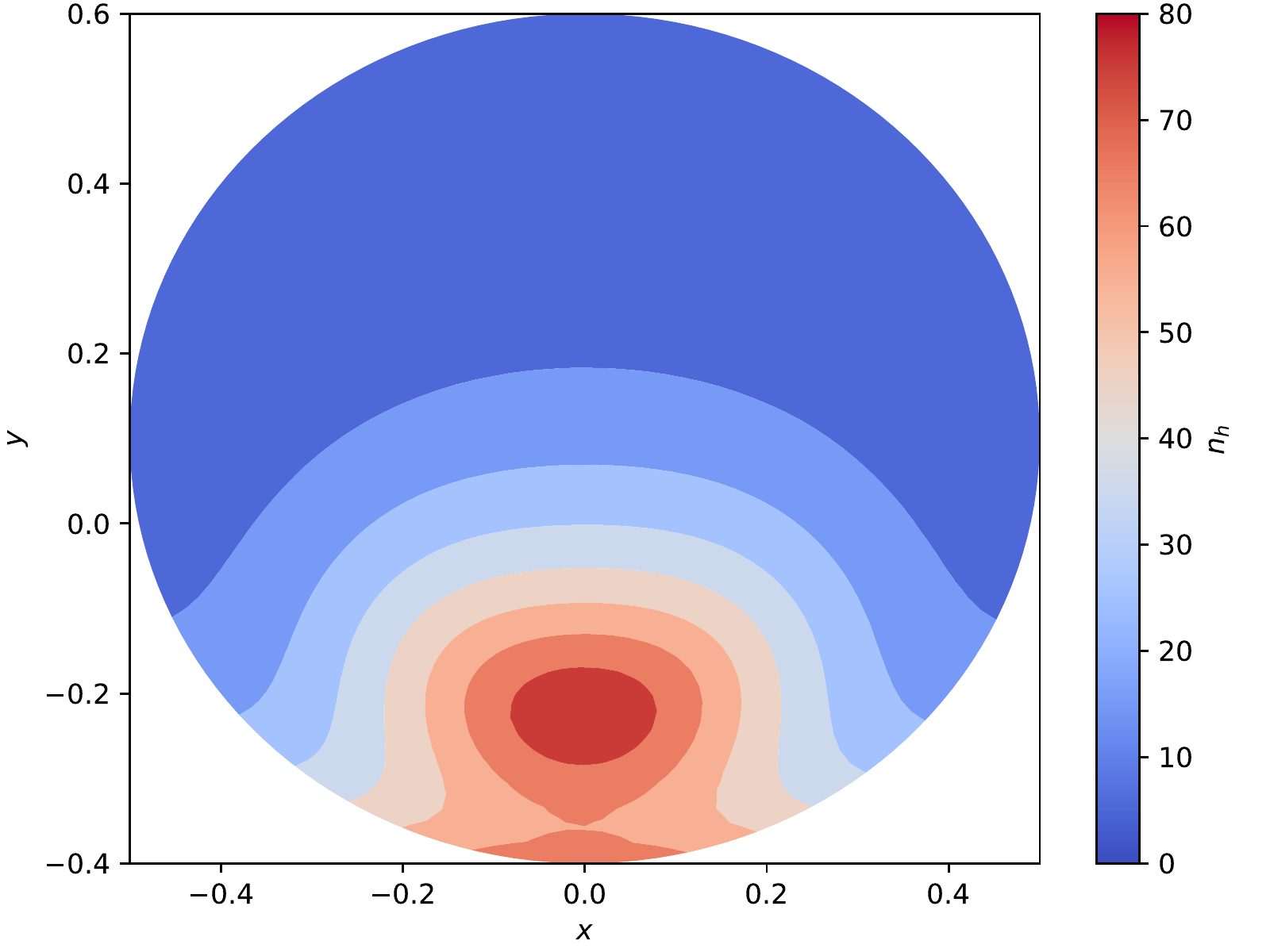}
    \end{subfigure}
    \begin{subfigure}[b]{0.23\textwidth}
        \centering
        \includegraphics[width=1.0\textwidth]{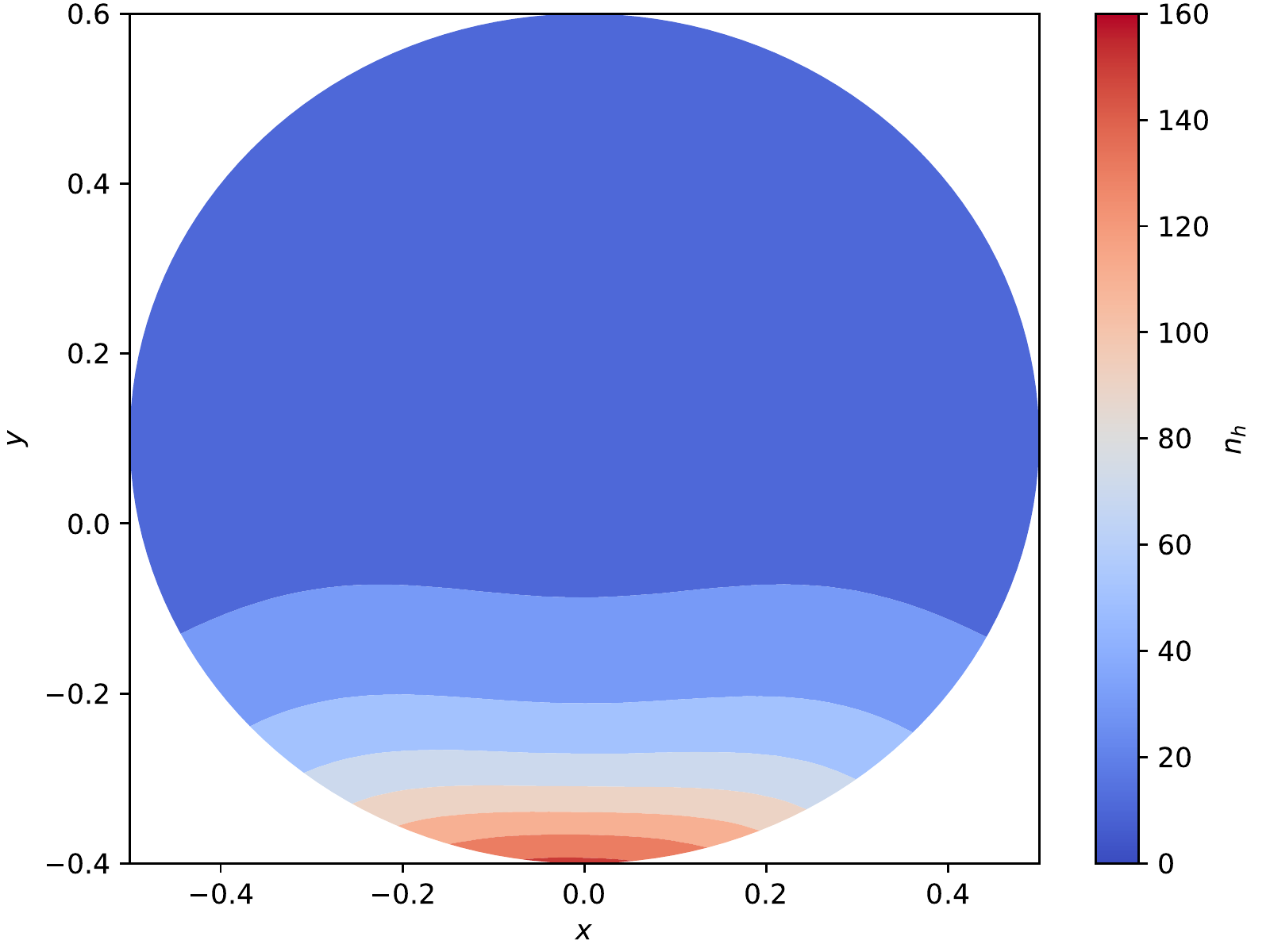}
    \end{subfigure}
    \begin{subfigure}[b]{0.23\textwidth}
        \centering
        \includegraphics[width=1.0\textwidth]{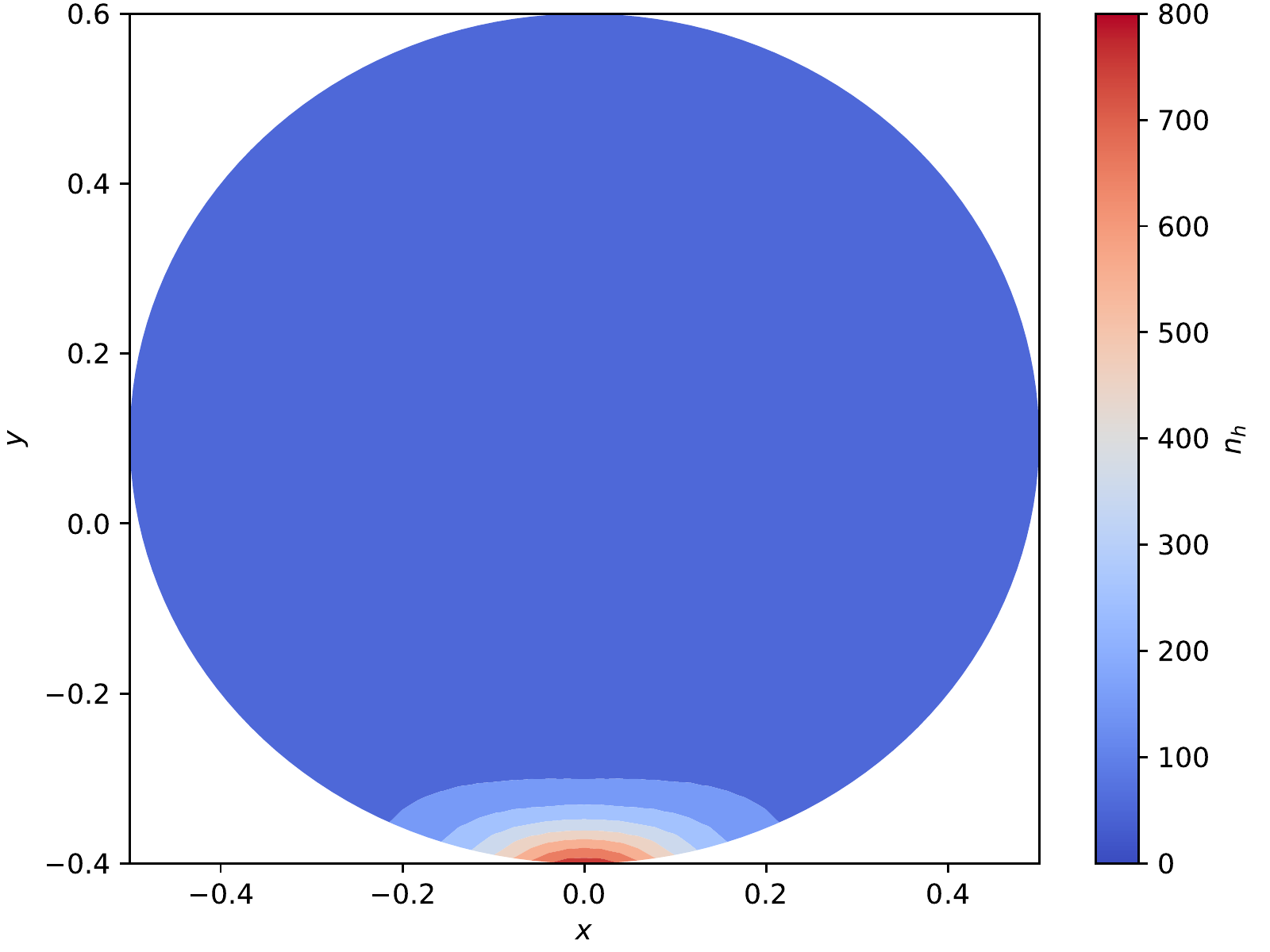}
    \end{subfigure}
    \begin{subfigure}[b]{0.23\textwidth}
        \centering
        \includegraphics[width=1.0\textwidth]{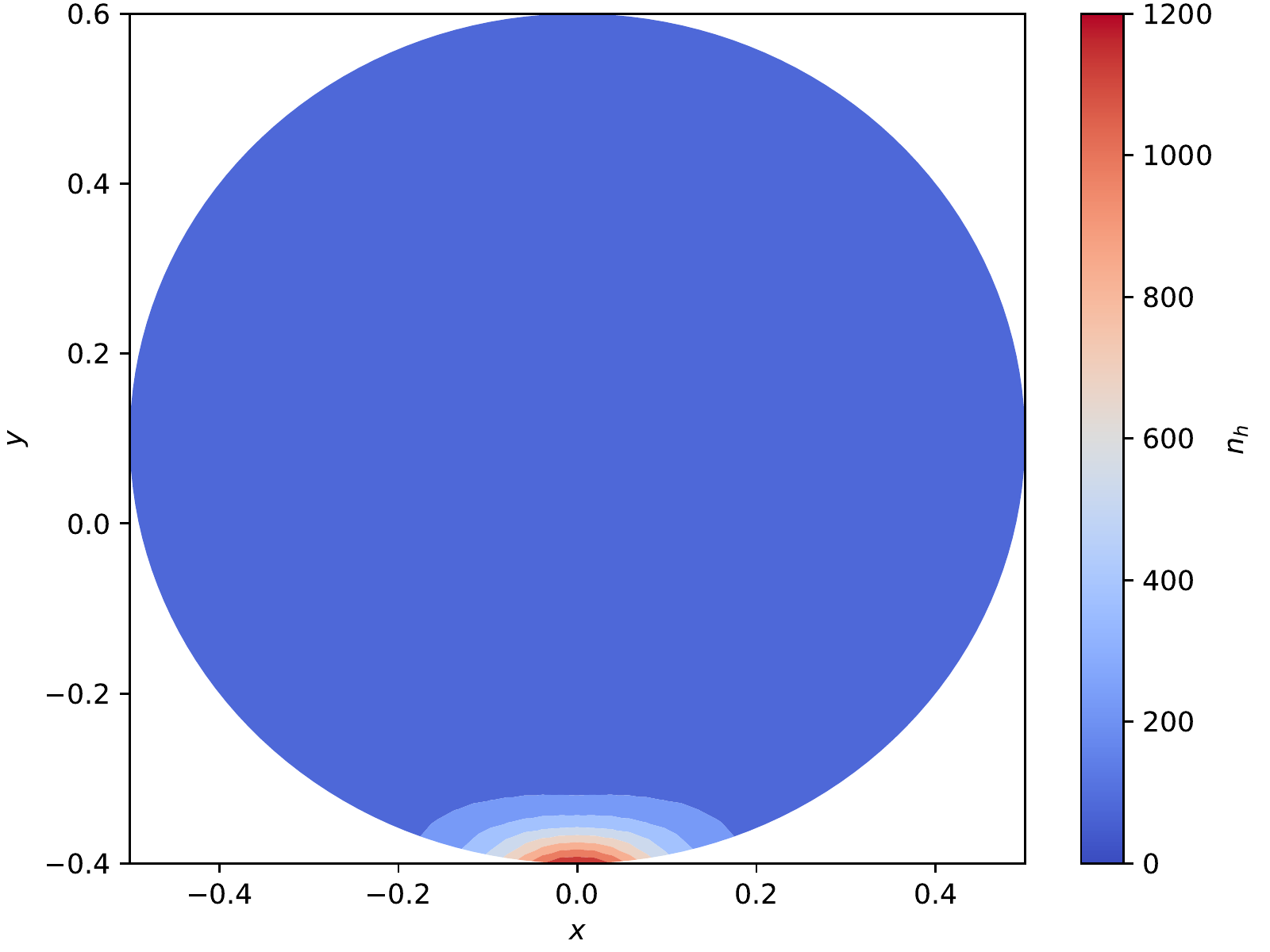}
    \end{subfigure}
            \caption{Snapshots at $t=0.02$, $t=0.04$, $0.2$ and $1.0$ of $\{n_h^m\}_m$ for $\eta_0=450$ and $\Phi_0=50$.}\label{Snapshots_450_50_nh}
            \end{figure}
    \begin{figure}
    \begin{subfigure}[b]{0.23\textwidth}
        \centering
        \includegraphics[width=1.0\textwidth]{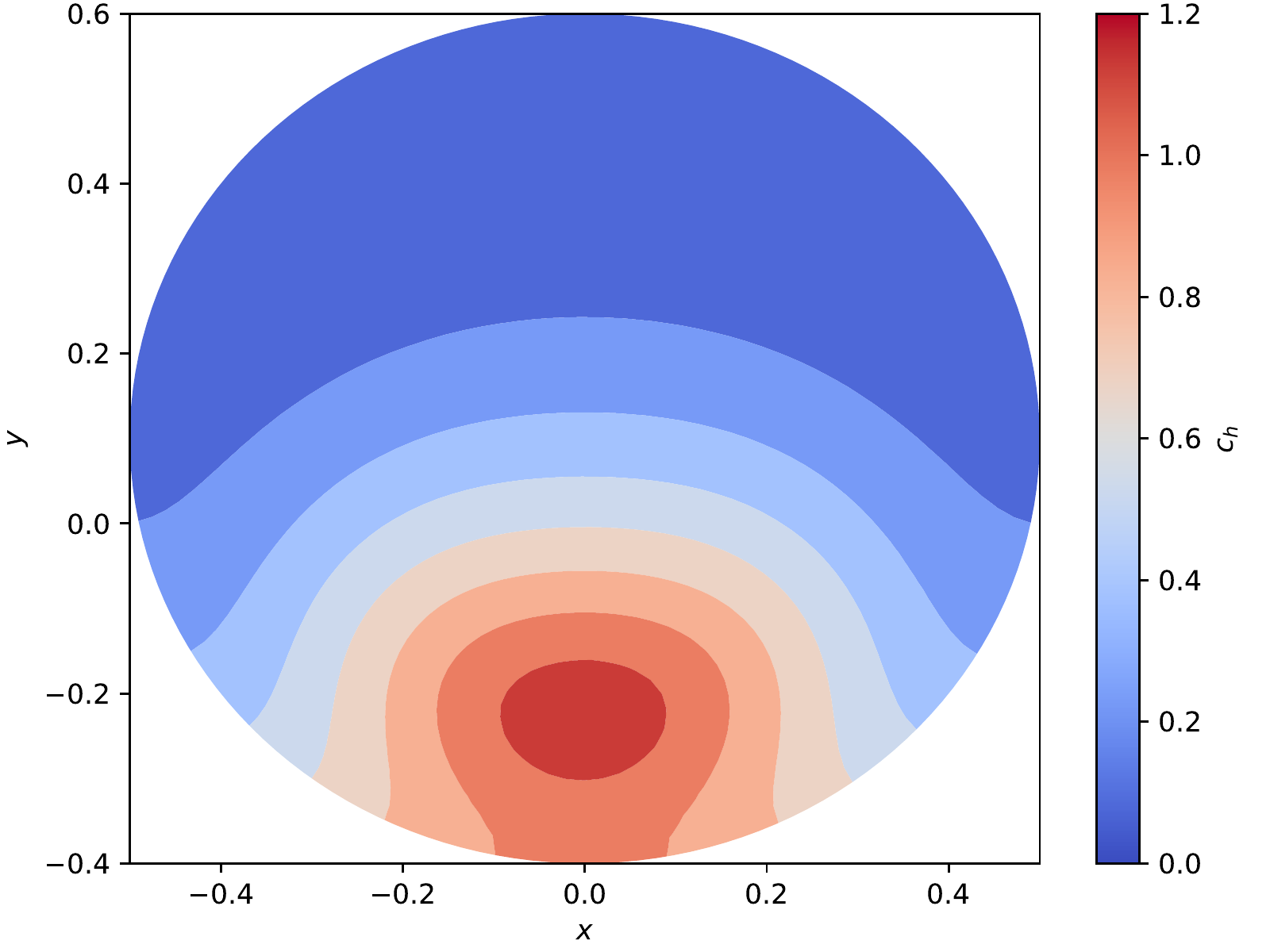}
    \end{subfigure}
    \begin{subfigure}[b]{0.23\textwidth}
        \centering
        \includegraphics[width=1.0\textwidth]{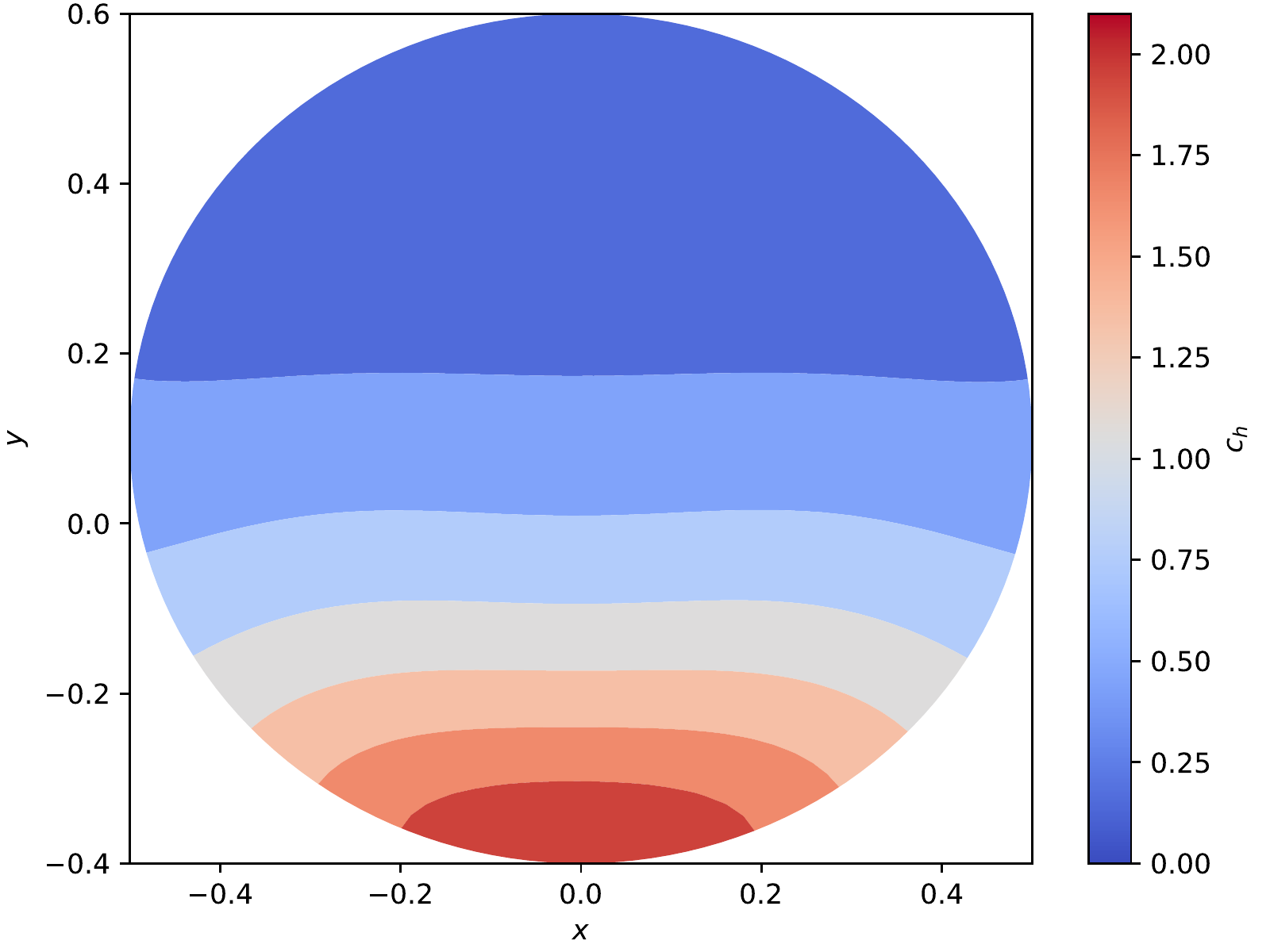}
    \end{subfigure}
    \begin{subfigure}[b]{0.23\textwidth}
        \centering
        \includegraphics[width=1.0\textwidth]{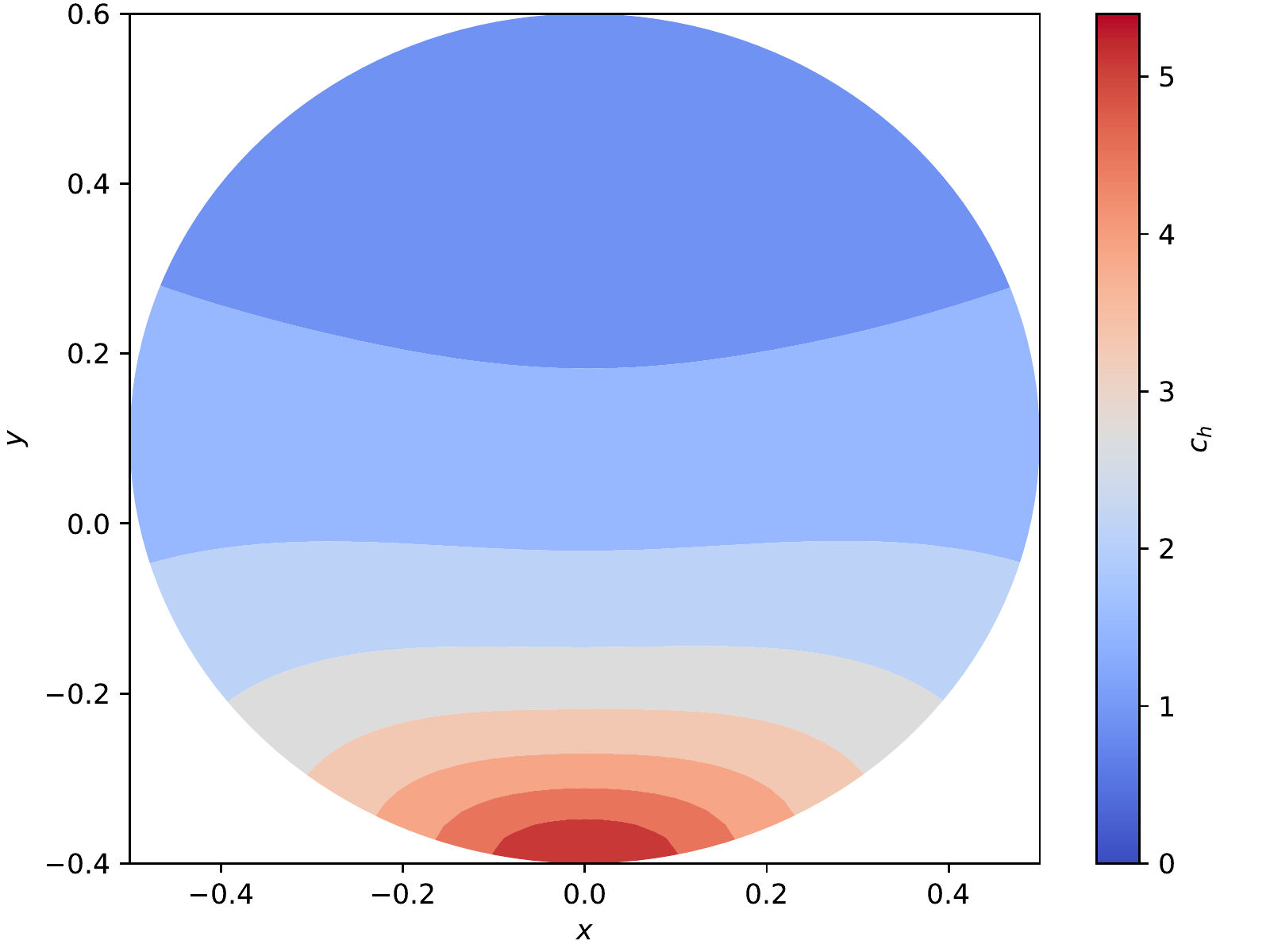}
    \end{subfigure}
    \begin{subfigure}[b]{0.23\textwidth}
        \centering
        \includegraphics[width=1.0\textwidth]{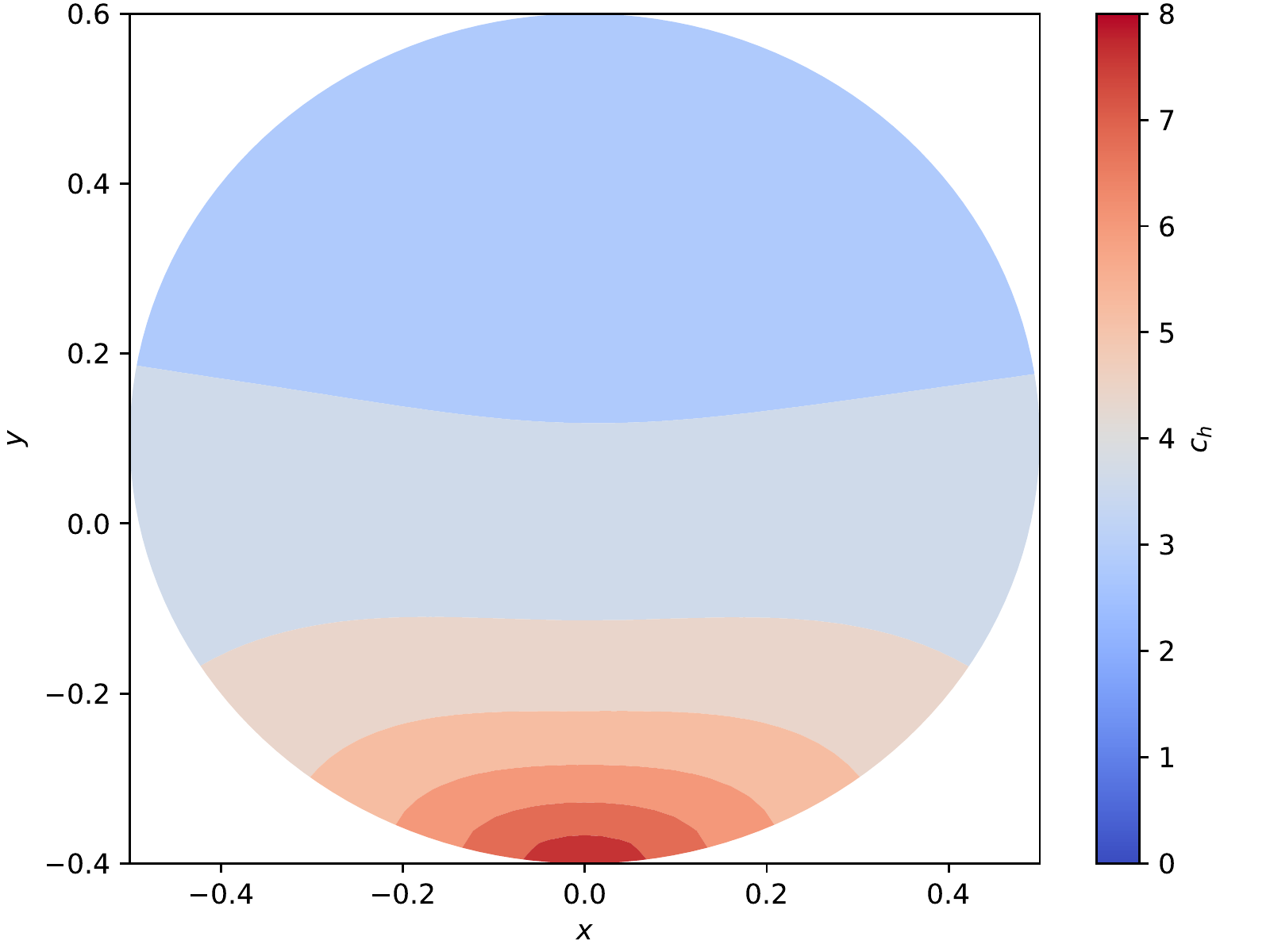}
    \end{subfigure}
    \caption{Snapshots at $t=0.02$, $t=0.04$, $0.2$ and $1.0$ of $\{c_h^m\}_m$ for $\eta_0=450$ and $\Phi_0=50$.}\label{Snapshots_450_50_ch}
            \end{figure}
    \begin{figure}
    \begin{subfigure}[b]{0.23\textwidth}
        \centering
        \includegraphics[width=1.0\textwidth]{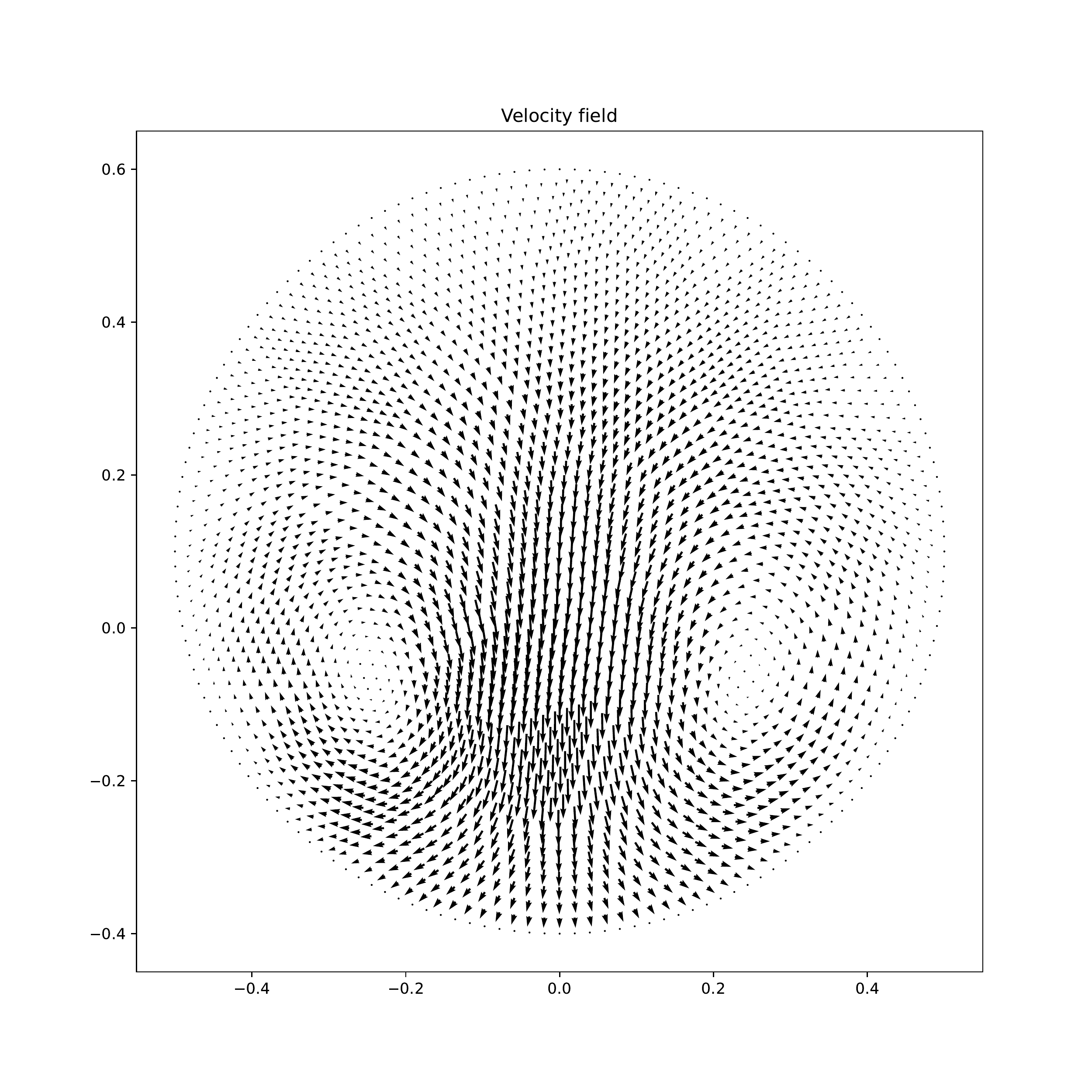}
    \end{subfigure}
    \begin{subfigure}[b]{0.23\textwidth}
        \centering
        \includegraphics[width=1.0\textwidth]{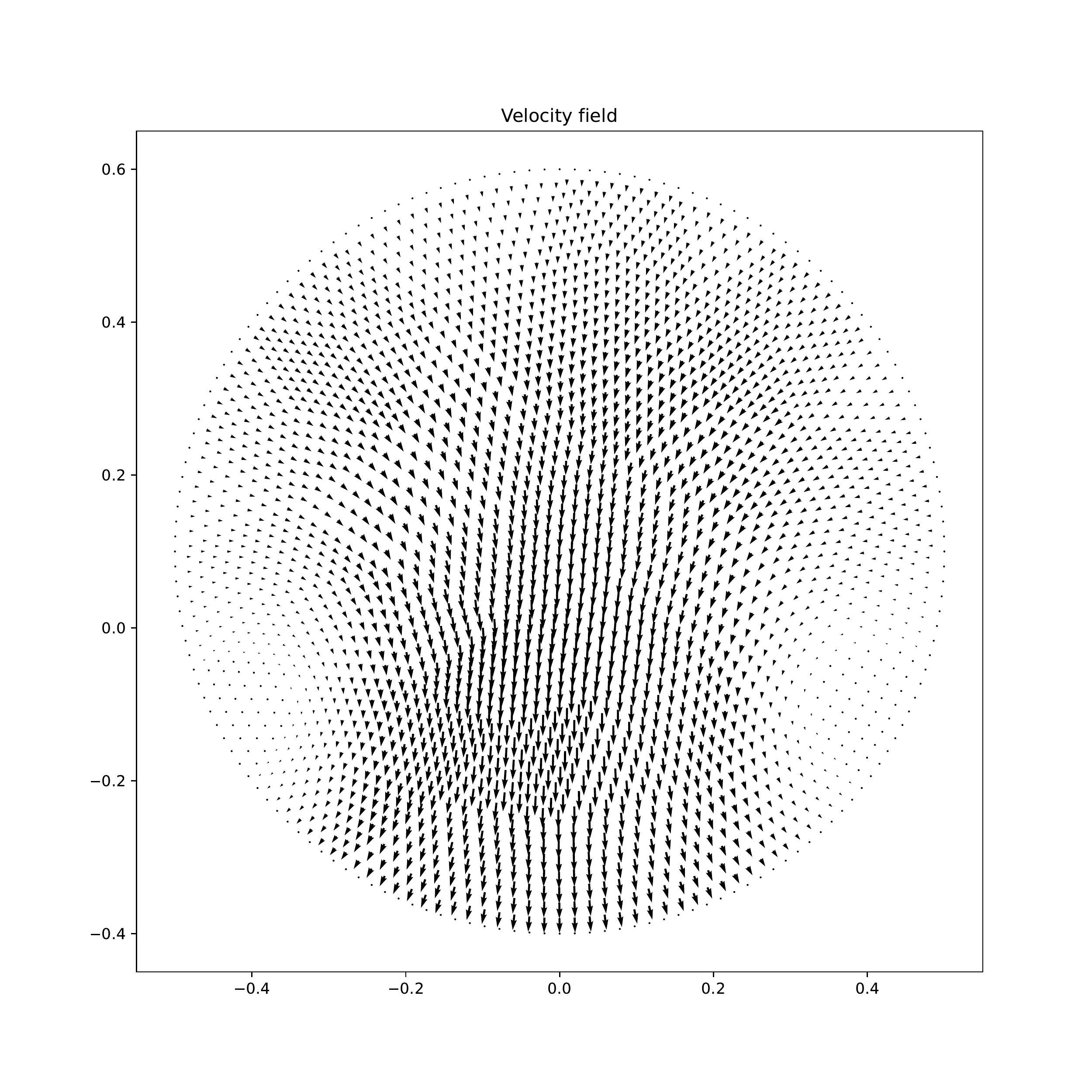}
    \end{subfigure}
    \begin{subfigure}[b]{0.23\textwidth}
        \centering
        \includegraphics[width=1.0\textwidth]{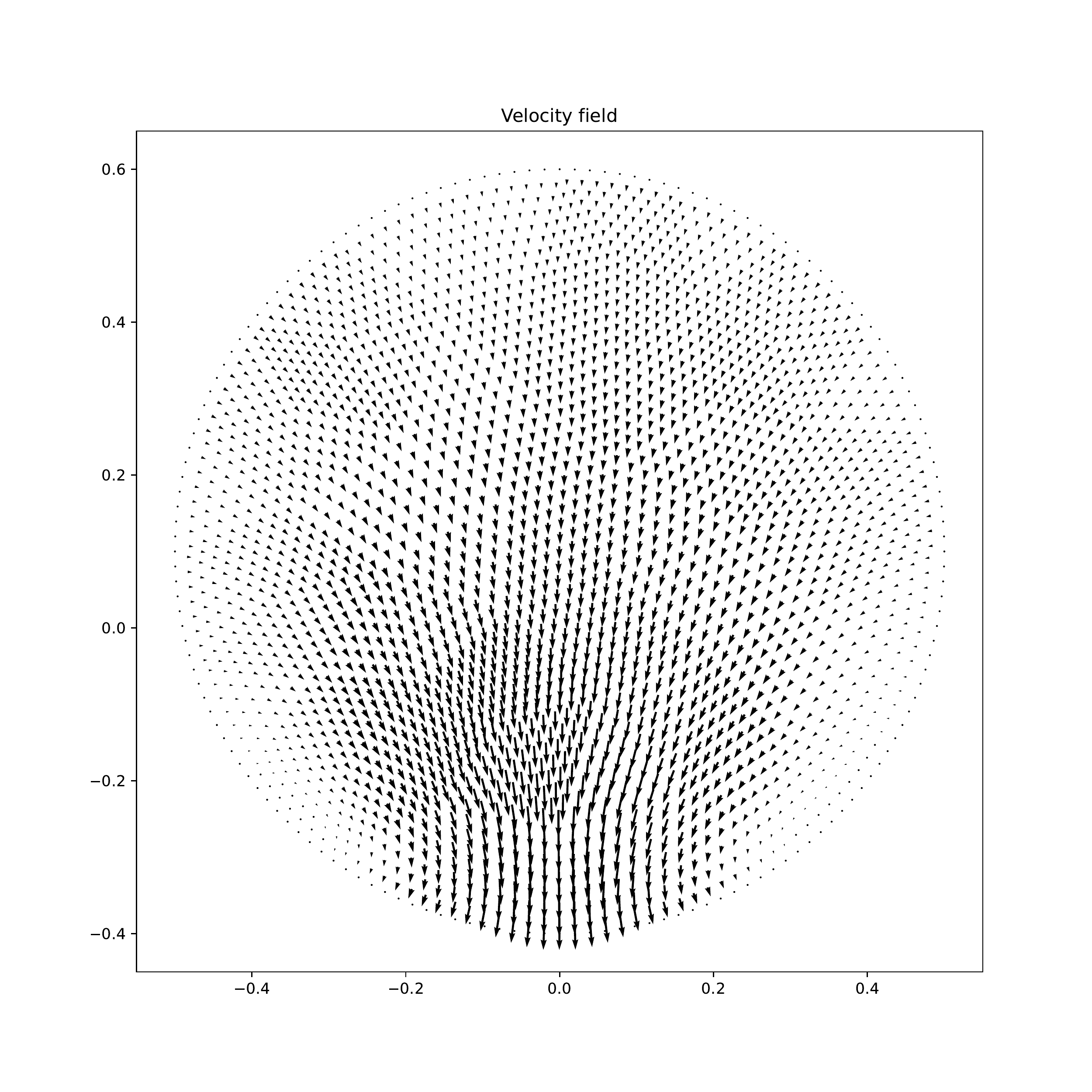}
    \end{subfigure}
    \begin{subfigure}[b]{0.23\textwidth}
        \centering
        \includegraphics[width=1.0\textwidth]{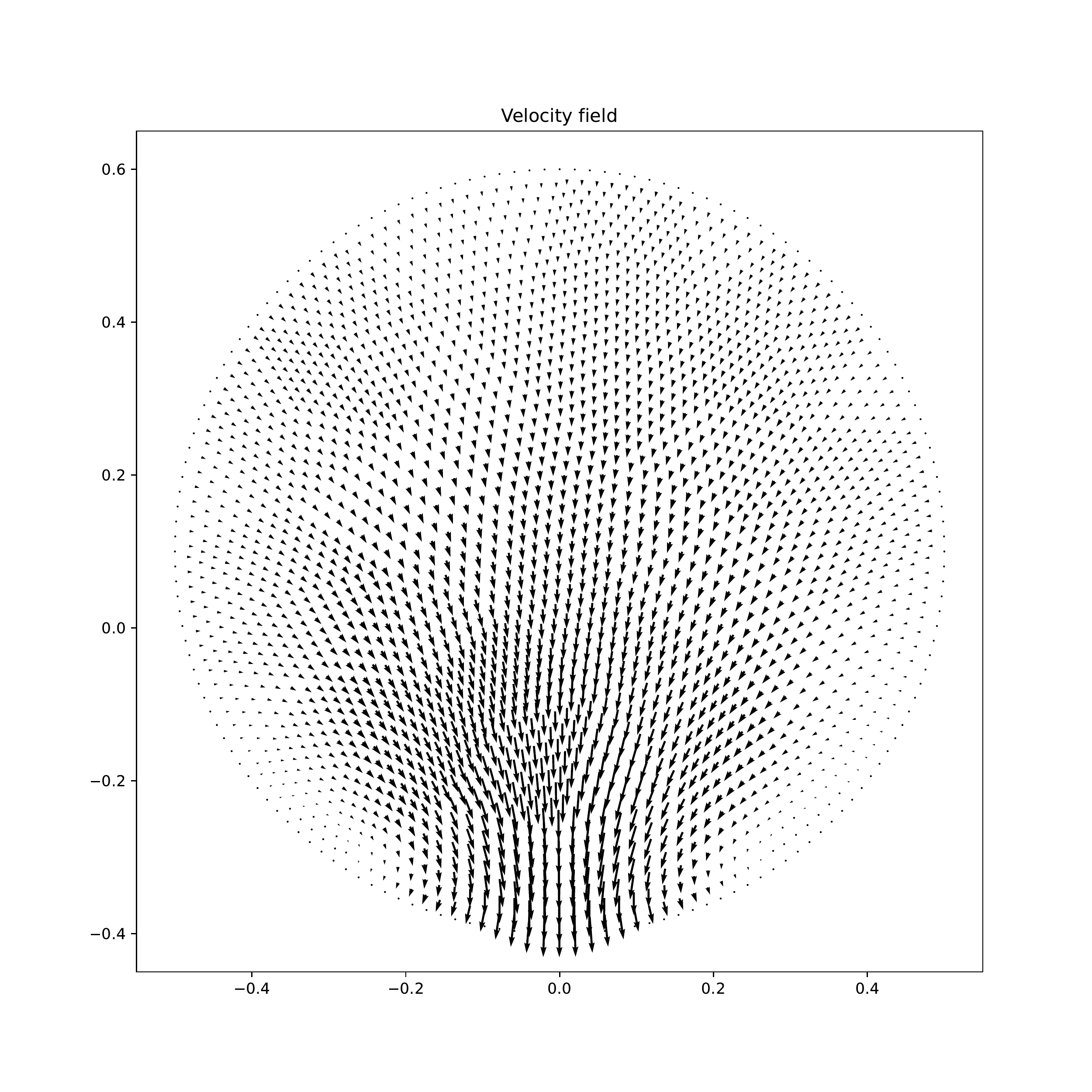}
    \end{subfigure}
    \caption{Snapshots at $t=0.02$, $t=0.04$, $0.2$ and $1.0$  of $\{\u_h^m\}_m$ for $\eta_0=450$ and $\Phi_0=50$.}\label{Snapshots_450_50_uh}
\end{figure}
\begin{figure}
    \begin{subfigure}[b]{0.23\textwidth}
        \centering
        \includegraphics[width=1.0\textwidth]{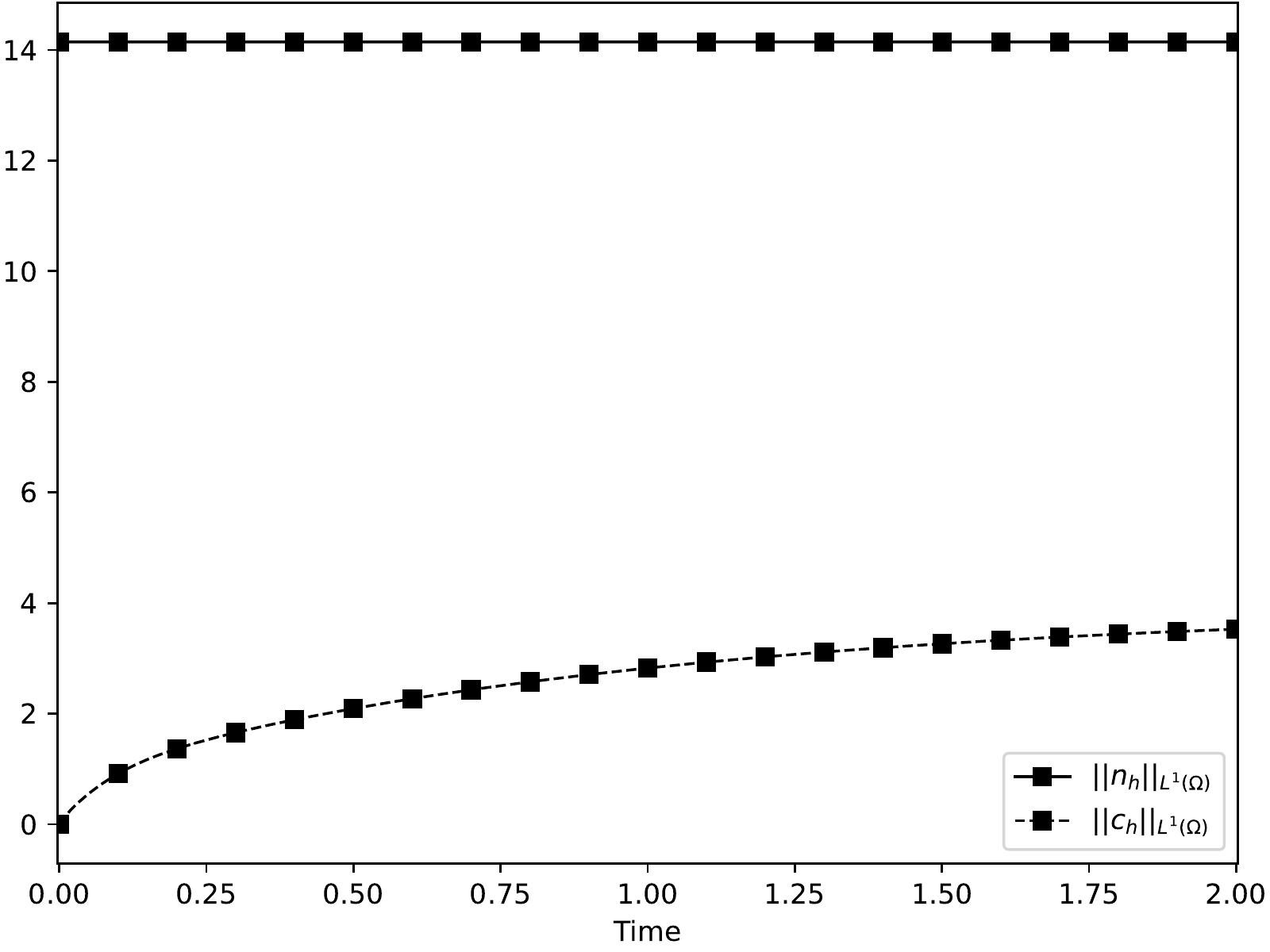}
    \end{subfigure}
    \begin{subfigure}[b]{0.23\textwidth}
        \centering
        \includegraphics[width=1.0\textwidth]{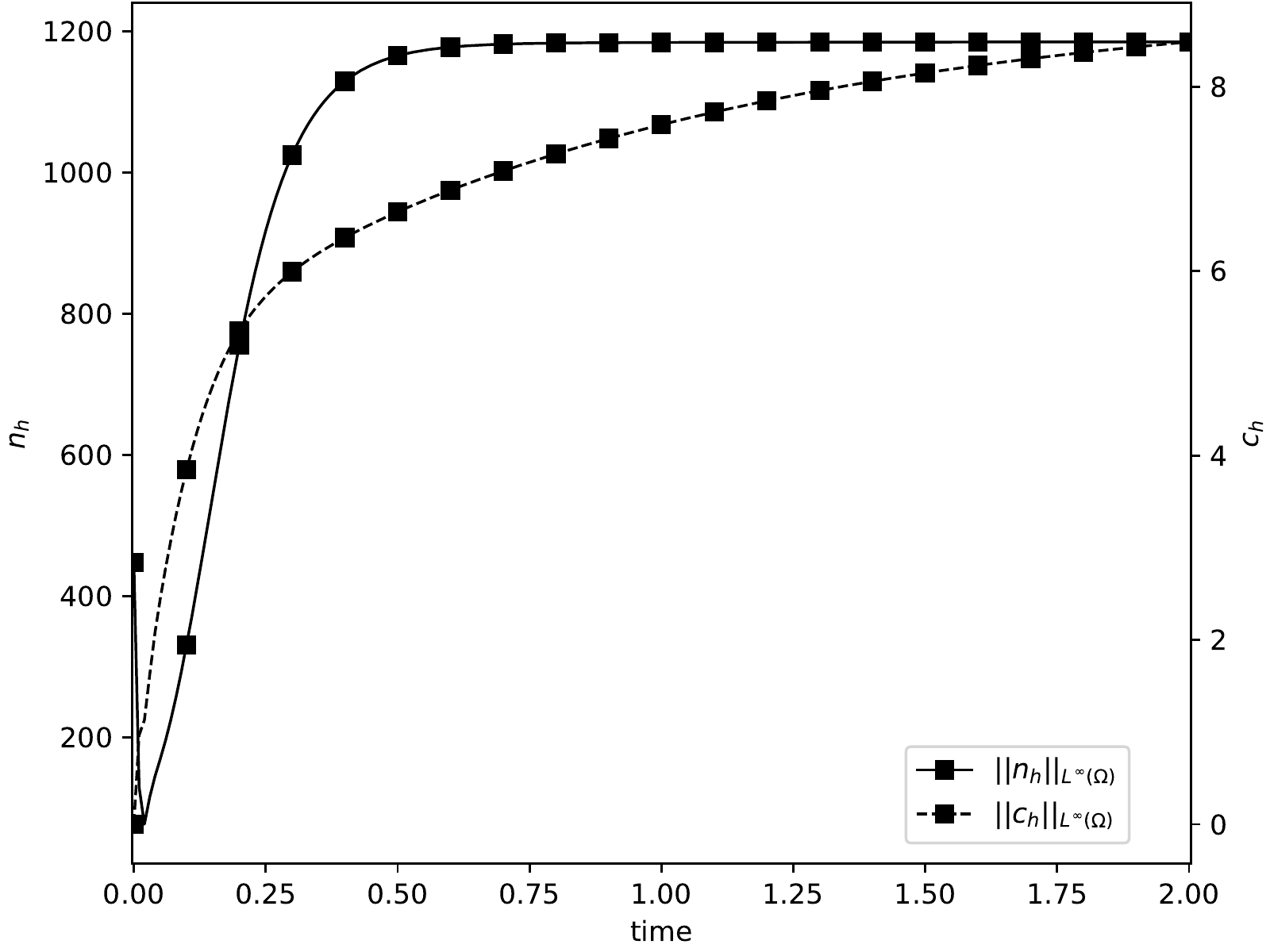}
    \end{subfigure}
    \begin{subfigure}[b]{0.23\textwidth}
        \centering
        \includegraphics[width=1.0\textwidth]{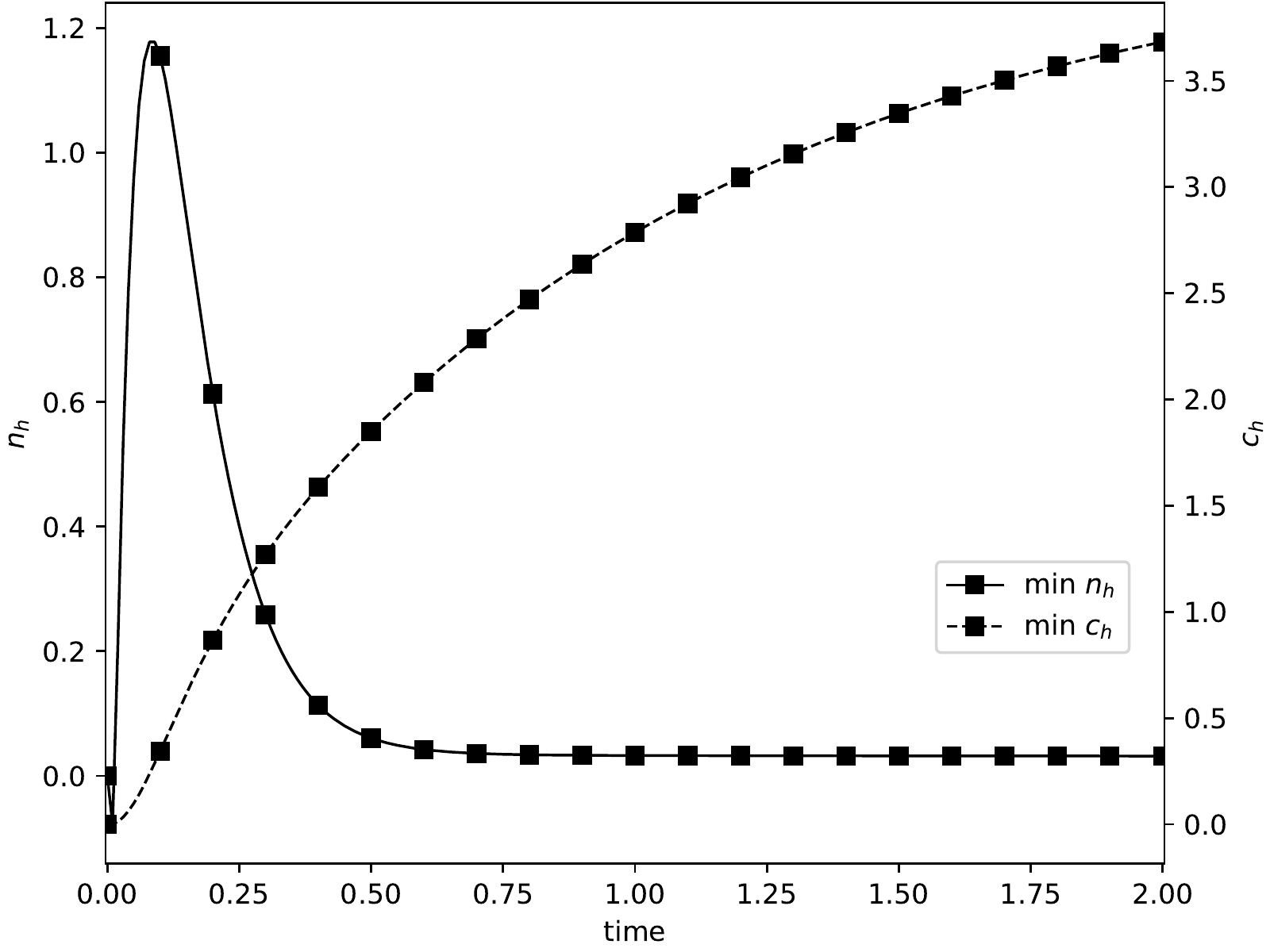}
    \end{subfigure}
            \caption{Plots of total mass, maxima and minima of $\{n_h^m\}_m$ and $\{c_h^m\}_m$ for $\eta=450$ and $\Phi_0=50$.}\label{Graphs_450_50}
\end{figure}

 \section{Conclusion}

In this paper a numerical method for approximating solutions to the Keller--Segel--Navier--Stokes system has been constructed. It consists of a finite element method together with a stabilising term, whose design is based on a shock capturing technique so as to preserve lower bounds such as positivity and non-negativity.

It is known that solutions to the Keller--Segel--Navier--Stokes system are uniformly bounded in time providing that the total mass for the organism density is below $2\pi$. Such a threshold is smaller than that for the Keller--Segel system, which corresponds to $4\pi$. Then we have made an attempt at answering the question whether or not the value $2\pi$ is critical through a set of numerical experiments.
The evidence found herein puts into new perspective the threshold value for proving boundedness solutions for the Keller--Segel--Navier--Stokes equations. We have discovered that the value $4\pi$ instead of $2\pi$ may be critical; thus inheriting it from the Keller--Segel subsystem. Furthermore, we have observed that the fluid intensification may lead to the depletion of chemotaxis and prevent possible singularity formation. Realising this possibility may be a watershed in the knowledge of the phenomenological interaction of chemotaxis in fluid scenarios. 

The above findings are relied on the fact that numerical solutions computed by the proposed algorithm satisfy lower and $L^1(\Omega)$ bounds and quasi-energy estimates. This latter property results from a new discretization of the chemotactic and convective terms and a Morse-Trudinger's inequality demonstrated for polynomial domains. All in all, we have found that  our numerical solutions are robust and  reliable in the numerical simulations.

An improvement of the numerical method that may be regarded is using stabilising techniques for the convective terms at least in the Keller--Segel subsystem, since when taking $\Phi_0$ very large numerical solutions do not fulfil lower bounds.

\end{document}